\documentclass[reqno,letterpaper,11pt]{article}
\pdfoutput=1

\usepackage[title]{appendix}

\usepackage{hyperref}

\usepackage{comment}

\usepackage{palatino}

\usepackage[autostyle]{csquotes}

\usepackage{amsmath}
\usepackage{amsfonts,amssymb,amsmath,latexsym,mathrsfs,bbm,stmaryrd}
\usepackage[all]{xy}
\usepackage[usenames]{color}
\usepackage{epsfig}

\usepackage{graphicx}
\usepackage{subcaption}
\usepackage{float}

\usepackage{amsthm}
\usepackage{thmtools}
\usepackage{thm-restate}
\declaretheorem[numberwithin=section]{theorem}
\declaretheorem[sibling=theorem]{proposition}
\declaretheorem[sibling=theorem]{definition}
\declaretheorem[sibling=theorem]{corollary}
\declaretheorem[sibling=theorem]{lemma}

\declaretheorem[sibling=theorem,style=remark]{remark}

\numberwithin{equation}{section}

\usepackage[mathcal]{euscript}
\usepackage{times}

\def\R{\mathbb R}
\def\C{\mathbb C}

\def\E{\mathbb E}

\def\1{\mathbbm{1}}

\def\bbD{\mathbb D}

\def\A{\mathscr{A}}

\def\adots{
	\mathinner{\mkern1mu\raise1pt\hbox{.}\mkern2mu\raise4pt\hbox{.}
		\mkern2mu\raise7pt\vbox{\kern7pt\hbox{.}}\mkern1mu}}

\newcommand{\Tr}{\mathrm{Tr}\,}

\long\def\symbolfootnote[#1]#2{\begingroup%
\def\thefootnote{\fnsymbol{footnote}}\footnote[#1]{#2}\endgroup}

\newcommand{\Lt}{\Lambda_t}

\newcommand{\LtC}{\overline{\Lambda_t}}

\newcommand{\St}{\Sigma_{t,\bar{\mu}}}

\newcommand{\Pt}{\Phi_{t,\bar{\mu}}}

\newcommand{\Ot}{\Omega_{t,\bar{\mu}}}

\newcommand{\Dt}{\Delta_{t,\mu}}

\newcommand{\bT}{\mathbb{T}}

\providecommand{\keywords}[1]{\textbf{\textit{Keywords:}} #1}

\begin{document}

\title{Brown measures of free circular and multiplicative Brownian motions with self-adjoint and unitary initial conditions}
\author{
	Ching-Wei Ho and Ping Zhong 
}

\date{}
\maketitle

\begin{abstract}
	Let $Z_N$ be a Ginibre ensemble and let $A_N$ be a Hermitian random matrix independent from $Z_N$ such that $A_N$ converges in distribution to a self-adjoint random variable $x_0$ in a $W^*$-probability space $(\A,\tau)$. For each $t>0$, the random matrix $A_N+\sqrt{t}Z_N$ converges in $\ast$-distribution to $x_0+c_t$, where $c_t$ is the circular variable of variance $t$, freely independent from $x_0$.
	We use the Hamilton--Jacobi method to compute the Brown measure $\rho_t$ of $x_0+c_t$. The Brown measure has a density that is \emph{constant along the vertical direction} inside the support. The support of the Brown measure of $x_0+c_t$ is related to the subordination function of the free additive convolution of $x_0+s_t$, where $s_t$ is the semicircular variable of variance $t$, freely independent from $x_0$. Furthermore, the push-forward of $\rho_t$ by a natural map is the law of $x_0+s_t$. 
	
	Let $G_N(t)$ be the Brownian motion on the general linear group and let $U_N$ be a unitary random matrix independent from $G_N(t)$ such that $U_N$ converges in distribution to a unitary random variable $u$ in $(\A,\tau)$. The random matrix $U_NG_N(t)$ converges in $\ast$-distribution to $ub_t$ where $b_t$ is the free multiplicative Brownian motion, freely independent from $u$. 
	 We compute the Brown measure $\mu_t$ of $ub_t$, extending the recent work by Driver--Hall--Kemp, which corresponds to the case $u=I$. The measure has a density of the special form 
	\[\frac{1}{r^2}w_t(\theta)\]
	 in polar coordinates in its support. The support of $\mu_t$ is related to the subordination function of the free multiplicative convolution of $uu_t$ where $u_t$ is the free unitary Brownian motion, freely independent from $u$. The push-forward of $\mu_t$ by a natural map is the law of $uu_t$.
	
	In the special case that $u$ is Haar unitary, the Brown measure $\mu_t$ follows the \emph{annulus law}. The support of the Brown measure of $ub_t$ is an annulus with inner radius $e^{-t/2}$ and outer radius $e^{t/2}$. In its support, the density in polar coordinates is given by 
	\[\frac{1}{2\pi t}\frac{1}{r^2}.\]

\end{abstract}


\tableofcontents

\section{Introduction}

It is a classical theorem by Wigner \cite{Wigner1955} that the eigenvalue distribution of a Gaussian unitary ensemble (GUE) $G_N$ converges to the semicircle law. An operator $a\in\A$, where $\A$ is a tracial von Neumann algebra, is said to be a limit in $\ast$-distribution of a sequence of $N\times N$ self-adjoint random matrices $A_N$ if, for any polynomial in two noncommuting variables,
\begin{equation}
	\label{eq:limitdistribution}
\lim_{N\to\infty}\frac{1}{N}\E\Tr[p(A_N, A_N^*)] = \tau[p(a, a^*)].
\end{equation}
In other words, the GUE has the limit in $\ast$-distribution as an operator having the semicircle law as its spectral distribution. The operators in $\A$ are called random variables.

Voiculescu \cite{Voiculescu1991} discovered that free probability can be used to study the large-$N$ limit of eigenvalue distributions of $X_N+G_N$, where $X_N$ is a sequence of self-adjoint random matrix independent from $G_N$, or a sequence of deterministic matrix that has a limit in distribution.

Biane proved that the limit in $\ast$-distribution of the unitary Brownian motion $U_N(t)$ on the unitary group $U(N)$ is the free unitary Brownian motion in a tracial von Neumann algebra \cite{Biane1997b}.
If $V_N$ is a sequence of unitary random matrices independent from $U_N$ or a sequence of deterministic unitary matrices that has a limit in $\ast$-distribution, then free probability can also be used to study the limit of the eigenvalue distribution of $V_NU_N(t)$.

Given a self-adjoint random variable $a\in \A$,
the spectral distribution, or the law, of $a$ is a probability measure on $\R$ defined to be the trace of the projection-valued spectral measure, whose existence is guaranteed by the spectral theorem. The law of $a$ can be identified and computed by the Cauchy transform 
\begin{equation}
\label{CauchyTrans}
G_a(z) = \tau((z-a)^{-1}), \qquad z\in \mathbb{C}^+.
\end{equation}
The spectral distribution of a unitary operator in $\A$ is a probability measure on the unit circle $\mathbb{T}$. 
When we consider non-normal random variables, the spectral theorem is no longer valid. Instead, we look at the Brown measure \cite{Brown1986}, which has been called the spectral distribution measure of a not-necessarily-normal random variable. The Brown measure of the free random variables provides a natural candidate of the limit of the eigenvalue distribution of the non-normal random matrices.

In this article, we calculate the density formulas for the Brown measure of the free circular Brownian motion with self-adjoint initial condition, as well as the free multiplicative Brownian motion with unitary initial condition. The latter extends the recent work of Driver, Hall and Kemp \cite{DHKBrown} for the case of free multiplicative Brownian motion starting at the identity operator. (See also the expository paper \cite{Hall2019pde}, which provides a gentle introduction to the PDE methods used in this paper.) Our result indicates that the Brown measure of the free \emph{circular} Brownian motion with self-adjoint initial condition is closely related to the free \emph{semicircular} Brownian motion with the same self-adjoint initial condition. Similarly, the free \emph{multiplicative} Brownian motion with unitary initial condition is closely related to the free \emph{unitary} Brownian motion with the same unitary initial condition. 

After the first version of this paper appeared on arXiv, there are subsequent work using a similar strategy to compute the Brown measure of free Brownian motions with nontrivial initial conditions. Demni and Hamdi \cite{DemniHamdi2020} analyze the free unitary Brownian motion with projection initial condition; Hall and the first author \cite{HallHo2020} compute the Brown measure of the imaginary multiple of free semicircular Brownian motion with bounded self-adjoint initial condition.

\subsection{Additive Case}
\label{AddIntro}
One of the fundamental non-normal random matrix models is the \emph{Ginibre ensemble} $Z_N$, which is a sequence of $N\times N$ random matrices with i.i.d. complex Gaussian entries, with variance $1/N$. The limiting empirical eigenvalue distribution, which is a normalized counting measure $\frac{1}{N}\sum_{j=1}^N \delta_{\lambda_j}$ of the eigenvalues $\{\lambda_j\}$ of $Z_N$, converges to the uniform probability measure on the unit disk \cite{Ginibre1965}. The reader is referred to \cite{Bordenave-Chafai-circular} for a survey on the circular law. The limit random variable, in the sense of $\ast$-distribution, is called the circular operator~\cite{DVV-circular}. If we consider the \emph{process} $Z_N(t)$ of random matrices with i.i.d. entries of complex Brownian motion at time $t/N$, the limiting empirical eigenvalue distribution at time $t$ is the uniform probability measure on the disk of radius $\sqrt{t}$. The limit in $\ast$-distribution of $Z_N(t)$ is a ``free stochastic process," a one-parameter family of random variables in $\A$,  which is called the free circular Brownian motion $c_t$. At each $t>0$, $Z_N(t)$ has the same distribution as $\sqrt{t}Z_N$, and $c_t$ has the same distribution as $\sqrt{t} c_1$.

The standard free circular Brownian motion starts with the condition $c_0=0$. We consider a more general free stochastic process: a free stochastic process that starts at an arbitrary random variable $x_0\in\A$ and has the same increments as the standard free circular Brownian motion. Such a process has the form $x_0+c_t$ where $x_0$ and $c_t$ are freely independent.

The circular operator $c_t$ is an $R$-diagonal operator \cite[Lecture 15]{SpeicherNicaBook}. Given a variable $x_0\in \A$ which may not be normal and $r\in\A$ which is $R$-diagonal and freely independent from $x_0$, Biane and Lehner \cite[Section 3]{BianeLehner2001} studied the Brown measure of $x_0+r$. They obtained an explicit density formula for particular $x_0$ and $r$. In Section 5 of the same paper, they studied, more specifically, the Brown measure of $x_0+c_t$ where $x_0\in\A$ may not be normal; the density at $\lambda\not\in \sigma(x_0)$ is given by
$$\frac{1}{\pi}\partial_{\bar{\lambda}}\left(\int_{t_\lambda}^t \frac{\partial_\lambda v_s(\lambda)}{s^2}ds-\frac{v_{t_\lambda}( \lambda)^2}{t_\lambda^2}\partial_\lambda t_\lambda\right)$$
where \[v_s(\lambda) = \inf\left\{v\geq 0: \int_\R \frac{d\mu_{|\lambda-x_0|}(x)}{x^2+v^2}\leq \frac{1}{s}\right\}\]
and $t_\lambda = \inf\{t: v_t(\lambda)>0\}$. It is  mentioned in their paper that $v_s(\lambda)$ and $t_\lambda$ are  related to the subordination function of $\tilde{x}_0+s_t$ with respect to $\tilde{x}_0$, where $\tilde{x}_0$ is the symmetrization of $|\lambda-x_0|$, and $s_t$ is the free semicircular Brownian motion, the real part of $\sqrt{2}\,c_t$. The quantities are not very explicit; there is an integration in the time variable. Furthermore, the formula is only valid outside the spectrum of $x_0$.

Our method shows that density formula for the special case when $x_0$ is self-adjoint can be computed more explicitly. Our results also illustrate unexpected connections between the Brown measure of $x_0+c_t$ and the subordination function for the free convolution of $x_0$ with $s_t$. 
 Suppose $x_0$ is self-adjoint in the rest of the paper. 
 Denote by $\mu$ the spectral distribution of $x_0$. In this case, the following function $v_t$ defined on $\R$ is fundamental in our analysis
 \begin{equation}
 \label{intro:vt}
 v_t(a) =  \inf\left\{b\geq 0: \int_\R \frac{d\mu(x)}{(x-a)^2+b^2}\leq \frac{1}{t}\right\}.
 \end{equation}
This definition of $v_t$ also coincides with the definition of Biane and Lehner's; the measure $\mu_{|\lambda-x_0|}$ is the push-forward of $\mu_{x_0}$ by the map $x\mapsto |\lambda-x|$, for each $\lambda$, because $x_0$ is self-adjoint.

Our main result in Section \ref{CircularCase} gives an explicit description to the Brown measure of $x_0+c_t$, using the Hamilton--Jacobi method used in the recent paper \cite{DHKBrown}.  In this paper, the Brown measure of $x_0+c_t$ is computed as an absolutely continuous measure on $\R^2$. The density is computed explicitly; the density is not given as an integral over the time parameter.

The operator $x_0+c_t$ is the limit in $\ast$-distribution (in the sense of~\eqref{eq:limitdistribution}) of the random matrix model $X_N+Z_N(t)$, by \cite[Theorem 2.2]{Voiculescu1991}, where $X_N$ is an $N\times N$ self-adjoint deterministic matrix or random matrix classically independent of $Z_N$, with $x_0$ in the limit of $X_N$ in $\ast$-distribution. On the level of Brown measure convergence, \'Sniady \cite{Sniady2002} showed that the empirical eigenvalue distribution of $X_N+Z_N(t)$ almost surely converges weakly to the Brown measure of $x_0+c_t$, without computing the Brown measure of $x_0+c_t$ explicitly. (Apply \cite[Theorem 6]{Sniady2002} by taking the not-necessarily-normal random matrix $A^{(N)}$ as the $X_N$ here.) Consequently, the result in this paper gives a formula for the weak limit of the empirical eigenvalue distribution of $X_N+Z_N(t)$.

In physics literature, Burda et al. \cite{BGNTW2014, BGNTW2015} studied the limiting eigenvalue distribution of the sum of the Ginibre ensemble $Z_N$ and a deterministic matrix $A_N$ (which is not assumed to be normal) using PDE methods. The formal large-$N$ limit of the PDE obtained in their work \cite[Eq. (31)]{BGNTW2015} is the same as the PDE we obtain for the additive case. They transformed the PDE into one that can be solved using the method of characteristics, whereas we use the Hamilton--Jacobi method to solve the PDE in this paper. In \cite[Section 6.2]{BGNTW2015}, they computed explicitly the limiting eigenvalue distribution of $A_N+Z_N$ when the eigenvalue distribution of $A_N$ is the Bernoulli distribution. It can be checked that the domain where the limiting eigenvalue distribution of $A_N+Z_N$ is nonzero agrees with our result. The density of the limiting eigenvalue distribution computed in their paper also agrees with the Brown measure density computed in our paper, after correcting minor algebraic errors.
We note, however, that, \cite{BGNTW2014, BGNTW2015} did not compute the limiting eigenvalue distribution of $A_N+Z_N$ for general self-adjoint matrix $A_N$, which is the main purpose of Section~\ref{CircularCase} of our paper.

Let $G_{x_0}(z)$ be the Cauchy transform of $x_0$ as defined in \eqref{CauchyTrans} and let $$H_t(z) = z+tG_{x_0}(z), \qquad z\in \mathbb{C}\setminus \sigma
(x_0).$$
It is known \cite{Biane1997} that $H_t$ maps the curves $\{(a, \pm v_t(a)): a\in \R\}$, which are analytic at the points where $v_t(a) > 0$, to $\R$. The restriction of $H_t$ to the set $\{a+ib \in \C^+: b>v_t(a)\}$ is the inverse of an analytic self-map on the upper half plane $F_t:\C^+\to\C^+$ such that
$$G_{x_0+s_t}(z) = G_{x_0}(F_t(z))$$ where $s_t$ is a semicircular variable with variance $t$ \cite{Biane1997}. The map $F_t$ is the \emph{subordination function with respect to $x_0$} and allows one to compute $G_{x_0+s_t}$ from $G_{x_0}$. The following theorem summarizes the results proved in Theorems \ref{ThAdd}, \ref{AddBound} and Proposition \ref{addunique}.

\begin{theorem}
	\label{IntroTh1}
	The Brown measure $\rho_t$ of $x_0+c_t$ has the following properties.
\begin{enumerate}
	\item The support of $\rho_t$ is the closure of an open set. More precisely, $\textrm{supp}\:\rho_t = \LtC$ where
	\[\Lt = \{a+ib\in \C: |b|<v_t(a)\},\] 
	where $v_t$ is defined in \eqref{intro:vt}.
	\item The measure $\rho_t$ is absolutely continuous and its density is constant in the vertical direction inside the support. That is the density $w_t$ has the property $w_t(a+ib) = w_t(a)$ in the support of $\rho_t$. Moreover, we have
	\[w_t(a)\leq \frac{1}{\pi t},\]
	and the inequality is strict unless $x_0$ is a scalar.
	
	\item Define $\Psi_t(a+ib) = H_t(a+iv_t(a))$ on $\textrm{supp}\, \rho_t$  agreeing with $H_t$ on the boundary of $\textrm{supp}\,\rho_t$ and constant along the vertical segments (the image of $\Psi$ only depends on $a$). Then the push-forward of $\rho_t$ by $\Psi$ is the law of $x_0+s_t$.
	\item The preceding properties uniquely define $\rho_t$. That is, $\rho_t$ is the unique measure whose support agrees with $\textrm{supp}\,\rho_t$, density is constant along the vertical segments and the push-forward under $\Psi$ is the distribution of $x_0+s_t$. This is because $\Psi$ restricted to $\R$ is a homeomorphism (\cite{Biane1997}) and the density is constant along the vertical direction.
	\end{enumerate}
\end{theorem}
The support of the Brown measure is symmetric about the real line. Figure \ref{IntroFigAdd} shows a random matrix simulation to the Brown measure of $x_0+c_t$ where $x_0$ has law $\frac{1}{4}\delta_{-0.8}+\frac{3}{4}\delta_{0.8}$ and $t=1$.

\begin{figure}[h]
	\begin{center}
		\includegraphics[width=0.5\linewidth]{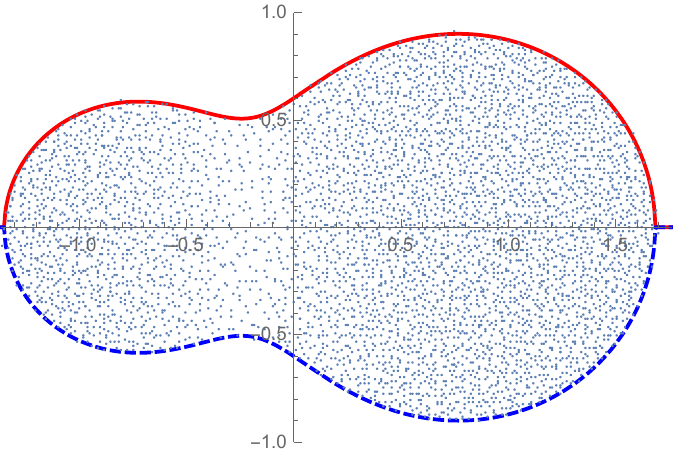}
		\caption{\label{IntroFigAdd} Matrix simulation of eigenvalues for $x_0+c_1$, where $x_0$ has distribution $\frac{1}{4}\delta_{-0.8}+\frac{3}{4}\delta_{0.8}$. \\
			The graphs of $v_t(a)$ (blue dashed) and $-v_t(a)$ (red) are superimposed.}
	\end{center}
\end{figure}

\subsection{Multiplicative Case}
The process $Z_N(t)$ can be viewed as a Brownian motion on the Lie algebra $gl(N, \C)$ of the general linear group $GL(N, \C)$, under the Hilbert-Schmidt inner product such that the real and imaginary parts are orthogonal.
The Brownian motion on $GL(N, \C)$ is defined to be the solution of the matrix-valued stochastic differential equation
$$d G_N(t) = G_N(t) \,dZ_N(t), \quad G_N(0)=I_N.$$
Kemp \cite{Kemp2016} proved that the limit of $G_N(t)$ in $\ast$-distribution is the free stochastic process $b_t$ that can be obtained by solving the free stochastic differential equation (see, for example, \cite{BianeSpeicher1998, KummererSpeicher1992})
\[
  db_t=b_t\,dc_t, \quad b_0=I.
\]
The process $b_t$, starting at the identity $I$, is called the \emph{free multiplicative Brownian motion}. Hall and Kemp \cite{HallKemp2019} showed that the Brown measure of $b_t$ is supported in certain compact set in $\C$. Later the Brown measure of $b_t$ was computed by Driver, Hall and Kemp \cite{DHKBrown}.

In this paper, we consider the free multiplicative Brownian motion starting at a unitary random variable $u$; such a process has the form $ub_t$, where $u$ is freely independent of $b_t$. 
Before we state the results of the Brown measure of $ub_t$, we need to introduce the free unitary Brownian motion $u_t$, the multiplicative analogue of the semicircular Brownian motion. The free unitary Brownian motion is the solution of the free stochastic differential equation
$$du_t = iu_t\,ds_t - \frac{1}{2}u_t\,dt, \quad u_0=I. $$
The law $\nu_t$ of $u_t$ was computed by Biane \cite{Biane1997b}. The law of free unitary Brownian motion with an arbitrary unitary initial condition was computed by the second author \cite{Zhong2015}.

Suppose that $a\in\A$ is a unitary random variable with spectral distribution, or the law, $\mu$. Then $\mu$ can be identified by the moment generating function
\[\psi_{a}(z) = \int_{\mathbb{T}}\frac{\xi z}{1-\xi z}\,d\mu(\xi), \qquad z\in\bbD.\]
Now, we consider the random variable $ub_t$ where $u$ is any unitary random variable freely independent from $b_t$. Let $\bar{\mu}$ be the spectral distribution of $u^*$ and set
\[\Phi_t(z) = z\exp\left(\frac{t}{2}\int_{\mathbb{T}}\frac{1+\xi z}{1-\xi z}d\bar{\mu}(\xi)\right), \qquad z\in \mathbb{C}\setminus \sigma
(u).\]
Consider the function 
\begin{equation}
\label{AddvtForm}
r_t(\theta)=\sup\left\{0<r<1: \frac{r^2-1}{2\log r}\int_{-\pi}^\pi \frac{1}{|1-re^{i(\theta+x)}|^2}\,d\bar{\mu}(e^{ix})<\frac{1}{t}\right\}, \theta\in(-\pi, \pi].
\end{equation}
The set over which the supremum is taken is nonempty, because $\lim_{r\to 0}\frac{r^2-1}{2\log r}\int_{-\pi}^\pi \frac{1}{|1-re^{i(\theta+x)}|^2}\,d\bar{\mu}(e^{ix})= 0$.
It is known \cite{Zhong2015} that the map $\Phi_t$ maps the curve $\{r_t(\theta)e^{i\theta}: \theta\in(-\pi, \pi]\}$, which is analytic whenever $r_t(\theta)<1$, to the unit circle. The restriction of $\Phi_t$ to the set $\{z\in\bbD: |z|<r_t(\theta)\}$ is the inverse of the analytic self-map $\omega_t:\bbD\to\bbD$ such that
$$\psi_{u^*u_t}(z) = \psi_{u^*}(\omega_t(z)).$$
Analogous to the additive case, the map $\omega_t$ is the \emph{subordination function with respect to $u^*$} and is of significance to the study in free probability. The regularity results of $r_t$ are summarized in Proposition~\ref{characterization-0}.

 We establish the following results in Theorem \ref{sLaplacian.main} and Proposition \ref{multunique}.
\begin{theorem}
	The Brown measure $\mu_t$ of $ub_t$ satisfies the following properties.
\begin{enumerate}
	\item The support of $\mu_t$ is the closure of an open set. More precisely, $\textrm{supp}\:\mu_t = \overline{\Delta_t}$, where
	\[\Delta_t = \left\{re^{i\theta}: r_t(\theta)<r<\frac{1}{r_t(\theta)}\right\}\]
	where $r_t$ is defined in \eqref{AddvtForm}.
	\item The measure $\mu_t$ is absolutely continuous and, inside the support, its density in polar coordinates has the form
	$$W_t(r,\theta)=\frac{1}{r^2}w_t(\theta),$$
	for some function $w_t$ depending on the argument only. That is, the density of $\mu_t$ is inversely proportional to $r^2$ along the radial direction, and the proportionality constant depends only on the argument $\theta$. The Brown measure $\mu_t$ is invariant under $z\mapsto 1/\bar{z}$. Moreover, we have
	\[w_t(\theta) < \frac{1}{\pi t}.\]
	\item For each point $\lambda\in \textrm{supp }\mu_t$, there exists a unique point $\lambda_\theta\in  \bar\bbD\cap \partial(\textrm{supp }\mu_t)$ such that $\lambda_\theta$ and $\lambda$ have the same argument (mod $2\pi$). 
	Define $\Gamma_t(\lambda) = \Phi_t(\lambda_\theta)$ agreeing with $\Phi_t$ on $\bar\bbD\cap \partial(\textrm{supp }\mu_t)$, constant along the radial segments. Then the push-forward of $\mu_t$ by $\Gamma_t$ is the distribution of $uu_t$. 
	\item The preceding properties uniquely define $\mu_t$. That is, $\mu_t$ is the unique measure whose support agrees with $\textrm{supp}\,\mu_t$, density has the form of $(1/r^2)\phi(\theta)$ in polar coordinates and the push-forward under $\Gamma_t$ is the distribution of $uu_t$. This is because $\Gamma_t$ restricted to $\bar\bbD\cap \partial(\textrm{supp }\mu_t)$ is a homeomorphism (\cite{Zhong2015}) and the density of $\mu_t$ is inversely proportional to $r^2$ along the radial direction.
\end{enumerate}
\end{theorem}

Denote the Haar unitary random variable by $h$; which means $h$ is a unitary operator, whose spectral distribution is the Haar measure on the unit circle. In the special case when $u=h$, the Brown measure of $hb_t$ is the \emph{annulus law}. It is absolutely continuous, supported in the annulus $A_t = \{e^{-t/2}\leq |z|\leq e^{t/2}\}$, rotationally invariant, and the density $W_t(r,\theta)$ is given by
\[W_t(r,\theta) = \frac{1}{2\pi t}\frac{1}{r^2}\]
on the annulus $A_t$. The density $W_t$ is independent of $\theta$ because the Brown measure is rotationally invariant.
We can also calculate the Brown measure of $hb_t$ using Haagerup-Larsen's formula for the Brown measures of $R$-diagonal operators \cite{HaagerupLarsen2000}. Indeed,  
the random variable $hb_t$ is $R$-diagonal \cite[Proposition 15.8]{SpeicherNicaBook} and the Brown measure of $hb_t$ can be calculated from the distribution of $(b_t^*b_t)^{1/2}$ (see Appendix for details).  

 Figure \ref{IntroFigMult} shows a random matrix simulation of the eigenvalue distribution of $ub_{0.8}$ where $u$ has distribution $\frac{1}{3}\delta_{e^{i\pi/3}}+\frac{2}{3}\delta_{e^{4\pi i/5}}$. It should be noted that, in this multiplicative case, it is an open problem to give a mathematical proof of that even when the initial condition is the identity, the empirical eigenvalue distribution of $G_N(t)$ converges to the Brown measure of $b_t$.

\begin{figure}[h]
	\begin{center}
		\includegraphics[width=0.55\linewidth]{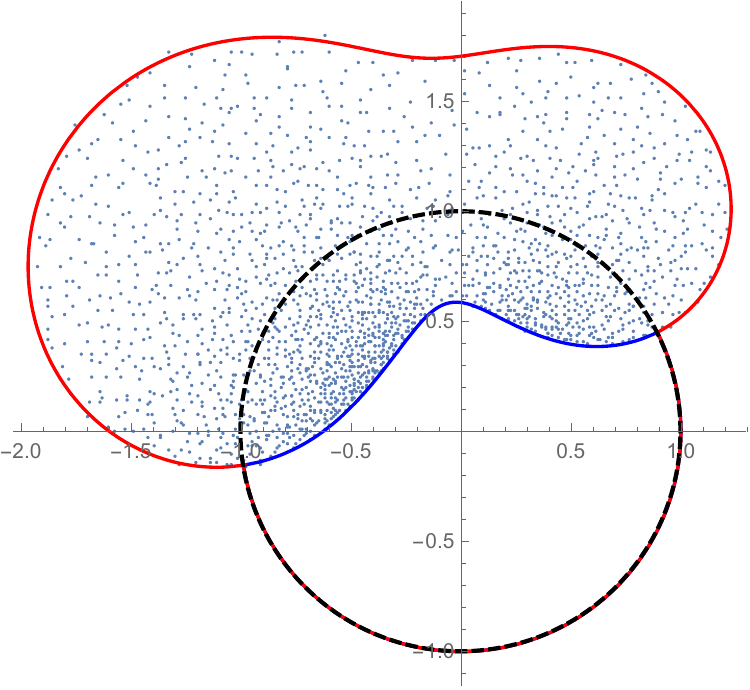}
		\caption{\label{IntroFigMult}Matrix simulations of eigenvalues for $ub_{0.8}$, $u$ has distribution $\frac{1}{3}\delta_{e^{i\pi/3}}+\frac{2}{3}\delta_{e^{4\pi i/5}}$. The curves $r_t(\theta)e^{i\theta}$ (blue), $\frac{1}{r_t(\theta)}e^{i\theta}$ (red), and the unit circle (black dashed) are superimposed.}
	\end{center}
\end{figure}

The paper is organized as follows. Section \ref{Prelim} consists of some background and preliminaries of free probability theory and the definition of the Brown measure. The distributions of the sum of two self-adjoint free random variables and the product of two unitary free random variables will be described using the subordination functions. In Section \ref{CircularCase}, we compute the Brown measure of the random variable $x_0+c_t$. The Brown measure is closely related to the subordination function and the distribution of $x_0+s_t$. 
Section \ref{MultCase} is concerned about the Brown measure of the random variable $ub_t$, for both cases when $u$ is and is not a Haar unitary, using the same Hamilton--Jacobi analysis but different initial conditions from \cite{DHKBrown}. The support and the density of the Brown measure are again related to a subordination function; but the subordination function is the one for $u^*u_t$ rather than $uu_t$. The Brown measure is connected to the distribution of $uu_t$ (but not $u^*u_t$).

\section{Preliminaries}
\label{Prelim}
\subsection{Free Probability}
\label{FreeProb}
  A {\bf $W^*$-probability space} is a pair $(\A, \tau)$ where $\A$ is a finite von Neumann algebra and $\tau$ is a normal, faithful tracial state on $\A$. The elements in $\A$ are called (noncommuntative) random variables.

The unital $\ast$ - subalgebras $\A_1, \cdots \A_n\subseteq \A$ are said to be free or freely independent in the sense of Voiculescu if, given any $i_1, i_2,\cdots, i_m\in\{1,\cdots,n\}$ with $i_k\not= i_{k+1}$, and $a_{i_j}\in\A_{i_j}$ satisfying $\tau(a_{i_k})=0$ for all $1\leq k\leq m$, 
	we  have $\tau(a_{i_1}a_{i_2}\cdots a_{i_m})=0$. The random variables $a_1,\cdots, a_m$ are free or freely independent if the unital $\ast$-subalgebras generated by them are free. 

For any self-adjoint (resp. unitary) element $a\in\A$, the {\bf law} or the {\bf distribution} $\mu$ of $a$ is a probability measure on $\R$ (resp. $\mathbb{T}$) such that whenever $f$ is a bounded continuous function on $\R$ (resp. on $\mathbb{T}$), we have
		$$\int f\;d\mu = \tau(f(a)).$$

For a measure $\mu$ on the real line, the Cauchy transform of $\mu$ is given by
$$G_\mu(z) := \int_\R\frac{1}{z-x}\:d\mu(x) = \tau((z-a)^{-1}), \quad z\in\C^+.$$
The Cauchy transform $G_\mu$ maps the upper half plane 
$\mathbb{C}^+$ into the lower half plane $\mathbb{C}^-$. It satisfies the asymptotic property $\lim_{y\uparrow
 +\infty}iyG_\mu(iy) = \mu(\R)$. The reader is referred to \cite{AkhiezerBook} for results about Cauchy transform.
The measure $\mu$ can be recovered from its Cauchy transform $G_\mu$ using the Stieltjes inversion formula, that expresses $\mu$ as a weak limit
\begin{equation}
\label{CauchyInv}
d\mu(x) = \lim_{y\downarrow 
 0} -\frac{1}{\pi}\textrm{Im}\:G_\mu(x+iy)\:dx.
\end{equation}

The $R$-transform of $\mu$ is defined by
\begin{equation}
\label{Rtransform}
  \mathcal{R}_\mu(z)=G_\mu^{\langle -1 \rangle}(z)-\frac{1}{z}
\end{equation}
where $G_\mu^{\langle -1\rangle}$ means the inverse function to $G_\mu$ in a truncated Stolz angle $\{z\in\C\vert\; \mathrm{Im}\,z>\beta, \left\vert \mathrm{Re}\,z\right\vert<\alpha\,\mathrm{Im}\,z\}$ for some $\alpha, \beta>0$.

For a measure $\mu$ on the unit circle $\mathbb{T}$, we consider the moment generating function on the open unit disk $\bbD$:
$$\psi_\mu(z) = \int_{\mathbb{T}}\frac{\xi {z}}{1-\xi{z}}\;d\mu(\xi)\quad z\in \bbD.$$
The $\eta$-transform of $\mu$ is defined as
\[
\eta_\mu(z) =\frac{\psi_\mu(z)}{1+\psi_\mu(z)}.
\]
Then the measure $\mu$ can be recovered using the Herglotz representation theorem, as a weak limit
\begin{equation}
\label{MultInv}
d\mu(e^{-i\theta}) = \lim_{r\uparrow
 1}\frac{1}{2\pi}\textrm{Re}\left(\frac{1+\eta_\mu(re^{i\theta})}{1-\eta_\mu(re^{i\theta})}\right)d\theta.
\end{equation}
When $\eta_\mu'(0)\neq 0$ (which is equivalent to the condition that $\mu$ has nonzero first moment), the $\Sigma$-transform and $S$-transform of $\mu$ are defined to be
\begin{equation}
\label{eq:Stransform}
  \Sigma_\mu(z)=\frac{\eta_\mu^{\langle -1 \rangle}(z)}{z},\quad \text{and}\quad S_\mu(z)=\Sigma_\mu\left(\frac{z}{1+z}\right),
\end{equation}
where these functions are defined in a neighborhood of zero.

\subsection{Free Brownian Motions}
In free probability, the semicircle law plays a similar role as the Gaussian distribution in classical probability. 
The semicircle law $\sigma_t$ with variance $t$ is compactly supported in the interval $[-2\sqrt{t}, 2\sqrt{t}]$ with density
$$d\sigma_t(x) = \frac{1}{2\pi t}\sqrt{4t-x^2}.$$

\begin{definition} 
	\begin{enumerate}
		\item	A free semicircular Brownian motion $s_t$ in a $W^*$-probability space $(\A, \tau)$ is a weakly continuous free stochastic process $(x_t)_{t\geq 0}$ with free and stationary semicircular increments.
	\item	A free circular Brownian motion $c_t$ has the form $\frac{1}{\sqrt{2}}(s_t+i s_t')$ where $s_t$ and $s_t'$ are two freely independent free semicircular Brownian motions.
	\end{enumerate}
		\end{definition}

In the unitary group $U(N)$, we can consider a Brownian motion $X_t$ on the Lie algebra $gl(N)$, after fixing an $\textrm{Ad}$-invariant inner product. Taking the exponential map gives us a unitary Brownian motion. More precisely, the unitary Brownian motion $U_t=U_t(N)$ can be obtained by solving the It\^{o} differential equation 
$$dU_t = iU_t\,dX_t - \frac{1}{2}U_t\,dt, \quad U_0=I.$$

\begin{definition}
In free probability, the \emph{free} unitary Brownian motion can be obtained by solving the \emph{free} It\^{o} differential equation
\begin{equation}\label{eqn:UBM}
du_t = iu_t\,ds_t - \frac{1}{2}u_t\,dt, \quad u_0=I. 
\end{equation}
where $s_t$ is a free semicircular Brownian motion. The free multiplicative Brownian motion $b_t$ is the solution for the 
\emph{free} It\^{o} stochastic differential equation
\begin{equation}\label{eqn:GBM}
    db_t=b_t\,dc_t, \quad b_0=I. 
\end{equation}
\end{definition}
We note that the \emph{right} increments of the free unitary Brownian motion $u_t$ are free. In other words, for every $0<t_1<t_2<\cdots<t_n$ in $\mathbb{R}$, the elements
\[
      u_{t_1}, u_{t_1}^{-1}u_{t_2}, \cdots, u_{t_{n-1}}^{-1}u_{t_n}
\]
form a free family. Similarly, one can show that the process $b_t$ has free right increments. These stochastic processes were introduced by Biane in \cite{Biane1997b}. He proved that the large-$N$ limit in $*$-distribution of the unitary Brownian motion $U_t=U_t(N)$ is the free unitary Brownian motion, and conjectured that the large $N$ limit of the Brownian motion on $GL(N; \mathbb{C})$ is the free multiplicative Brownian motion $b_t$. Kemp \cite{Kemp2016} proved that $b_t$ is the limit in $\ast$-distribution of the Brownian motion on $GL(N;\C)$.

The connection between the Brownian motions on the Lie groups $U(N)$ and $GL(N;\mathbb{C})$ is natural. The heat kernel function on $GL(N; \mathbb{C})$ is the analytic continuation of that on $U(N)$ (see 
\cite{Hall2001, Liao} for instance). Consider now the free unitary Brownian motion with initial condition $uu_t$ where $u$ is a unitary random variable freely independent from $\{ u_t\}_{t>0}$. The process $uu_t$ is the solution of the free stochastic differential equation in (\ref{eqn:UBM}) with initial condition $u$. Similarly, the solution $g_t$ of the free stochastic differential equation
\begin{equation}\label{eqn:GBMu}
   dg_t=g_t\,dc_t, \quad g_0=u.
\end{equation}
has the form $g_t=ub_t$.

\subsection{Free Additive Convolution}
\label{FAConv}
Our main results show that the Brown measure of the free circular Brownian motion with self-adjoint initial condition has direct connections with the spectral distribution of the free additive Brownian motion with the same self-adjoint initial condition and some analytic functions related to it. In this section, we review some relevant facts about the free additive convolution to setup the notations. 

Suppose that the self-adjoint random variables $x,y\in \A$ are freely independent. It is known that the distribution of $x+y$ is determined by the distributions of $x$ and $y$. The free additive convolution of $x$ and $y$ is then defined to be the distribution of $x+y$. 
The subordination relation in free convolution was first established by Voiculescu~\cite{Voiculescu1993} for free additive convolutions under some generic conditions, and was further developed by Biane \cite{Biane1998} to free multiplicative convolutions and again by Voiculescu \cite{DVV-general} to a very general setting (see also \cite{BB2007new}). There exists a unique pair of analytic self-map $\omega_1, \omega_2: \mathbb{C}^+ \to\mathbb{C}^+$ such that
	\begin{enumerate}
		\item $\textrm{Im}\, \omega_j(z)\geq \textrm{Im}\, z$ for all $z\in \mathbb{C}^+$, $j=1,2$;
		\item $G_x(\omega_1(z)) = G_y(\omega_2(z)) = (\omega_1(z)+\omega_2(z)-z)^{-1}$ for all $z\in\mathbb{C}^+$;
		\item $G_{x+y}(z) =  G_x(\omega_1(z)) = G_y(\omega_2(z)) $ for all $z\in \mathbb{C}^+$.
	\end{enumerate}
Point 3 tells us if we could compute (one of) the subordination functions $\omega_1$ and $\omega_2$, we could compute the Cauchy transform of $x+y$ in terms of the Cauchy transform of $x$ or $y$. The Cauchy transform of $x+y$ then determines the law of $x+y$ by \eqref{CauchyInv}. Although the subordination functions, in general, cannot be computed explicitly, a lot of regularity results can be deduced from the subordination relation (see \cite{Belinschi2008, BercoviciVoiculescu1998} or the survey \cite[Chapter 6]{survey-2013}). Denote by $\mu$ and $\nu$ the spectral distributions of $x$ and $y$ respectively. The free additive convolution of $\mu$ and $\nu$ is defined to be the spectral distribution of $x+y$ and is denoted as $\mu\boxplus \nu$. 
The $\mathcal{R}$-transform~\eqref{Rtransform} linearizes the free additive convolution in the sense that $\mathcal{R}_{\mu\boxplus \nu}(z)=\mathcal{R}_\mu(z)+\mathcal{R}_\nu(z)$ in the domain where all the three $\mathcal{R}$-transforms are defined. 

In the special case when $x$ is an arbitrary self-adjoint variable $x_0$ with law $\mu$ and $y$ is the semicircular variable $s_t$, we denote $\omega_1$ by $F_t$ the subordination function such that
$$G_{x_0+s_t}(z) = G_{x_0}(F_t(z)).$$
Biane \cite{Biane1997} computed that \[H_t(z) = z+tG_{x_0}(z), \qquad z\in\C^+\]
 is the \emph{left inverse} of $F_t$; that is, $H_t(F_t(z)) = z$ for $z\in\C^+$.

Define the function 
\begin{equation}\label{def:v_t}
v_t(a) = \inf\left\{b>0: \int_{\mathbb{R}} \frac{d\mu(x)}{(a-x)^2+b^2}\leq \frac{1}{t}\right\}, \quad a\in\R.
\end{equation}
So, whenever $v_t(a)>0$, $v_t(a)$ is the unique positive $b$ with 
\begin{equation}
\label{b.eq}
\int_{\mathbb{R}} \frac{d\mu(x)}{(a-x)^2+b^2} = \frac{1}{t}.
\end{equation}
Therefore, $v_t$ satisfies the equality given in the following lemma.
\begin{lemma}[Lemma 2 of \cite{Biane1997}]\label{lemma:identity-vt}
	\label{IneqT}
	If $v_t(a) > 0$, then 
		\begin{equation}\label{eq:identity-v-t-a}
	\int_{\mathbb{R}} \frac{d\mu(x)}{(a-x)^2+v_t(a)^2} = \frac{1}{t}.
	\end{equation}
\end{lemma}
The function $v_t$ is analytic at $a$ whenever $v_t(a)>0$. And $H_t$ takes $\{a+iv_t(a): a\in\R\}$ to the real line, since, for each $a\in\R$, if there exists a $b>0$ such that
\[\textrm{Im}\:H_t(a+ib) = b\left(1-t\int_\R\frac{1}{(a-x)^2+b^2}\,d\mu(x)\right)=0,\] 
then it is the unique $b$ which satisfies \eqref{b.eq}. 

We summarize Biane's result as follow.

\begin{proposition}[Proposition 1 of \cite{Biane1997}]
	\label{HFMap}
	The subordination function $F_t$ satisfying 
	$$G_{x_0+s_t}(z) = G_{x_0}(F_t(z))$$
	defined on $\mathbb{C}^+$
	is a one-to-one conformal mapping into $\mathbb{C}^+$.
	The inverse of $F_t$ can be computed explicitly as
	$$H_t(z) = z+tG_{x_0}(z)$$
	which is conformal from $\{a+ib\in \mathbb{C}^+: b > v_t(a)\}$ to $\mathbb{C}^+$. The function $H_t$ extends to a homeomorphism from $\overline{\{a+ib\in \mathbb{C}^+: b > v_t(a)\}}$ to $\overline{\mathbb{C}^+}$ since the domain of $H_t$ is a Jordan domain.
\end{proposition}

When $v_t(a)>0$, the point $a+iv_t(a)\in \C^+$ is mapped to the real line by $H_t$. In fact, the law of $x_0+s_t$ at the point $H_t(a+iv_t(a))$ is computed and expressed in terms of $v_t$.

\begin{proposition}[Corollary 3 and Lemma 5 of \cite{Biane1997}]
	\label{AddDensity}
	Let 
	$$\psi_t(a) = H_t(a+iv_t(a)) = a+t\int_\R\frac{(a-x)\,d\mu(x)}{(a-x)^2+v_t(a)^2}.$$
	Then $\psi_t:\R\to\R$ is a homeomorphism and at the point $\psi_t(a)$ the law $\mu_t$ of $x_0+s_t$ has the density given by
	$$p_t(\psi_t(a))=\frac{v_t(a)}{\pi t}.$$
	Moreover, the function $\psi_t$ satisfies
	$$\psi_t'(a)\geq \frac{2}{t}v_t(a)^2(1+v_t'(a)^2)>0$$
	for any $a\in U_t$.
	\end{proposition}

\begin{proposition}[Proposition 3 of \cite{Biane1997}]
	\label{Acomponent}
The support of the law $\mu_t$ of $x_0+s_t$ is the closure of its interior and the number
of connected components of $U_t$ is a non-increasing function of $t$. 
\end{proposition}

\subsection{Free Multiplicative Convolution}\label{section:multiplicative_convolution}
We will show that the Brown measure for the free multiplicative Brownian motion with unitary initial condition can be described by certain analytic functions and their geometric properties related to the free unitary Brownian motion with the same unitary initial condition. We review some basic facts about free multiplicative convolution in this section. 

Let $u, v\in\A$ be two freely independent unitary random variables, with spectral distributions $\mu$ and $\nu$ respectively. The distribution of $uv$ is determined by $\mu$ and $\nu$ and is denoted by $\mu\boxtimes \nu$; it is called the free multiplicative convolution of $\mu$ and $\nu$. The $\Sigma$-transform~\eqref{eq:Stransform} has the property that $\Sigma_{\mu\boxtimes \nu}(z)=\Sigma_\mu(z)\Sigma_\nu(z)$ in the domain where all these $\Sigma$-transforms are defined. We refer the readers to \cite{BB2005, BV1992} for more details on free multiplicative convolution on $\mathbb{T}$. The subordination relation for free additive convolution was extended by Biane \cite{Biane1998} to the multiplicative case.

\begin{theorem}[\cite{ BB2007new, Biane1998}]\label{thm:sub-2}
Let $(\A, \tau)$ be a $W^*$-probability space, and $u, v\in \A$ two unitary random variables that are free to each other with distributions $\mu$ and $\nu$, respectively. If $\mu$ is not the Haar measure on $\mathbb{T}$ and $\nu$ has nonzero first moment,
then there exists a unique pair of analytic self-maps $\omega_1, \omega_2: \mathbb{D}\rightarrow \mathbb{D}$ such that
\begin{enumerate}
 \item For $i=1, 2$, $|\omega_i(z)|\leq |z|$ for $z\in \mathbb{D}$. In particular, $\omega_i(0)=0$.
 \item $\omega_1(z)\omega_2(z)=z\eta_{\mu\boxtimes\nu}(z)$, for all $z\in\mathbb{D}$. 
 \item $\eta_{\mu}(\omega_1(z))=\eta_\nu(\omega_2(z))=\eta_{\mu\boxtimes \nu}(z)$. 
\end{enumerate}
The functions $\eta_\mu, \eta_\nu$ are defined in Section \ref{FreeProb}.
\end{theorem}

 As in the additive case, point 3 tells us if we could compute (one of) the subordination functions $\omega_1$ and $\omega_2$, we could compute the function $\eta_{\mu\boxtimes \nu}$ of $uv$ in terms of $\eta_u$ or $\eta_v$. The function $\eta_{uv}$ determines the law of $uv$ by \eqref{MultInv}. The subordination functions are in general impossible to compute explicitly in the multiplicative case, as in the additive case. 
 
When any one of unitaries $u, v$ is a Haar unitary, one can check by the definition of free independence that all moments of $uv$ vanish and hence the distribution of $uv$ is always the Haar measure (uniform measure) on $\bT$. We denote by $h$ a Haar unitary and also by $h$ the Haar measure on $\mathbb{T}$. Note that $\eta_{h}(z)\equiv{0}$. When $u=h$ is a Haar unitary, the subordination function of $\eta_{h\boxtimes \nu}=\eta_h$ with respect to $\eta_\mu=\eta_h$ is not unique. However, we shall see that there is a canonical choice in our study. 

From now on, we fix a unitary operator $u$ that is free to the free unitary Brownian motion $u_t$; we do not restrict $u$ to be not a Haar unitary variable. The spectral distribution of $uu_t$ has been studied by the second author in \cite{Zhong2014, Zhong2015}. This is the multiplicative analogue of Biane's work presented in Section \ref{FAConv}. We shall now briefly review these results. Let $\mu$ be the distribution of $u$ and $\lambda_t$ the distribution of $u_t$. 
It is convenient for us to use the subordination function with respect to $u^*$ but not $u$ to describe our main results (essentially due to how the measure is recovered by the $\eta$-transform as shown in \eqref{MultInv}).
Then we define $\bar{\mu}$ by
$$d\bar{\mu}(e^{ix}) = d\mu(e^{-ix}),$$ 
which is the distribution of $u^*$. 
When $u$ is not a Haar unitary, denote by $\omega_t$ the subordination function of $\eta_{\bar{\mu}\boxtimes\lambda_t}$ with respect to $\eta_{\bar{\mu}}$ as in Theorem \ref{thm:sub-2}. That is,
\begin{equation}\label{eqn:eqn-subordination}
	\eta_{\bar{\mu}\boxtimes \lambda_t}(z)=\eta_{\bar{\mu}}(\omega_t(z)).
\end{equation}
When $u$ is a Haar unitary, the subordination function is chosen to be $\omega_t(z)=e^{-t/2}z$. The subordination relation (\ref{eqn:eqn-subordination}) also holds (since both sides are zero).

We now describe the left inverse function of $\omega_t$ (see \cite[Lemma 3.4]{Zhong2014} and \cite[Proposition 2.3]{Zhong2015}). 
Set 
\begin{equation}\label{eqn:St}
\St(z)=\exp\left(\frac{t}{2}\int_{\mathbb{T}}\frac{1+\xi z}{1-\xi z}d\bar{\mu}(\xi) \right)=\exp\left(\frac{t}{2}\int_{\mathbb{T}}\frac{\xi+ z}{\xi -z}d{\mu}(\xi) \right), \quad z\in \C\setminus \sigma(u). 
\end{equation}
{It is known that $\Sigma_{t,\bar{\mu}}(z)$ is the L\'{e}vy–Khintchine representation of a free infinitely divisible distribution on the unit circle \cite{BV1992}.}
When $u$ is a Haar unitary, (\ref{eqn:St}) is reduced to
\[
\Sigma_{t,\bar{h}}(z)=\Sigma_{t,{h}}(z)\equiv e^{t/2}. 
\]
By \cite{Zhong2015}, the left inverse of $\omega_t$ is the function 
\[
\Pt(z)=z\St(z).
\]
That is, $\Pt(\omega_{t}(z))=z$ for all $z\in \mathbb{D}$. 

Denote
\begin{equation}\label{def:Omega_t}
\Ot=\{\omega_{t}(z): z\in \mathbb{D} \}.
\end{equation}
The subordination function $\omega_t$ is a one-to-one conformal mapping from $\mathbb{D}$ onto $\Ot$ and can be extended to a homeomorphism from $\overline{\mathbb{D}}$ onto $\overline{\Ot}$.
It is known \cite[Lemma 3.2]{Zhong2014} that
\begin{equation}
\label{OmegaDef}
\Ot=\{z\in\mathbb{D}:|\Phi_{t, \bar{\mu}}(z)|<1 \}.
\end{equation}

We now describe the boundary of $\Ot$ and the density formula of $\mu\boxtimes \lambda_t$. More details can be found in Section \ref{some-sets}. Following \cite[Page 1361]{Zhong2015}, we set 
\begin{equation}\label{Ut-00}
\begin{split}
U_{t, \bar{\mu}}
=\left\{e^{i\theta}\in\mathbb{T} : \int_{-\pi}^{\pi}\frac{1}{|1-e^{i(\theta+x)}|^2}d\bar{\mu}(e^{ix})>\frac{1}{t}\right\}
\end{split}
\end{equation}
and $U_{t, \bar{\mu}}^c=[-\pi,\pi]\backslash U_{t, \bar{\mu}}$. 
We also define a function 
$r_t: [-\pi,\pi]\rightarrow (0,1]$ as
\begin{equation}
\begin{split}\label{def:r_t_theta}
r_t(\theta)
=\sup\left\{0< r<1: \frac{r^2-1}{2\log r}\int_{-\pi}^{\pi}\frac{1}{|1-re^{i(\theta+x)}|^2}d\bar{\mu}(e^{ix}) 
< \frac{1}{t}\right\}.
\end{split}
\end{equation}
Indeed, whenever $r_t(\theta)<1$, $r_t(\theta)$ is the unique $r\in(0,1)$ such that
\[
\frac{r^2-1}{2\log r}\int_{-\pi}^{\pi}\frac{1}{|1-re^{i(\theta+x)}|^2}d\bar{\mu}(e^{ix}) 
= \frac{1}{t}.
\]

\begin{remark}
	\label{rem:rtanalyticity}
We can prove that the function $r_t$ is analytic when $0<r_t(\theta)<1$ by the implicit function theorem applied to the function $F(r,\theta)=\log|\Phi_{t,\bar\mu}(re^{i\theta})|$. For each fixed $\theta$ such that $0<r_t(\theta)<1$, $r_t(\theta)$ is the unique solution in the unit disk such that $F(r_t(\theta),\theta) = 0$. By page 1360 of \cite{Zhong2015}, we can write
\[\log\vert\Phi_{t,\bar\mu}(re^{i\theta})\vert = (\log r) h_t(r,\theta)\]
for some function $h_t$ such that $h_t(r_t(\theta),0) = 0$ and $\partial h_t/\partial r<0$ for all $r<1$. (We will revisit this factorization of $\log\vert\Phi_{t,\bar\mu}(re^{i\theta})\vert$ in Section~\ref{some-sets}.) Then we can compute
\[\left.\frac{\partial}{\partial r}\right|_{r=r_t(\theta)}\log\vert\Phi_{t,\bar\mu}(re^{i\theta})\vert = \frac{h_t(r_t(\theta),\theta)}{r_t(\theta)}+(\log r_t(\theta))\left.\frac{\partial h_t}{\partial r}\right|_{r=r_t(\theta)}>0\]
since $h_t(r_t(\theta),\theta) = 0$ and $\partial h_t/\partial r<0$. Now, it follows from the implicit function theorem that $r_t$ is an analytic function in $\theta$.
\end{remark}

The following theorem summarizes the regularity of $r_t$.
\begin{proposition}[\cite{Zhong2015}]
	\label{characterization-0}
The function $r_t$ defined in (\ref{def:r_t_theta}) is continuous everywhere and analytic at $\theta$ whenever $r_t(\theta)<1$. 
The sets $\Omega_{t, \bar{\mu}}$ and $\partial\Omega_{t, \bar{\mu}}$ can be characterized by the function $r_t(\theta)$ as follows.
\begin{enumerate}
	\item $\Omega_{t, \bar{\mu}}=\{re^{i\theta}: 0\leq r<r_t(\theta), \theta\in[-\pi,\pi]\}$.
	\item $\partial\Ot=\{r_t(\theta)e^{i\theta}: \theta\in[-\pi,\pi] \}$ and $\partial\Ot$ is a continuous closed curve which encloses the origin. For $\theta\in U_{t,\bar{\mu}}$, the value $r_t(\theta)$ is the unique solution $r\in (0,1)$ of the following equation:
	\[
	      \frac{r^2-1}{2\log r}\int_{-\pi}^{\pi}\frac{1}{|1-re^{i(\theta+x)}|^2}d\bar{\mu}(e^{ix}) 
	    = \frac{1}{t}.
	\]
	\item The map $\Phi_{t,\bar{\mu}}(r_t(\theta)e^{i\theta})\mapsto r_t(\theta)e^{i\theta}$ is a homeomorphism from $\mathbb{T}$ onto $\partial\Omega_{t, \bar{\mu}}$. 
\end{enumerate}
\end{proposition}
\begin{proof}
	That $r_t$ is continuous follows from \cite[Proposition 3.7]{Zhong2015}. The analyticity of $r_t$ when $r_t(\theta)<1$ follows from Remark~\ref{rem:rtanalyticity}.
	
	Point 1 is \cite[Theorem 3.2(1)]{Zhong2015}. The display equation in Point 2 is just the one preceding Remark~\ref{rem:rtanalyticity}. That $\partial\Ot=\{r_t(\theta)e^{i\theta}: \theta\in[-\pi,\pi] \}$ follows from \cite[Corollary 3.3]{Zhong2015} since $\Phi_{t,\bar\mu}$ is the left inverse of $\omega_t$ (the function $\omega_t$ is called $\eta_t$ in \cite{Zhong2015}). The boundary curve $\partial\Ot$ encloses the origin because we always have $r_t(\theta)>0$. Point 3 also follows from \cite[Proposition 3.7]{Zhong2015}; by Proposition 3.7 in \cite{Zhong2015} we know that the maps $\theta\mapsto r_t(\theta)e^{i\theta}$ and $e^{i\theta}\mapsto \Phi_{t,\bar\mu}(r_t(\theta)e^{i\theta})$ are homeomorphisms.
	\end{proof}

\begin{remark}
	The paper \cite{Zhong2015} did not include the Haar unitary case. However, 	when $u$ is a Haar unitary, the above description for the boundary set also holds. 
Indeed, In this case $\Omega_{t, \bar{h}}$ is the disk centered at the origin with radius $e^{-t/2}$. In fact, one can verify that, for any $0<r<1$, 
\begin{align*}
\frac{r^2-1}{2\log r}\int_{-\pi}^{\pi} &\frac{1}{|1-re^{i(\theta+x)}|^2} dh (e^{ix})\\
   &=\frac{1}{2\pi}\frac{r^2-1}{2\log r}\int_{-\pi}^{\pi}\frac{1}{|1-re^{i x}|^2}dx =-\frac{1}{2\log r}. 
\end{align*}
Hence $r_t(\theta)=e^{-t/2}$ for all $\theta$ in this case by the definition (\ref{def:r_t_theta}) and results in Proposition \ref{characterization-0} are also valid. 
\end{remark}
As the inverse map of $\omega_t$, the restriction of the map $\Pt$ to $\Ot$ is conformal map and can be extended to a homeomorphism of $\overline{\Ot}$ onto $\overline{\mathbb{D}}$. 
We can write
\[
     \arg\left( \Phi_{t,\bar{\mu}}(r_t(\theta)e^{i\theta})\right)=\theta+t\int_{-\pi}^\pi \frac{r_{t}(\theta) \sin(\theta+x)}{|1-r_t(\theta)e^{i(\theta+x)}|^2}d\bar{\mu}(e^{ix}). 
\]
We can now obtain the following result, which is a slight modification from \cite[Theorem 3.8]{Zhong2015}.
\begin{theorem}\label{thm:mut}
	Let $\mu_t$ be the free multiplicative convolution of $\mu$ and $\lambda_t$, the spectral measure of $uu_t$.
	Then $\mu_t$ has a density $p_t$ with respect to the Haar measure given by 
	\begin{equation}\label{eq:3.17}
	p_t(\Phi_{t, \bar{\mu}} (r_t(\theta)e^{i\theta})) =-\frac{\log r_t(\theta)}{\pi t}.
	\end{equation}
\end{theorem}
\begin{proof}
	When $u$ is a Haar unitary, $r_t(\theta)=e^{-t/2}$ for all $\theta$. The formula (\ref{eq:3.17}) is reduced to the Haar measure on tbe unit circle $\bT$. 
	
	Next, we consider $u$ that is not a Haar unitary. Let $q_t$ be the density of $u^*u_t$. Directly applying the result in \cite[Theorem 3.8]{Zhong2015} (in which $u$ was playing the role of $u^*$ here) gives
	\[q_t(\overline{\Phi_{t,\bar{\mu}}(r_t(\theta)e^{i\theta})}) =-\frac{\log r_t(\theta)}{\pi t}.\]
	The spectral distribution $\lambda_t$ of $u_t$ is symmetric about the $x$-axis; hence $(uu_t)^* = u_t^*u^*$ has the same distribution as $u^*u_t^*$ and $u^*u_t$ in the tracial $W^*$-probability space $(\A, \tau)$. We then have
	 $$p_t(\Phi_{t, \bar{\mu}} (r_t(\theta)e^{i\theta})) = q_t(\overline{\Phi_{t,\bar{\mu}}(r_t(\theta)e^{i\theta})}) =-\frac{\log r_t(\theta)}{\pi t},$$
	 giving the desired result.
	\end{proof}

\begin{proposition}[Corollary 3.9 of \cite{Zhong2015}]
	\label{Mcomponent}
The support of the law $\mu_t$ of $uu_t$ is the closure of its interior and the number
of connected components of $U_{t,\mu}$ is a non-increasing function of $t$. 
\end{proposition}

\subsection{The Brown Measure}
If $x\in \A$ is a normal random variable, then there exists a spectral measure $E_x$ such that
$$x = \int \lambda \;dE_x(\lambda),$$
by the spectral theorem. The law $\mu_x$ of $x$ is then computed as 
\[\mu_x(A) = \tau(E_x(A))\]
for each Borel set $A$.

However, if $x$ is not normal, then the spectral theorem does not apply. The Brown measure was introduced by Brown \cite{Brown1986} and is a natural candidate of the spectral distribution of a non-normal operator.

Given $x\in\A$, the Fuglede--Kadison determinant \cite{FugledeKadison1952} $\mathcal{D}(x)$ of $x$ is defined as
\[
\mathcal{D}(x)=\exp[\tau(\log(|x|))]\in [0,\infty).
\]
Define a function $L_x$ on $\C$ by
$$L_x(\lambda) = \log \mathcal{D}(x) = \tau[\log(|a-\lambda|)].$$
This function is subharmonic. For example, when $A\in M_n(\C)$, then
$$L_A(\lambda) = \log |\det (A-\lambda I)|^{1/n}.$$
The {\bf Brown measure} \cite{Brown1986} of $x$ is then defined to be the distributional Laplacian of $L_x$
$$\rho_x = \frac{1}{2\pi}\Delta L_x.$$

In this paper, we compute the Brown measures using the following strategy. We first regularize the function $L_x$ by looking at
\[\lambda\mapsto \tau[\log((\lambda-x)^*(\lambda-x)+\varepsilon)]\]
for $\varepsilon>0$. For any $\varepsilon>0$, the above quantity is always well-defined as a real number. We then take the limit as $\varepsilon\downarrow 0$ and attempt to take Laplacian. In the cases that we consider, the function $L_x$ is indeed analytic on an open set in which the Brown measure of $x$ has full measure. The Brown measure of $x$ is then the $1/(2\pi)$ multiple of the ordinary Laplacian of $L_x$; that is
\begin{equation}
\label{BrownAsLimit}
d\rho_x(\lambda) =\frac{1}{4\pi}\Delta_{\lambda} \lim_{\varepsilon\downarrow
	0}\tau[\log((\lambda-x)^*(\lambda-x)+\varepsilon)]\,d^2\lambda
\end{equation}
where $d^2\lambda$ denotes the Lebesgue measure on $\R^2$. See \cite[Chapter 11]{MingoSpeicherBook} for more details and \cite{HaagerupSchultz2007} for discussions on unbounded operators.

In this paper, we compute the Brown measures of a sum or a product of two freely independent random variables. The paper \cite{BSS2018} studied the Brown measure of polynomials of several free random variables using operator-valued free probability and linearization \cite{Anderson2013}; it does not appear to be easy to apply their framework to get analytic results in our case. 

\section{Free Circular Brownian Motion}
\label{CircularCase}
Let $(c_t)_{t\geq 0}$ be a free circular Brownian motion. We write $x_t = x_0+c_t$ and $x_{t,\lambda} = \lambda - x_t$, where $\lambda\in\C$.
Define
\begin{equation}
\label{SEq}
S(t,\lambda, \varepsilon) = \tau[\log(x_{t,\lambda}^*x_{t, \lambda}+\varepsilon)]
\end{equation}
for $\varepsilon>0$.
When $t=0$, since $x_0$ is self-adjoint,
\[S(0,\lambda, \varepsilon) = \tau[\log((x_0-\lambda)^*(x_0-\lambda)+\varepsilon)]\]
can be written as an integral, instead of a trace,
\[S(0,\lambda, \varepsilon) = \int_\R \log(|x-\lambda|^2+\varepsilon)\,d\mu(x)\]
where $\mu$ is the spectral distribution of $x_0$.

In order to compute the Brown measure of $x_0+c_t$, we need to compute the Laplacian $\Delta_\lambda S(t, \lambda, 0)$ of $S(t,\lambda,0)$ with respect to $\lambda$, where
\begin{equation}
\label{St0def}
S(t,\lambda,0) = \lim_{\varepsilon\downarrow
 0} S(t, \lambda, \varepsilon).
\end{equation}
 In Section \ref{addPDE}, we compute a first-order, nonlinear partial differential equation of Hamilton--Jacobi type of $S$. We then solve a system of ODEs that depends on $\lambda$ and the initial condition $\varepsilon_0$. We try to choose, for each $\lambda$, an initial condition $\varepsilon_0$ such that
\begin{enumerate}
	\item The lifetime of the solution is $t$, and
	\item  $\lim_{s\uparrow
 t} \varepsilon(s)=0.$
	\end{enumerate}
Then we use the solution of the Hamilton--Jacobi equation to compute $S(t,\lambda, 0)$, which is in terms of $\lambda$ and the initial condition $\varepsilon_0$, and compute the Laplacian $\Delta_\lambda S(t,\lambda,0)$.

\subsection{The Hamilton--Jacobi Equation}\label{addPDE}
In this section, we find a first-order, nonlinear partial differential equation of Hamilton--Jacobi type of the function $S$ defined in \eqref{SEq}
\[\frac{\partial S}{\partial t} = \varepsilon\left(\frac{\partial S}{\partial \varepsilon}\right)^2.\]
The variable $\varepsilon$ is positive. 

We remark that this PDE corresponds to the formal large-$N$ limit of the PDE computed in \cite[Eq. (27)]{BGNTW2015} by Nowak et al. after identifying $\varepsilon$ with $|w|^2$.

\subsubsection{The PDE of $S$}
We first compute the time-derivative of $S$, using free It\^{o} calculus. Free stochastic calculus was developed in the 1990s by Biane, K\"{u}mmerer, Speicher, and many others; see for example \cite{Biane1997b, BianeSpeicher1998,  KummererSpeicher1992}. 

Suppose that $f_t$ and $h_t$ are processes adapted to $c_t$. The following ``stochastic differentials" involving these processes can be computed and simplified as below (See \cite[Theorem 4.1.2]{BianeSpeicher1998}):
\begin{equation}
\label{eq:Ito}
\begin{split}
	dc_t \,f_t \,dc_t^* &= dc_t^* \,f_t \,dc_t = \tau(f_t)\,dt\\
	dc_t \,f_t \,dc_t &= dc_t^* \,f_t \,dc_t^* = 0\\
	dc_t \,dt &= dc_t^*\, dt = 0\\
	\tau(f_t\, dc_t\, h_t) &= \tau(f_t\, dc_t^*\, h_t) = 0.
	\end{split}
\end{equation}
We can use the free It\^{o} product rule for processes $a_t^{(1)},\ldots,a_t^{(n)}$ adapted to $c_t$:
\begin{equation}
\label{ItoForm}
\begin{split}
	d(a_t^{(1)}\cdots a_t^{(n)}) &= \sum_{j=1}^n(a_t^{(1)}\cdots a_t^{(j-1)})\,da_t^{(j)}(a_t^{(j+1)}\cdots a_t^{(n)}) \\
	&+ \sum_{1\leq j<k\leq n}(a_t^{(1)}\cdots a_t^{(j-1)})\,da_t^{(j)}(a_t^{(j+1)}\cdots a_t^{(k-1)})\,da_t^{(k)}(a_t^{(k+1)}\cdots a_t^{(n)}).
\end{split}
\end{equation}
\begin{lemma}
	The time-derivative of $S$ satisfies
	\begin{equation}\label{DiffEq}
\frac{\partial S}{\partial t}(t,\lambda,\varepsilon) =\varepsilon\tau[(x_{t,\lambda}^*x_{t,\lambda}+\varepsilon)^{-1}]\tau[(x_{t,\lambda}^*x_{t,\lambda}+\varepsilon)^{-1}].
	\end{equation}
	\end{lemma}
\begin{proof}
Fix $t\geq 0$ and $\lambda\in\mathbb{C}$. For any $\varepsilon$ with $\text{Re}(\varepsilon)>0$, the operator $x_{t,\lambda}^*x_{t,\lambda}+\varepsilon$ is invertible. We can express the function $S(t, \lambda, \cdot)$ defined in \eqref{SEq} as a power series of $\varepsilon$ and hence can be analytically continued to the right half plane.
	
	For each $|\varepsilon|>\|x_{t,\lambda}^*x_{t,\lambda}\|$, we can expand $S(t,\lambda, \varepsilon) = \tau(\log(x_{t,\lambda}^*x_{t,\lambda}+\varepsilon))$ into power series
	\begin{equation}
	\label{PowerSeries}
	\tau(\log(x_{t,\lambda}^*x_{t,\lambda}+\varepsilon)) = \log \varepsilon + \sum_{n=1}^\infty\frac{(-1)^{n-1}}{n\varepsilon^n}\tau[(x_{t,\lambda}^*x_{t,\lambda})^n].
	\end{equation}
If we apply \eqref{ItoForm} to $d((x_{t,\lambda}^*x_{t,\lambda})^n)$, we get
$$\frac{\partial}{\partial t}\tau[(x_{t,\lambda}^*x_{t,\lambda})^n] = n\sum_{j=1}^n \tau[(x_{t,\lambda}^*x_{t,\lambda})^j]\tau[(x_{t,\lambda}^*x_{t,\lambda})^{n-j-1}].$$
Thus, using \eqref{PowerSeries}, we have
\begin{align*}
	\frac{\partial S}{\partial t}(t,\lambda,\varepsilon) &= \varepsilon\left(\frac{1}{\varepsilon}\sum_{k=0}^\infty \frac{(-1)^k}{\varepsilon^k}\tau[(x_{t,\lambda}^*x_{t,\lambda})^k]\right)\left(\frac{1}{\varepsilon}\sum_{l=0}^\infty \frac{(-1)^l}{\varepsilon^l}\tau[(x_{t,\lambda}^*x_{t\lambda})^l]\right)\\
	& = \varepsilon\tau[(x_{t,\lambda}^*x_{t,\lambda}+\varepsilon)^{-1}]\tau[(x_{t,\lambda}^*x_{t,\lambda}+\varepsilon)^{-1}].
\end{align*}
Since all the quantities in \eqref{DiffEq} is analytic in $\varepsilon$, it indeed holds for all $\varepsilon$ in the right half plane of the complex plane; in particular, it holds for all $\varepsilon>0$. 
\end{proof}
\begin{proposition}
	\label{AddPDE}
	For each $\lambda\in\C$, the function $S(\cdot, \lambda, \cdot)$ satisfies the first-order nonlinear partial differential equation
	$$\frac{\partial S}{\partial t} = \varepsilon\left(\frac{\partial S}{\partial \varepsilon}\right)^2$$
	with initial condition
	\[S(0, \lambda, \varepsilon) = \tau(\log((x_0-\lambda)^*(x_0-\lambda)+\varepsilon)).\]
	The PDE does not depend on derivative of the real or imaginary parts of $\lambda$. The variables for the PDE are only $t$ and $\varepsilon$.
\end{proposition}
\begin{proof}
	Applying the fact (\cite[Lemma 1.1]{Brown1986}) that
	$$\frac{\partial S}{\partial \varepsilon} = \tau[(x_{t,\lambda}^*x_{t,\lambda}+\varepsilon)^{-1}]$$
	to \eqref{DiffEq} concludes the proof. 
\end{proof}
The equation in Proposition \ref{AddPDE} is a first-order, nonlinear PDE of
Hamilton--Jacobi type (see for example, Section 3.3 in
the book of Evans \cite{EvansBook}). We will use Hamilton--Jacobi method to \emph{analyze} the function $S$. (The solution, which is the function $S$ defined in \eqref{SEq}, already exists.) The Hamiltonian function 
\begin{equation}
\label{AddHamiltonian}
H(\varepsilon, p_\varepsilon) = -\varepsilon p_\varepsilon^2
\end{equation}
 satisfying
\[\frac{\partial S}{\partial t} = -H\left(\varepsilon, \frac{\partial S}{\partial \varepsilon}\right)\]
 is obtained by replacing the partial derivative in Proposition \ref{AddPDE} by the momentum variable $p_\varepsilon$ and adding a minus sign.

\subsubsection{Solving the Differential Equations}
We consider the Hamilton's equations for the Hamiltonian \eqref{AddHamiltonian}, which consists of the following system of two coupled ODE's:
\begin{equation}
\label{ODEs}
	\frac{d\varepsilon}{dt}=\frac{\partial H}{\partial p_\varepsilon},\qquad\frac{dp_\varepsilon}{dt}=-\frac{\partial H}{\partial \varepsilon}.
\end{equation}
To apply the Hamilton--Jacobi method, we take arbitrary initial condition $\varepsilon_0>0$ for $\varepsilon$ but choose initial condition $p_0$ for the momentum variable $p_\varepsilon$ as
\[p_{0}=p_{\varepsilon}(0) = \tau((|\lambda-x_0|^2+\varepsilon_0)^{-1}).\]
The initial momentum $p_0$ can be written as an integral as
\begin{equation}
\label{Initialp.Add}
p_{0} = \int\frac{1}{|\lambda-x|^2+\varepsilon_0}\,d\mu(x),
\end{equation}
where $\mu$ is the spectral distribution of $x_0$. 
The momentum $p_0$ cannot be chosen arbitrarily and it depends on the initial condition $\varepsilon_0$, as seen in the above formula. The Hamiltonian is a constant of motion and can be expressed as
\begin{equation}\label{eqn:Hamiltonian-additive}
  H(\varepsilon, p_\varepsilon)=H(\varepsilon_0, p_0)=-\varepsilon_0 p_0^2. 
\end{equation}

We first state the result from \cite[Proposition 6.3]{DHKBrown}.
\begin{proposition}
	\label{GeneralSol}
	Fix a function $H({\bf x}, {\bf p})$ defined for ${\bf x}$ in an open set $U\subseteq \R^n$ and ${\bf p}$ in $\R^n$. Consider a smooth function $S(t, {\bf x})$ on $[0, \infty)\times U$ satisfying 
	$$\frac{\partial S}{\partial t} = -H({\bf x}, \nabla_{\bf x}S).$$
	Suppose the pair $({\bf x}(t), {\bf p}(t))$ with values in $U\times \R^n$ satisfies the Hamilton's equations
	$$\frac{d x_j}{dt} = \frac{\partial H}{\partial p_j}({\bf x}(t), {\bf p}(t));\quad \frac{dp_j}{dt} = -\frac{\partial H}{\partial x_j}({\bf x}(t), {\bf p}(t))$$
	with initial conditions ${\bf x}(0) = {\bf x}_0$ and ${\bf p}(0) = (\nabla_{\bf x}S)(0, {\bf x}_0)$.
	Then we have
	\begin{equation}
	\label{FirstHJ}
	S(t, {\bf x}(t)) = S(0, {\bf x}_0)-H({\bf x}_0, {\bf p}_0)t+\int_0^t {\bf p}(s)\cdot \frac{d{\bf x}}{ds}\,ds
	\end{equation}
	and \begin{equation}\label{eq:2ndHJ}
	(\nabla_{\bf x}S)(t, {\bf x}(t))={\bf p}(t).
	\end{equation}
	These two formulas are valid only as long as the solution curve $({\bf x}(t), {\bf p}(t))$ exists in $U\times\R^n$.
\end{proposition}
We apply this result with $n=1$ and $U = (0,\infty)$ in the system \eqref{ODEs}.
\begin{proposition} In the system of coupled ODEs \eqref{ODEs} with $\varepsilon(0) = \varepsilon_0$ and $p_\varepsilon(0) = p_0$ given in \eqref{Initialp.Add}, for any $\lambda$, the first Hamilton--Jacobi formula \eqref{FirstHJ} reads
	\label{FormulaS}
	$$S(t, \lambda, \varepsilon(t)) = \tau(\log(|\lambda-x_0|^2+\varepsilon_0))-\varepsilon_0\tau((|\lambda-x_0|^2+\varepsilon_0)^{-1})^2t.$$
	This formula is valid only as long as the solution $(\varepsilon(t),p_\varepsilon(t))$ exists in $(0,\infty)\times \R$; the lifetime of the solution depends on the initial condition $\varepsilon_0$ of $\varepsilon$.
\end{proposition}
\begin{proof}
	Since the Hamiltonian $H$ is given by (\ref{eqn:Hamiltonian-additive}),  by (\ref{ODEs}), we have 
	\begin{align*}
	 p_\varepsilon(s)\cdot\frac{d}{ds}\varepsilon(s)& =-2\varepsilon(s) p_\varepsilon(s)^2 
	    = 2H(\varepsilon, p_\varepsilon(s))\\
	    &=2H(\varepsilon_0, p_0)=-2\varepsilon_0\tau((|\lambda-x_0|^2+\varepsilon_0)^{-1})^2.
	\end{align*}
The expression of this proposition then follows from \eqref{FirstHJ}.
\end{proof}

\begin{proposition}
	\label{pSol}
	Let $p_0 = \tau((|\lambda-x_0|^2+\varepsilon_0)^{-1})$ be the initial value of $p_\varepsilon$. If the solution to the Hamiltonian system exists up to time $t_*$, then for all $t\in (0, t_*)$, the solution of $p_\varepsilon$ is
	\begin{equation}\label{eqn:p_varepsilon}
p_\varepsilon(t) = \frac{1}{\frac{1}{p_0}-t}.	
	\end{equation}
\end{proposition}
\begin{proof}
	This follows directly from solving one equation in the system of ODEs (\ref{ODEs}) 
	$$\frac{dp_\varepsilon}{dt} = -\frac{\partial H}{\partial \varepsilon} = p_\varepsilon^2$$
	with the initial value $p_0 = \tau((|\lambda-x_0|^2+\varepsilon_0)^{-1})$.
\end{proof}
Since $\varepsilon(t) = -H(t)/p_\varepsilon(t)^2=-H_0/p_\varepsilon(t)^2$, we have the following corollary.
\begin{corollary}
	\label{epsilonFunc}
	Under the same hypothesis in Proposition \ref{pSol}, we can solve $\varepsilon(t)$ as
	$$\varepsilon(t) = -\left(\frac{1}{p_0}-t\right)^2 H_0 = \varepsilon_0(1-tp_0)^2.$$
\end{corollary}

The pair $(\varepsilon(t), p_\varepsilon(t))$ given by Proposition \ref{pSol} and Corollary \ref{epsilonFunc} gives a solution to the Hamiltonian system \eqref{ODEs} up to time 
\[t_*(\lambda, \varepsilon_0) = \frac{1}{p_0},\]
when the denominator in (\ref{eqn:p_varepsilon}) blows up. 
The value of $p_0$ given in (\ref{Initialp.Add}) and hence
the lifetime $t_*$ of the solution depends on $\lambda$ and the initial condition $\varepsilon_0$. 

\subsubsection{Two regimes and the connection to the subordination map}
To compute the Brown measure $\rho_t$ of $x_0+c_t$, we want to compute $\Delta S(t, \lambda, 0)$. Thus, we want to take $\varepsilon(t)=0$ for the formula of $S(t,\lambda, \varepsilon(t))$ in Proposition \ref{FormulaS}. By Corollary \ref{epsilonFunc}, the condition $\varepsilon(t) =0$ can be achieved by either letting $\varepsilon_0\downarrow
 0$, or considering $(1-tp_0) \downarrow
 0$ by making a suitable choice of the initial condition $\varepsilon_0$. 

Write $\lambda= a+ib$. As $\varepsilon_0$ decreases, the lifetime of the path
\begin{equation*}
\label{lifetimedef}
t_*(\lambda, \varepsilon_0)= \frac{1}{p_0} =\frac{1}{\int_{\mathbb{R}}\frac{d\mu(x)}{(a-x)^2+b^2+\varepsilon_0}}
\end{equation*}
decreases. Let
\begin{equation}
\label{lifetimeT}
T(\lambda) = \lim_{\varepsilon_0\downarrow
 0}t_*(\lambda, \varepsilon_0)=  \frac{1}{\int_{\mathbb{R}}\frac{d\mu(x)}{(a-x)^2+b^2}} 
\end{equation}
be the lifetime of the path in the limit as $\varepsilon_0\downarrow
 0$, where the second equality comes from an application of the Monotone Convergence theorem. 
 The first regime, $\varepsilon_0\downarrow
 0$, only works if the lifetime of the path is at least $t$ when $\varepsilon_0$ is very small. 
Hence we study the set where the lifetime of the path is at least $t$ as follows 
\[\{\lambda\in\C: T(\lambda)\geq t\}=\left\{\lambda=a+ib: \int_\mathbb{R}\frac{d\mu(x)}{(a-x)^2+b^2}\leq\frac{1}{t} \right\}.\]
 The second regime is the open set
\begin{equation}\label{eqn:Lambada-t-T}
\Lt = \{\lambda\in\C: T(\lambda)<t\}=\left\{\lambda=a+ib: \int_\mathbb{R}\frac{d\mu(x)}{(a-x)^2+b^2}>\frac{1}{t} \right\},
\end{equation}
where we can use letting $(1-tp_0) \downarrow
 0$, by taking a proper limit of the initial condition $\varepsilon_0$ in Section \ref{InsideLt}. 

We note that the first regime, letting $\varepsilon_0\downarrow 0$ will be used in Section \ref{OutsideLtC} that can be described using the subordination function $F_t$ in Proposition \ref{HFMap} and the function $v_t$ defined in \eqref{def:v_t}. 
We note that $T$ may not be continuous, depending on the choice of $\mu$. For example, if $\mu = \1_{[0,1]}(x) \cdot 3x^2\,dx$, then $T(0) = \frac{1}{3}$ but $T(a)=0$ for all $0<a<1$.

We now identify $\Lt$ and draw connection between the definition of $\Lt$ and the free additive convolution discussed in Section \ref{FAConv}. Recall from \eqref{def:v_t} that 
\begin{equation}
v_t(a) = \inf\left\{b>0: \int_\R\frac{d\mu(x)}{(a-x)^2+b^2}\leq \frac{1}{t}\right\}, \quad a\in\R.
\end{equation}
Using \eqref{lifetimeT}, by (\ref{eqn:Lambada-t-T}), we see that
\begin{equation}\label{Lambada-t-v-t}
\Lt = \{a+ib\in\C: |b|< v_t(a)\}.
\end{equation}
Thus, $\Lt$ agrees with the region defined in Point 1 of Theorem \ref{IntroTh1}. We will prove that the Brown measure $\rho_t$ has support inside $\LtC$. 

We recall, from Section \ref{FAConv} that, the subordination function $F_t$ satisfying
\[G_{x_0}(F_t(z)) = G_{x_0+s_t}(z), \qquad z\in\C^+\]
has left inverse 
\[H_t(z) = z+tG_{x_0}(z).\]
From Proposition \ref{HFMap}, the function $F_t$ can be continuously extended to $\overline{\C^+}$ (\cite[Lemmas 3 and 4]{Biane1997}) and
\[
   F_t(\mathbb{C}^+\cup \mathbb{R})=\{ a+ib\in \mathbb{C}^+: b\geq v_t(a) \}.
\]
 Its inverse $H_t$ maps
\[\{a+ib\in \C^+: b\geq v_t(a)\}\]
onto $\C^+\cup \mathbb{R}$. Due to (\ref{Lambada-t-v-t}), the set $\Lt$ can be written in terms of $F_t$ by
\[\Lt = \C\setminus\{F_t(z), \overline{F_t(z)}: z\in \mathbb{C}^+\cup \mathbb{R} \}.\]
The boundary of $\Lt$ is given by the closure of the set
$$\{ z, \overline{z}: z\in F_t(\mathbb{R})\cap \mathbb{C}^+ \}=\{z, \overline{z}: z \in \mathbb{C}^+ \textrm{ and } H_t(z) \in \mathbb{R}\}.$$
It is notable that the subordination function $F_t$ for $x_0+s_t$ also plays a role in the analysis $x_0+c_t$, where $c_t$ is the free circular Brownian motion.

\subsection{Computation of the Laplacian of $S$}
We now outline the strategy for the computation of the Brown measure in different regimes as follows. (i) If $T(\lambda)>t$, then the lifetime of the solution to Hamiltonian system remains greater than $t$ when $\varepsilon_0$ tends to zero. Hence, we could let $\varepsilon_0=0$ in \eqref{FormulaS}. (ii) If $T(\lambda)<t$, the value of $\varepsilon_0$ is chosen such that $p_0=1/t$. This condition is equivalent to $b^2+\varepsilon_0=v_t(a)^2$ as in Lemma \ref{DeriEpsilon} which allows us to calculate the Brown measure in this regime. (iii) Finally, we need to eliminate any mass on the set $\{\lambda\in\C: T(\lambda) = t\}$. We do this by showing that the \emph{restriction} of the Brown measure $\rho_t$ to $\Lt$ is a probability measure; therefore, using the fact that the Brown measure is a probability measure, we conclude that there is no mass on the boundary $\{\lambda\in\C: T(\lambda) = t\}$. Our approach for the boundary $\{\lambda\in\C: T(\lambda) = t\}$ is different from the approach for the boundary in \cite{DHKBrown}, where they showed directly the Brown measure of the free multiplicative Brownian motion has no mass on the boundary by showing that the partial derivatives of the corresponding $S$ are continuous.


\subsubsection{Outside $\LtC$}
\label{OutsideLtC}
By Corollary \ref{epsilonFunc}, 
$$\varepsilon(t) = \varepsilon_0\left(1-tp_0\right)^2$$
depends on the initial condition $\varepsilon_0>0$. In the case when $\lambda\not\in\LtC$, we use the regime of taking $\varepsilon_0\downarrow
 0$ to make $\varepsilon(t)\downarrow
 0$.
When $\lambda\not\in \LtC$ is fixed, by (\ref{Lambada-t-v-t}), we have 
$$p_0 = \int_{\mathbb{R}}\frac{d\mu(x)}{(a-x)^2+b^2+\varepsilon_0}\leq \int_{\mathbb{R}}\frac{d\mu(x)}{(a-x)^2+b^2}<\int_{\mathbb{R}}\frac{d\mu(x)}{(a-x)^2+v_t(a)^2}\leq\frac{1}{t},$$
where the definition (\ref{def:v_t}) of $v_t$ was used.
In this case, by \eqref{lifetimeT}, the lifetime of the solution 
\[t_*(\lambda, \varepsilon)>t, \quad\textrm{for all $\varepsilon_0>0$}.\]
The initial momentum $p_0$ can then be extended continuously at $\varepsilon_0=0$ by Monotone Convergence Theorem, and $1-tp_0\geq 0$, for all $\varepsilon_0\geq 0$. It is clear that for a fixed $\lambda=a+ib$ such that $|b|>v_t(a)$, $p_0$ is a decreasing function of $\varepsilon_0\geq0$ and hence $1-tp_0$ is increasing. This proves the following lemma.
\begin{lemma}
	\label{AddSurj}
	For a fixed $\lambda\not\in\LtC$, $\varepsilon=\varepsilon_0(1-tp_0)^2$ is an increasing continuous function of $\varepsilon_0>0$; $\varepsilon$ can be extended continuously as $\varepsilon_0\downarrow
 0$.  
We have 
\[
    \lim_{\varepsilon_0 \downarrow
 0}\varepsilon=0, \quad\text{and}\quad \lim_{\varepsilon_0\rightarrow \infty}\varepsilon=\infty. 
\]
	Thus, for every $\varepsilon\geq 0$, there is a unique $\varepsilon_0$ such that $\varepsilon = \varepsilon_0(1-tp_0)^2$.
\end{lemma}

Fix $t$ and $\lambda$. As $\varepsilon$ is an increasing function of $\varepsilon_0$, we can express $\varepsilon_0$ as a function $\varepsilon_0(\varepsilon)$ of $\varepsilon$ such that with initial condition $\varepsilon_0(\varepsilon)$, we have $\varepsilon(t) = \varepsilon$. Hence, 
by Proposition \ref{FormulaS},
\begin{align}
S(t, \lambda, \varepsilon(t))& = \tau(\log(|\lambda-x_0|^2+\varepsilon_0))-\varepsilon_0\tau((|\lambda-x_0|^2+\varepsilon_0)^{-1})^2t\nonumber\\
&=\tau(\log(|\lambda-x_0|^2+\varepsilon_0(\varepsilon)))-\varepsilon_0(\varepsilon)\tau((|\lambda-x_0|^2+\varepsilon_0(\varepsilon))^{-1})^2t.\label{S_tformula}
\end{align}
Using also the fact that $\varepsilon_0(\varepsilon)\downarrow
 0$ as $\varepsilon\downarrow
 0$ gives the main result of this section.
\begin{theorem}
	\label{ThOut}
	For a fixed $\lambda\not\in\LtC$, 
	\begin{equation}
	\label{St0}
	S(t,\lambda,0) = \lim_{\varepsilon\downarrow
 0}S(t, \lambda, \varepsilon) = \int_{\mathbb{R}} \log((a-x)^2+b^2)\,d\mu(x)
	\end{equation}
	is real-valued.	Thus, for all $\lambda\not\in\LtC$, $S(t,\lambda, 0)$ is analytic and
	$$\Delta_\lambda S(t,\lambda, 0) = 0.$$
	In particular, the support of the Brown measure $\rho_t$ of $x_0+c_t$ is inside $\LtC$.
\end{theorem}
\begin{proof} 

We first prove that the right hand side of~\eqref{St0} exists as a (finite) real number; it suffices to consider $\lambda\in\bar{\Lambda}_t^c\cap\R$. Let $I\subset\R$ be an interval that is outside the closed set $\bar{\Lambda}_t$. Then, for any $a\in I$, we have
\[0\leq -\frac{1}{\pi}\mathrm{Im}\,[G_{x_0}(a+i\varepsilon)] =\frac{1}{\pi} \int\frac{\varepsilon}{(a-x)^2+\varepsilon^2}\,d\mu(x)\leq \frac{\varepsilon}{\pi t}.\]
Letting $\varepsilon\downarrow 0$ in the above equation gives that $\mu$ is $0$ in $\bar{\Lambda}_t^c\cap \R$, since the measure $-\1_{I}\frac{1}{\pi}\mathrm{Im}\,[G_{x_0}(a+i\varepsilon)]\,da$ converges weakly to $\left.\mu\right|_{I}$ as $\varepsilon\downarrow 0$ for any arbitrary interval $I\subset \bar{\Lambda}_t\cap\R$ by the Stieltjes inversion formula~\eqref{CauchyInv}.
	
	By letting $\varepsilon\downarrow 0$ in~\eqref{S_tformula}, the fact that $\varepsilon_0(\varepsilon)\downarrow 0$ as $\varepsilon\downarrow0$ gives~\eqref{St0}. Now, since, for all $\lambda\not\in\LtC$, $(a-x)^2+b^2>0$, we can move the Laplacian into the integral
	\[\Delta_\lambda S(t,\lambda,0) = \int_\R\Delta_\lambda \log((a-x)^2+b^2)\,d\mu(x)=0.\]	
	Recall that the Brown measure $\rho_t$ is defined to be the \emph{distributional} Laplacian of the function $S(t,\lambda,0)$ with respect to $\lambda$. The last assertion about the support of the Brown measure $\rho_t$ follows from the fact that 
	\[\Delta_\lambda S(t,\lambda, 0) = 0\]
	outside $\LtC$ imples that the support of the \emph{distribution} $\Delta_\lambda S(t, \lambda,0)$ is inside $\LtC$.
\end{proof}

\subsubsection{Inside $\Lt$}
\label{InsideLt}
By \eqref{eqn:Lambada-t-T}, if $\lambda\in \Lt$,  we have $T(\lambda)<t$ (see the definition \eqref{lifetimeT} of $T(\lambda)$), showing that
$$\int_{\mathbb{R}} \frac{d\mu(x)}{|\lambda-x|^2}=\int_{\mathbb{R}} \frac{d\mu(x)}{(a-x)^2+b^2}>\frac{1}{t},\quad \lambda\in\Lt.$$
By the fact that 
\[\lim_{\varepsilon_0\downarrow
 0}t_*(\lambda, \varepsilon_0) = T(\lambda),\]
for each $\lambda\in \Lt$, if the initial condition $\varepsilon_0>0$ is small enough, 
the lifetime of the solution path $t_*(\lambda, \varepsilon_0)<t$. That means 
the solution does not exist up to time $t$.
Since, given any $\lambda$, the function  $\varepsilon_0\mapsto t_*(\lambda, \varepsilon_0\,)$ is strictly increasing, and $\lim_{\varepsilon_0\uparrow
 \infty}t_*(\lambda,\varepsilon_0)=\infty$, there exists a unique $\varepsilon_0(\lambda)$ such that $t_*(\lambda, \varepsilon_0(\lambda)) = t$. In other words, with this choice of $\varepsilon_0 = \varepsilon_0(\lambda)$, for $\lambda=a+ib$, we have, since $t_*(\lambda, \varepsilon_0(\lambda)) =t$,
\begin{equation}
	\label{IntConst}
	\int_{\mathbb{R}}\frac{d\mu(x)}{(a-x)^2+b^2+\varepsilon_0(\lambda)}=\frac{1}{t_*(\lambda, \varepsilon_0(\lambda))} = \frac{1}{t}.
\end{equation}
\begin{lemma}
	\label{DeriEpsilon}
	With the choice of $\varepsilon_0=\varepsilon_0(\lambda)$ satisfying (\ref{IntConst}), $\varepsilon_0$ is a function of $(a,b)$ and is determined by
	\begin{equation}\label{eqn:varepsilon_0_lambda}
	b^2+\varepsilon_0(\lambda)=v_t(a)^2.
	\end{equation}
\end{lemma}
\begin{proof}
	This follows directly from the fact that such a $\varepsilon_0(\lambda)$ is unique. The $\varepsilon_0(\lambda)$ such that $b^2+\varepsilon_0(\lambda)=v_t(a)^2$ satisfies \eqref{IntConst} by \eqref{eq:identity-v-t-a}.
\end{proof}

If $t_*(\lambda, \varepsilon_0(\lambda)) = t$, then by $\eqref{eqn:varepsilon_0_lambda}$ and Corollary \ref{epsilonFunc}, $\varepsilon(t) = 0$. This means that the solution of the system of ODEs \eqref{ODEs} does not exist at time $t$. For all $\varepsilon_0>\varepsilon_0(\lambda)$, the lifetime of the system \eqref{ODEs} is greater than $t$; hence it makes sense to look at $S(t, \lambda, \varepsilon(t))$, with any initial condition $\varepsilon_0>\varepsilon_0(\lambda)$. Our strategy is to consider the solution of $S$ by solving the system of ODEs \eqref{ODEs} at time $t$ as the limit as $\varepsilon_0\downarrow \varepsilon_0(\lambda)$ so that $\varepsilon(t)\downarrow 0$.

Recall that \[S(t,\lambda,0) = \lim_{\varepsilon\downarrow
	0} S(t, \lambda, \varepsilon)\] is defined as in \eqref{St0def}. We now compute $\Delta_\lambda S(t, \lambda, 0)$.

\begin{theorem}
	\label{ThIn}
	For $\lambda=a+ib\in \Lt$, with the choice $b^2+\varepsilon_0(\lambda)= v_t(a)^2$, using Proposition \ref{FormulaS}, we have
	$$S(t,\lambda,0) = \lim_{\varepsilon\downarrow
 0} S(t, \lambda, \varepsilon) = \int_{\mathbb{R}} \log((a-x)^2+v_t(a)^2)\,d\mu(x) -\frac{v_t(a)^2-b^2}{t}.$$
	In particular, $S(t, \lambda, 0)$ is analytic inside $\Lt$. 
	
	The Laplacian of $S(t, \lambda, 0)$ with respect to $\lambda$ is computed as
	$$\Delta_\lambda S(t, \lambda, 0)=\frac{4}{t}\left(1-\frac{t}{2}\frac{d}{d a}\int_{\mathbb{R}}\frac{x}{(a-x)^2+v_t(a)^2}\,d\mu(x)\right).$$
	In particular, for $\lambda\in\Lambda_t$, $\Delta_\lambda S(t, \lambda, 0)$ is independent of $b=\text{\emph{Im}}\:\lambda$.
\end{theorem}
\begin{proof}
If $\lambda\in\Lambda_t$ and $\varepsilon>\varepsilon_0(\lambda)$, the lifetime $t_*(\lambda,\varepsilon_0)$ is greater than $t$. Using Proposition \ref{FormulaS}, we can compute $S(t, \lambda, \varepsilon(t))$, for each initial condition $\varepsilon_0>\varepsilon_0(\lambda)$, in terms of the initial condition $\varepsilon_0$ explicitly.  By~\eqref{S_tformula}, at time $t$,
\begin{equation}
\label{St0HJ}
\begin{split}
S(t,\lambda,\varepsilon(t)) =& \tau(\log((a-x_0)^2+b^2+\varepsilon_0))-\varepsilon_0\tau(((a-x_0)^2+b^2+\varepsilon_0)^{-1})^2t\\
=&\int_\R \log((a-x)^2+b^2+\varepsilon_0)\,d\mu(x)-\varepsilon_0t\left(\int_\R\frac{d\mu(x)}{(a-x)^2+b^2+\varepsilon_0}\right)^2.\
\end{split}
\end{equation}
Let $\varepsilon_0\downarrow
 \varepsilon_0(\lambda)=v_t(a)^2-b^2$ as in \eqref{eqn:varepsilon_0_lambda}.
Then, in the limit, we have 
\[
  \lim_{\varepsilon_0\downarrow \varepsilon_0(\lambda)}\varepsilon(t)=\lim_{\varepsilon_0\downarrow \varepsilon_0(\lambda)}\varepsilon_0\left(1-tp_0\right)^2=0,
\]
as 
\[
\lim_{\varepsilon_0\downarrow \varepsilon_0(\lambda)} tp_0=t\int_{\mathbb{R}}\frac{d\mu(x)}{(a-x)^2+v_t(a)^2}=1. 
\]
Therefore, using \eqref{St0HJ}, we have 
\begin{equation}
\label{St0HJcomputed}
\begin{split}
S(t,\lambda,0) =& \lim_{\varepsilon\downarrow 0} S(t,\lambda,\varepsilon)\\
=& \lim_{\varepsilon_0\downarrow\varepsilon_0(\lambda)}\left[\int_\R \log((a-x)^2+b^2+\varepsilon_0)\,d\mu(x)-\varepsilon_0t\left(\int_\R\frac{d\mu(x)}{(a-x)^2+b^2+\varepsilon_0}\right)^2\right]\\ =&\int_\R\log((a-x)^2+v_t(a)^2)\,d\mu(x)-(v_t(a)^2-b^2)t\left(\int_\R\frac{d\mu(x)}{(a-x)^2+v_t(a)^2}\right)^2\\
=&\int_\R\log((a-x)^2+v_t(a)^2)\,d\mu(x)+\frac{b^2-v_t(a)^2}{t}
\end{split}
\end{equation}
where the last equality comes from Lemma \ref{lemma:identity-vt}. 

Now we compute the derivatives of $S(t,\lambda, 0)$ using \eqref{St0HJcomputed}. The partial derivative with respect to $b$ is
\begin{equation}
\label{Da}
	\frac{\partial S}{\partial b}(t,\lambda,0) = \frac{2b}{t}.
\end{equation}
We note that, by \cite[Lemma 2]{Biane1997}, $v_t$ is analytic at any point $a$ such that $v_t(a)>0$. If $\lambda\in \Lt$, $v_t(a)>0$ and the partial derivative with respect to $a$ is
\begin{equation}
\label{Db}
\begin{split}
\frac{\partial S}{\partial a}=&\int_\R \frac{1}{(a-x)^2+v_t(a)^2}\left(2(a-x)+\frac{\partial}{\partial a}v_t(a)^2\right)d\mu(x)-\frac{1}{t}\frac{\partial}{\partial a}v_t(a)^2\\
=&\left(2a+\frac{\partial}{\partial a}v_t(a)^2\right)\int_\R\frac{d\mu(x)}{(a-x)^2+v_t(a)^2}-\int \frac{2x\,d\mu(x)}{(a-x)^2+v_t(a)^2}-\frac{1}{t}\frac{\partial }{\partial a}v_t(a)^2\\
=&\frac{2a}{t}-\int_\R \frac{2x\,d\mu(x)}{(a-x)^2+v_t(a)^2},
\end{split}
\end{equation}
where the identity in Lemma \ref{IneqT} was used. 
Using \eqref{Da} and \eqref{Db}, we can compute the Laplacian
\begin{align*}
\Delta_\lambda S(t,\lambda,0) =& \left( \frac{\partial^2 S}{\partial a^2}+\frac{\partial^2 S}{\partial b^2}  \right)(t,\lambda,0) \\
=& \frac{2}{t}+\frac{2}{t}-2\frac{d}{da}\int_\R\frac{x\,d\mu(x)}{(a-x)^2+v_t(a)^2}\\
=&\frac{4}{t}\left(1-\frac{t}{2}\frac{d}{da}\int_\R \frac{x\,d\mu(x)}{(a-x)^2+v_t(a)^2}\right).
\end{align*}
This finishes the proof. 
\end{proof}

\subsection{The Brown measure of $x_0+c_t$}
\begin{lemma}
	\label{ProbM}
	The integral 
	$$\frac{1}{4\pi}\int_{\Lt}\Delta_\lambda  S(t, \lambda, 0)\,da\,db = 1;$$
	hence, $\frac{1}{4\pi}\Delta_\lambda  S(t, \lambda, 0)\,da\,db$ defines a probability measure on $\Lt$.
\end{lemma}
\begin{proof}
	Recall that, by Lemma \ref{AddDensity}, 
	\begin{equation}
	\label{Recallpsi}
	\psi_t(a) = a+t\int_\R\frac{(a-x)\,d\mu(x)}{(a-x)^2+v_t(a)^2}.
	\end{equation}
	This is a computation using Fubini's theorem. Define 
	$$V_t = \{a\in\R: v_t(a)>0\}.$$
	 By Theorem \ref{ThIn},
		\begin{align*}
		\frac{1}{4\pi}\int_{\Lt} \Delta_\lambda  S(t, \lambda, 0)\, da\,db=&\iint_{\Lt} \frac{1}{\pi t}\left(1-\frac{t}{2}\frac{d}{d a}\int_{\mathbb{R}}\frac{x}{(a-x)^2+v_t(a)^2}\,d\mu(x)\right)db\,da \\
		=& \int_{V_t}\frac{2v_t(a)}{\pi t}\left(1-\frac{t}{2}\frac{d}{d a}\int_\R\frac{x}{(a-x)^2+v_t(a)^2}\,d\mu(x)\right)da\;\; \textrm{(by definition of $\Lt$)}\\
		=& \int_{V_t} \frac{v_t(a)}{\pi t}\frac{d}{da}\left(2a-t\int_\R \frac{x}{(a-x)^2+v_t(a)^2}\,d\mu(x)\right)da\\
		=& \int_{V_t} \frac{v_t(a)}{\pi t}\frac{d}{da}\left(a+t\int_\R \frac{a-x}{(a-x)^2+v_t(a)^2}\,d\mu(x)\right)da\qquad\textrm{(by Lemma \ref{IneqT})}\\
		=& \int_{V_t} \frac{v_t(a)}{\pi t}\, d\psi_t(a)\qquad\textrm{by \eqref{Recallpsi}}\\
		=& \int_\R \,d(\mu\boxplus \sigma_t) = 1.
	\end{align*}
	The last equality follows from Proposition \ref{AddDensity}.
\end{proof}

Since the Brown measure is a probability measure, using the definition of the Brown measure, the above lemma shows that the Brown measure $\rho_t$ of $x_0+c_t$ is supported on $\overline{\Lt}$ and is absolutely continuous on the support.
\begin{corollary}
	\label{BrownForm}
	By Lemma \ref{ProbM}, we have
	\[
	  \int_{\Lambda_t}d\rho_t=1.
	\]
	Hence, the Brown measure $\rho_t$ has the form
	$$d\rho_t(\lambda)=\frac{1}{4\pi}\1_{\Lt}\Delta_\lambda  S(t, \lambda, 0)\,da\,db.$$
	\end{corollary}

\begin{theorem}
	\label{ThAdd}
	The Brown measure $\rho_t$ of $x_0+c_t$ has support $\LtC$ and has the form 
	\begin{equation}
	\label{AddBrown}
	d\rho_t (a+ib) =\1_{\Lambda_t}(a+ib)w_t(a)\,db\,da
	\end{equation}
	where
	\begin{equation*}
	w_t(a) = \frac{1}{\pi t}\left(1-\frac{t}{2}\frac{d}{d a}\int_\R\frac{x}{(a-x)^2+v_t(a)^2}\,d\mu(x)\right).
	\end{equation*}
	The function $w_t$ is strictly positive inside $\Lt$. 
	
	Let $\Psi(\lambda) = H_t(a+iv_t(a))$, where $\lambda = a+ib\in \Lt$ (see Proposition \ref{HFMap} for the definition of $H_t$). Then the push-forward of the Brown measure to the real line by $\Psi$ is $\mu\boxplus \sigma_t$, where $\mu$ is the distribution of $x_0$. 
\end{theorem}

When $\mu=\delta_0$, one can verify that $v_t(a) = \1_{a^2\leq t}\sqrt{t-a^2}$ and $w_t(a) = 1/(\pi t)$. Then the Brown measure is, as expected, the circular law.
\begin{proof}
	First note that \eqref{AddBrown} comes from Corollary \ref{BrownForm} and Theorem \ref{ThIn}. Define 
	$$V_t = \{a\in\R: v_t(a)>0\}.$$
	For the assertion about the push-forward, given any smooth function $g$, we have
	\begin{align*}
		&\iint_{\Lt} g\circ \Psi(a,b) \frac{1}{\pi t}\left(1-\frac{t}{2}\frac{d}{d a}\int_\R\frac{x\,d\mu(x)}{(a-x)^2+v_t(a)^2}\right)\,db\,da\\
		=& \int_{V_t} g(\psi(a))\frac{2v_t(a)}{\pi t}\left(1-\frac{t}{2}\frac{d}{d a}\int_\R\frac{x\,d\mu(x)}{(a-x)^2+v_t(a)^2}\right)da\\
		=& \int_{V_t} g(\psi(a))\frac{v_t(a)}{\pi t}\frac{d}{da}\left(2a-t\int_\R \frac{x\,d\mu(x)}{(a-x)^2+v_t(a)^2}\right)da\\
		=& \int_{V_t} g(\psi(a))\frac{v_t(a)}{\pi t}\frac{d}{da}\left(a+t\int_\R \frac{(a-x)\,d\mu(x)}{(a-x)^2+v_t(a)^2}\right)da\qquad\textrm{by Lemma \ref{IneqT}}\\
		=& \int_{V_t} g(\psi(a)) \frac{v_t(a)}{\pi t} \,d\psi_t(a)\qquad\textrm{by \eqref{Recallpsi}}\\
		=& \int_\R g\,d(\mu\boxplus \sigma_t).
	\end{align*}
	The last equality follows from Proposition \ref{AddDensity}. The above computation also shows that when $(a,b)\in\Lt$, we have
	$$\frac{1}{\pi t}\left(1-\frac{t}{2}\frac{d}{d a}\int_\R\frac{x}{(a-x)^2+v_t(a)^2}\,d\mu(x)\right) = \frac{1}{\pi t} \psi_t'(a)\geq \frac{2}{t^2}v_t(a)^2(1+v_t'(a)^2)>0$$
	by Proposition \ref{AddDensity}. This shows that the density of $\rho_t$ is strictly positive inside $\Lt$.
	\end{proof}

\begin{theorem}
	\label{AddBound}
The density of the Brown measure of $x_0+c_t$ may be expressed as 
\begin{equation}\label{eqn:w_t:additive-derivative}
w_t(a)=\frac{1}{2\pi t}\frac{d\psi_t(a)}{da},
\end{equation}
and is bounded by
\begin{equation}\label{w_t:bound-additive}
 w_t(a)\leq \frac{1}{\pi t}. 
\end{equation}
The inequality is strict unless $\mu$ is a Dirac mass.
\end{theorem}
\begin{proof}
The formula \eqref{eqn:w_t:additive-derivative} was already obtained in the proof of Theorem \ref{ThAdd}; Theorem 4.6 in \cite{BB2005} can be used to prove the inequality \eqref{w_t:bound-additive} but we give a direct proof here.

Recall that  $\psi_t(a) = H_t(a+iv_t(a))$ and $H_t(z)=z+tG_{x_0}(z)$. Let $z = a+iv_t(a)$. If $v_t(a)>0$, we have
\begin{align}
	\left|1-H_t'(z)\right|&=\left| t\frac{d}{dz}\int_\mathbb{R}\frac{d\mu(x)}{z-x} \right|\nonumber\\
	   &=\left|t\int_{\mathbb{R}} \frac{d\mu(x)}{(z-x)^2}\right|\nonumber\\
	   &\leq t\int_\mathbb{R}\frac{d\mu(x)}{|z-x|^2}=1,\label{eqn:H_t<1}
\end{align}
where we used the definition of $v_t(a)$ in (\ref{def:v_t}). If $\mu$ is not a Dirac mass at $0$, the inequality in \eqref{eqn:H_t<1} must be strict, since $1/(z-x)^2$ does not have the same phase for $\mu$-almost every $x$. 

Since $\psi_t(a) = H_t(a+iv_t(a))$ is real-valued, by \eqref{eqn:H_t<1}, we must have
\[0\leq\frac{\psi_t(a)}{da}\leq2\]
and the inequalities on both sides are strict unless $\mu$ is a Dirac mass.
\end{proof}

\begin{corollary}
	The support of the Brown measure $\rho_t$ of $x_0+c_t$ is the closure of the open set $\Lambda_t$. The number of connected components of $\Lt$ of the support of $\rho_t$ is a non-increasing function $t$.
\end{corollary}
\begin{proof}
	The support of $\rho_t$ is the closure of $\Lambda_t$ follows directly by the facts that $\rho_t$ is strictly positive on $\Lambda_t$ and it has full measure on $\Lambda_t$.
	
	The number of connected components of the open set $V_t = \{a\in\R : v_t(a)>0\}$ is non-increasing, by Proposition~\ref{Acomponent}. Since $\Lambda_t$ has the same number of components as $V_t$ by the definition of $\Lambda_t$, the number of connected components of $\Lambda_t$ is also non-increasing.
	\end{proof}

\begin{proposition}
	\label{addunique}
	The Brown measure $\rho_t$ of $x_t=x_0+c_t$ is the unique measure $\sigma$ on $\LtC$ with the following two properties: $(1)$ the push-forward of $\sigma$ by the map $\lambda\mapsto\Psi_t(\lambda) = H_t(a+iv_t(a))$ is the distribution of $x_0+s_t$, and $(2)$ $\sigma$ is absolutely continuous with respect to the Lebesgue measure and its density is given by
	\[
	W(a,b)=h(a),
	\]
	for some continuous function $h$. 
\end{proposition}
\begin{proof}
	Define 
	$$V_t = \{a\in\R: v_t(a)>0\}.$$
	Let $y = \psi_t(a)$. By Theorem \ref{AddDensity}, the law of $x_0+s_t$ has the density given by 
	\[p_t(y) = \frac{v_t(a)}{\pi t}.\]
	We also recall, from Theorem \ref{AddDensity}, that $\psi_t'(a)>0$ for all $a\in V_t$. The push-forward by $\Psi$ of $\rho_t$ can be computed as, for any smooth function $g$,
	\begin{align*}
		\int_{\Lt} (g\circ \Psi)(a,b) h(a)\,da\,db = &\int_{V_t} (g\circ\psi_t)(a) \cdot 2v_t(a)h(a)\,da\\
		=& \int_{V_t} (g\circ \psi_t)(a)\frac{2v_t(a)h(a)}{\psi_t'(a)}\,dy
	\end{align*}
	By assumption, this means the measure $\frac{2v_t(a)h(a)}{\psi_t'(a)}dy$ is the law of $x_0+s_t$. It follows that
	\[\frac{2v_t(a)h(a)}{\psi_t'(a)} = \frac{v_t(a)}{\pi t}, \quad a\in V_t.\]
	Then $h$ can be solved explicitly as, using the definition of $\psi$ from Theorem \ref{AddDensity},
	\begin{align*}
		h(a) =& \frac{1}{2\pi t}\psi_t'(a)\\
		=& \frac{1}{2\pi t}\frac{d}{da} \left(a+t\int_{\R}\frac{a-x}{(a-x)^2+v_t(a)^2}\,d\mu(x)\right)\\
	=& \frac{1}{2\pi t}\frac{d}{da}\left(2a-t\int_\R\frac{x}{(a-x)^2+v_t(a)^2}\,d\mu(x)\right)\qquad\textrm{Using Lemma \ref{IneqT}}	\\
	=& \frac{1}{\pi t}\left(1-\frac{t}{2}\frac{d}{d a}\int_\R\frac{x}{(a-x)^2+v_t(a)^2}\,d\mu(x)\right)
		\end{align*}
		which is $w_t(a)$ in \eqref{AddBrown}. This establishes uniqueness.
	\end{proof}

Before we move on to the multiplicative case, we show some computer simulations. As indicated in Section \ref{AddIntro}, in this additive case, we can apply the result by \'Sniady \cite[Theorem 6]{Sniady2002}: if $X_N$ is a sequence of $N\times N$ self-adjoint deterministic matrix or random matrix classically independent of the Ginibre ensemble $Z_N(t)$ at time $t$, and if $x_0$ is the limit of $X_N$ in $\ast$-distribution, then the empirical eigenvalue distribution of $X_N+Z_N(t)$ converges weakly to the Brown measure of $x_0+c_t$.

The two figures use computer simulations to compare the push-forward of the Brown measure of $x_0+c_t$ under $\Psi$ and the spectral distribution of $x_0+s_t$. The latter distribution is stated in Proposition \ref{AddDensity}. In each of the figures, part (a) consists of the eigenvalue plot of $x_0+c_t$ at $t=0.8$, and $t = 2$; part (b) shows the histogram of the image of the eigenvalues in part (a) pushed-forward under $\Psi$, with the density of $x_0+s_t$ superimposed. We can see that both histograms in part (b) follow the theoretical distribution -- the spectral distribution of $x_0+s_t$.

\begin{figure}[h]
	\begin{center}
	\begin{subfigure}[h]{0.55\linewidth}
		\includegraphics[width=\linewidth]{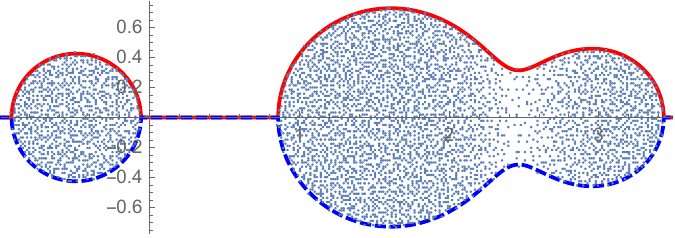}
		\caption{Eigenvalue plot for $x_0+c_{0.8}$ and the graphs $v_{0.8}(a)$ (red) and $-v_{0.8}(a)$ (blue dashed)}
		\end{subfigure}\\
	
		\begin{subfigure}[h]{0.5\linewidth}
		\includegraphics[width=\linewidth]{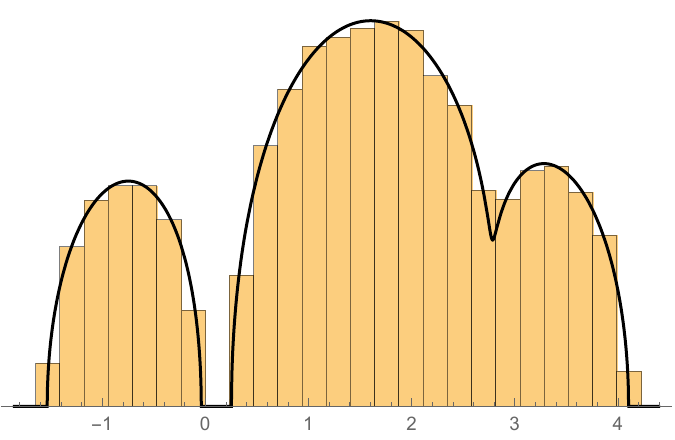}
		\caption{Push-forward of eigenvalues of $x_0+c_{0.8}$ by $\Psi_{0.8}$ and the density of $x_0+s_{0.8}$ superimposed}
		\end{subfigure}
	\end{center}
\begin{center}
	\caption{$2\,000\times 2\,000$ Matrix simulations at $t=0.8$ for $x_0$ distributed as $\frac{1}{5}\delta_{-0.5}+\frac{3}{5}\delta_{1.6}+\frac{1}{5}\delta_{3}$}
\end{center}
\end{figure}

\begin{figure}[h]
	\begin{center}
	\begin{subfigure}[h]{0.5\linewidth}
		\includegraphics[width=\linewidth]{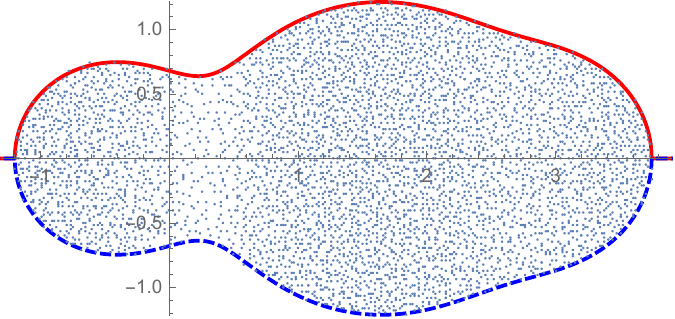}
		\caption{Eigenvalue plot for $x_0+c_2$ and the graphs $v_2(a)$ (red) and $  -v_2(a)$ (blue dashed)}
	\end{subfigure}\\

	\begin{subfigure}[h]{0.5\linewidth}
	\includegraphics[width=\linewidth]{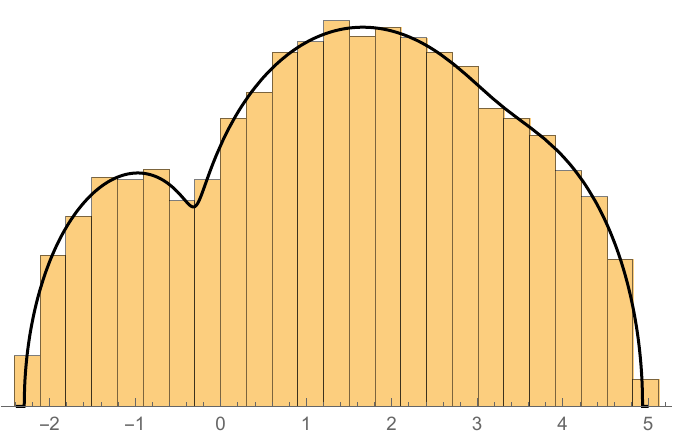}
	\caption{Push-forward of eigenvalues of $x_0+c_2$ by $\Psi_2$ and the density of $x_0+s_{2}$ superimposed}
\end{subfigure}
\end{center}
	\begin{center}
		\caption{$2\,000\times 2\,000$ Matrix simulations at $t=2$ for $x_0$ distributed as $\frac{1}{5}\delta_{-0.5}+\frac{3}{5}\delta_{1.6}+\frac{1}{5}\delta_{3}$}
	\end{center}
\end{figure}

\section{Free Multiplicative Brownian Motion}\label{MultCase}
In this section, we compute the Brown measure of the random variable $ub_t$ where $u$ is a unitary random variable freely independent from the free multiplicative Brownian motion $b_t$.  The strategy in the computation in the multiplicative case is similar to the additive case. We first compute the PDE of the corresponding $S$ which can be solved using the Hamilton--Jacobi method. We then find all the constants of motion of the Hamilton equation and solve a system of ODEs. Because the PDE of $S$ is the same as the one computed in \cite{DHKBrown}, the analysis is similar to the arguments in \cite{DHKBrown}, with initial conditions for our case.

The analysis in this case is much more technical than the additive case. Nevertheless, the idea is that given a point $\lambda$ on the complex plane, we want to find initial conditions $\lambda_0$ and $\varepsilon_0$ such that
\begin{enumerate}
	\item The solution of the Hamiltonian system exists up to time $t$;
	\item $\lim_{s\uparrow
 t}\varepsilon(s) = 0$ and $\lim_{s\uparrow
 t}\lambda(s) = \lambda$.
	\end{enumerate}

In the process, we will see that for $\lambda\not\in\overline{\Dt}$, where $\Dt$ is a certain open set described in Section \ref{some-sets}, the initial condition of $\varepsilon_0$ can be taken arbitrarily small and the above properties still hold. Thus, when $\lambda\not\in\overline{\Dt}$, we can make $\varepsilon(t)$ arbitrarily close to $0$ by making $\varepsilon_0$ arbitrarily close to $0$. For $\lambda\not\in\overline{\Dt}$, one can show that $\Delta_\lambda S(t,\lambda,0) = 0$ and hence the Brown measure is supported in $\overline{\Dt}$. In this paper, we use a different approach; we show that the Brown measure is $0$ outside $\overline{\Dt}$ by showing that the Brown measure has mass $1$ on $\Dt$.

For $\lambda\in \Dt$, we take a proper limit of the initial condition $\varepsilon_0$ such that the lifetime of the solution is up to time $t$. Then we use the Hamilton--Jacobi formulas to compute the Laplacian of $S$.

\subsection{The Differential Equations}
\label{sect:TDE}
Let $b_{t}$ be the free multiplicative
Brownian motion and $u$ be a unitary operator that is freely independent from $b_t$. 
Denote $g_t=ub_t$. 
We consider the function $S$ defined by%
\begin{equation}
S(t,\lambda,\varepsilon)=\tau\lbrack\log((g_{t}-\lambda)^{\ast}(g_{t}-\lambda)+\varepsilon)]
\label{Sdefinition}%
\end{equation}
and set 
\[
s_{t}(\lambda)=\lim_{\varepsilon\downarrow
0}S(t,\lambda,\varepsilon).
\]
Then ~\eqref{BrownAsLimit} shows that the density $W(t,\lambda)$ of the Brown
measure of $g_t$ is computed in terms of $s_{t}(\lambda)$ as
\begin{equation}
W(t,\lambda)=\frac{1}{4\pi}\Delta_\lambda s_{t}(\lambda). \label{WtDef}%
\end{equation}

By applying the free It\^{o} formula, one can prove the following result. The PDE in the following theorem is the same as the one in \cite{DHKBrown}, except that it is written in logarithmic polar coordinates, and has a different initial condition. We refer the interested reader to \cite{DHKBrown} for the proof for the case when $u=I$ and the same argument there works for arbitrary unitary operator $u$, because $g_t$ satisfies the same free SDE, but different initial condition as $b_t$.
\begin{theorem}
	\label{thePDE.thm}The function $S$ in (\ref{Sdefinition}) satisfies the
	following PDE in logarithmic polar coordinates:
	\begin{equation}
	\frac{\partial S}{\partial t}=\varepsilon\frac{\partial S}{\partial \varepsilon}\left(
	1+(\left\vert \lambda\right\vert ^{2}-\varepsilon)\frac{\partial S}{\partial \varepsilon}%
	-\frac{\partial S}{\partial \rho}\right)
	,\quad\lambda=e^{\rho}e^{i\theta}, \label{thePDE}%
	\end{equation}
	with the initial condition%
	\begin{equation}\label{SinitialCond}%
	S(0,\lambda,\varepsilon)=\tau[\log (u-\lambda)^*(u-\lambda)+\varepsilon ]=\int_\bT \log (|\xi-\lambda|^2+\varepsilon) d\mu(\xi),
	\end{equation}
	where $\mu$ is the spectral distribution of $u$. 
\end{theorem}

In \cite{DHKBrown}, Driver, Hall and Kemp studied the properties of the solutions for the PDE (\ref{thePDE}) under the case that the unitary initial condition $u=I$. Since we are solving the same PDE with a different initial condition, the properties of the solution for (\ref{thePDE}) with the initial condition (\ref{SinitialCond}) are similar to those obtained in their work. 

The equation (\ref{thePDE}) is a first-order, nonlinear PDE of
Hamilton--Jacobi type (see for example, Section 3.3 in
the book of Evans \cite{EvansBook}). We will use the polar coordinates and the logarithmic polar coordinates for $\lambda$; we write $\lambda = e^{\rho}e^{i\theta} = re^{i\theta}$. We define the Hamiltonian corresponding to (\ref{thePDE}) by
\begin{equation}
H(\rho,\theta,\varepsilon,p_{\rho},p_\theta,p_{\varepsilon})=-\varepsilon p_{\varepsilon}(1+r^2p_{\varepsilon}-\varepsilon p_\varepsilon-p_{\rho}).
\label{theHamiltonian}%
\end{equation}
Although the right hand side of~\eqref{theHamiltonian} is independent of $\theta$ and $p_\theta$, the function $S$ does depend on $\theta$. We consider $\theta$ and $p_{\theta}$ as variables of $H$ so that we can apply the second Hamilton--Jacobi formula \eqref{eq:2ndHJ} to compute $\frac{\partial S}{\partial \theta}$. We consider Hamilton's equations for this Hamiltonian; that is, we
consider this system of six coupled ODEs:%
\begin{align}
\frac{d\rho}{dt}  &  =\frac{\partial H}{\partial p_{\rho}};\quad~~~\frac{d\theta}{dt} = \frac{\partial H}{\partial p_\theta}\quad~~~\frac{d\varepsilon}{dt}=\frac{\partial
	H}{\partial p_{\varepsilon}};\nonumber\\
\frac{dp_{\rho}}{dt}  &  =-\frac{\partial H}{\partial \rho};\quad~~~\frac{dp_\theta}{dt} = -\frac{\partial H}{\partial\theta} \quad\frac{dp_{\varepsilon}}{dt}=-\frac{\partial
	H}{\partial \varepsilon}. \label{theODEs}%
\end{align}
Since the right hand side of~\eqref{theHamiltonian} is independent of $\theta$ and $p_\theta$, it is clear that $d\theta/dt = dp_\theta/dt = 0$. Hence, $\theta$ and $p_\theta$ are independent of $t$.

Here we require that $\varepsilon(t)$ be positive, while all other quantities are real-valued. For convenience, we write
\[
\lambda(t)=r(t)e^{i\theta} = e^{\rho(t)}e^{i\theta}.
\]
Note that we have used, in the above equation, that the Hamiltonian is independent of $p_\theta$. The initial conditions for $\rho$ and $\varepsilon>0$ are arbitrary. 
Given 
\begin{equation}
\rho(0)=\rho_{0};\quad \varepsilon(0)=\varepsilon_{0}, \label{initialConditions1}%
\end{equation}
We write $r(0) = r_0 = e^{\rho_0}$, and initial condition \eqref{SinitialCond} of $S$ as 
\begin{align}
  S(0,\lambda,\varepsilon)&=\int_\bT \log (|\xi-\lambda|^2+\varepsilon) d\mu(\xi)\nonumber\\
  &=\int_{-\pi}^\pi \log (|e^{i\alpha}-\lambda|^2+\varepsilon) d\mu(e^{i\alpha})\nonumber\\
  &=\int_{-\pi}^\pi \log\Big( 1+r_0^2-2r_0\cos(\theta-\alpha)+\varepsilon \Big)d\mu(e^{i\alpha}).
\end{align}
The initial
momenta $p_{\rho,0}=p_\rho(0)$, $p_\theta$ and $p_{0}=p_\varepsilon(0)$ are chosen as 
\[
  p_{\rho,0} = p_\rho(0)= \frac{\partial  S(0,\lambda_0,\varepsilon_0)}{\partial \rho_0},\quad p_{\theta} = \frac{\partial  S(0,\lambda_0,\varepsilon_0)}{\partial \theta},
  \quad p_0 = p_\varepsilon(0)=\frac{\partial  S(0,\lambda_0,\varepsilon_0)}{\partial \varepsilon_0}.
\]
With these choices of momenta, we will apply Proposition~\ref{GeneralSol} to see that the momenta correspond to the partial derivatives of $S$ along the curve $(\rho(t), \theta, \varepsilon(t))$.

Recall that $\mu$ is the spectral distribution of $u$, we can write the above definitions explicitly as follows.
\begin{align}
p_{\rho,0} = p_\rho(0)&= \frac{\partial  S(0,\lambda_0,\varepsilon_0)}{\partial \rho_0}
\nonumber\\
  &=\int_{-\pi}^\pi \frac{2r_0(r_0-\cos(\theta-\alpha))}{1+r_0^2-2r_0\cos(\theta-\alpha)+\varepsilon_{0}}d\mu(e^{i\alpha});
 \label{condition_a}
\end{align}

\begin{align}
	p_{\theta}&= \frac{\partial  S(0,\lambda_0,\varepsilon_0)}{\partial \theta}
	\nonumber\\
	&=\int_{-\pi}^\pi \frac{2r_0\sin(\theta-\alpha)}{1+r_0^2-2r_0\cos(\theta-\alpha)+\varepsilon_{0}}d\mu(e^{i\alpha});
	\label{condition_theta}
\end{align}

and 
\begin{align}
p_0 = p_\varepsilon(0)&=\frac{\partial  S(0,\lambda_0,\varepsilon_0)}{\partial \varepsilon_0} \nonumber\\
&=\int_{-\pi}^\pi \frac{1}{1+r_0^2-2r_0\cos(\theta-\alpha)+\varepsilon_{0}}d\mu(e^{i\alpha})\nonumber\\
&=\int_\bT \frac{1}{|\xi-\lambda_0|^2+\varepsilon_0}d\mu(\xi),
\label{condition_epsilon}
\end{align}
where $\lambda_0=r_0e^{i\theta}$.

After a change of variable to the rectangular coordinates, the Hamiltonian system \eqref{theHamiltonian} is the same as the one studied in \cite{DHKBrown} because the PDE of $S$ is the same. The solution of the system of coupled ODEs \eqref{theODEs}, given initial conditions, is very similar to \cite[Section 6]{DHKBrown}; if we do not write the initial momenta explicitly but leave as symbols $p_{\rho,0}$ and $p_0$, the solutions look pretty much the same. Adapted the new initial conditions, we will see how the Laplacian of $S$ changes and we will be able to analyze and identify the Brown measure of $ub_t$.

\begin{lemma}
	\label{H0.lem}The value of the Hamiltonian at $t=0$ is%
	\begin{equation}
	H_{0}=-\varepsilon_{0}p_{0}^{2}. \label{H0Formula}%
	\end{equation}
\end{lemma}
\begin{proof}
We calculate 
\begin{align*}
&\quad r_0^{2}p_0-\varepsilon_0p_0-p_{\rho,0}\\
    &=\int_{-\pi}^\pi \frac{-r_0^2+2r_0\cos(\theta_0-\alpha)-\varepsilon_0}{1+r_0^2-2r_0\cos(\theta_0-\alpha)+\varepsilon_0}d\mu(e^{i\alpha})\\
    &=-1+p_0.
\end{align*}
Hence, $H_0=-\varepsilon_0p_0(1+r_0^{2}p_0-\varepsilon_0p_0-p_{\rho,0})=-\varepsilon_0p_0^2$.
\end{proof}

We record the following result which is modified from \cite[Theorem 6.2]{DHKBrown} for our choice of initial conditions. 
\begin{theorem}
	\label{HJ.thm}Assume $\lambda_{0}\neq0$ and $\varepsilon_{0}>0.$ Suppose a solution to
	the system (\ref{theODEs}) with initial conditions 
	(\ref{initialConditions1}), and (\ref{condition_a})
	- (\ref{condition_epsilon}) exists with $\varepsilon(t)>0$ for $0\leq t<T.$ Then we
	have%
	\begin{align}
	S(t,\lambda(t),\varepsilon(t))   =&\int_\bT \log (|\xi-\lambda_0|^2+\varepsilon_0) d\mu(\xi)
	-\varepsilon_{0}t\left(\int_{\mathbb{T}}\frac{1}{|\xi-\lambda_0|^2+\varepsilon_0}d\mu(\xi)   \right)^2 \nonumber\\
	&  +\log\left\vert \lambda(t)\right\vert -\log\left\vert \lambda
	_{0}\right\vert \label{Sformula}%
	\end{align}
	for all $t\in\lbrack0,T).$ Furthermore, the derivatives of $S$ with respect to
	$\rho$ and $\varepsilon$ satisfy
	\begin{align}
	\frac{\partial S}{\partial \varepsilon}(t,\lambda(t),\varepsilon(t))  &  =p_{\varepsilon}(t);\nonumber\\
	\frac{\partial S}{\partial \rho}(t,\lambda(t),\varepsilon(t))  &  =p_{\rho}(t).
	\label{Sderivatives}%
	\end{align}
	
\end{theorem}
\begin{proof}
	We compute
	\begin{align*}
	p_\rho\frac{d\rho}{dt}+p_\varepsilon\frac{d\varepsilon}{dt}& = p_\rho\frac{\partial H}{\partial p_\rho}+p_\varepsilon\frac{\partial H}{\partial p_\varepsilon}\\
	& =\varepsilon p_\varepsilon-2\varepsilon p_\varepsilon(1+r^2p_\varepsilon-\varepsilon p_\varepsilon-p_\rho) \\
	&=\varepsilon p_\varepsilon+2 H=\varepsilon p_\varepsilon+2H_0.
	\end{align*}
By Proposition \ref{GeneralSol}, we have
\[
   S(t,\lambda(t), \varepsilon(t)) = S(0, \lambda_0, \varepsilon_0)+tH_0+\int_0^t \varepsilon(s)p_\varepsilon(s)\,ds.
\]
We also have 
\[\frac{d\rho}{dt} = \frac{\partial H}{\partial p_\rho} = \varepsilon p_\varepsilon,\]
which yields that 
\begin{equation}
\int_{0}^{t}\varepsilon(s)p_{\varepsilon}(s)\,ds=\log\left\vert \lambda(t)\right\vert
-\log\left\vert \lambda_{0}\right\vert . \label{logLambdaIntegral}%
\end{equation}
If we plug in the values of $S(0, \lambda_0, \varepsilon_0)$ in (\ref{SinitialCond}), $H_0$ and $p_0$ in (\ref{H0Formula}), we obtain (\ref{Sformula}).
\end{proof}

It is very important to understand the constants of motion of a Hamiltonian system. 
\begin{proposition}\label{conservation}
The following quantities remain constant along any solution of (\ref{theODEs}):
\begin{enumerate}
\item The Hamiltonian $H$,
\item The argument $\theta$ of $\lambda$ if $\lambda_0\neq 0$,
\item The momentum $p_\theta$ with respect to $\theta$,
\item The function defined by $\Psi:=\varepsilon p_\varepsilon+\frac{1}{2}p_\rho$.
\end{enumerate}
\end{proposition}
\begin{proof}
These are the same constants of motion in \cite[Propositions 6.4, 6.5]{DHKBrown}. It can also be checked directly using \eqref{theODEs}.
\end{proof}

\subsection{Solving the ODEs}
We will need the following values to solve the coupled ODEs. Using initial conditions (\ref{condition_a})-(\ref{condition_epsilon}) and the fact that $\Psi$ is a constant of motion, we have
\begin{align}
	\Psi&=\varepsilon p_\varepsilon+\frac{1}{2}p_\rho\nonumber\\
	   &=\int_{-\pi}^\pi \frac{|\lambda_0|^2-|\lambda_0|\cos(\theta_0-\alpha)+\varepsilon_0}{1+|\lambda_0|^2-2|\lambda_0|\cos(\theta_0-\alpha)+\varepsilon_0}d\mu(e^{i\alpha}). \label{eqn:Psi}
\end{align}	
We set
	   \begin{align}
	   C&= 2\Psi-1\nonumber\\
	   &=\int_{-\pi}^\pi \frac{|\lambda_0|^2-1+\varepsilon_0}{1+|\lambda_0|^2-2|\lambda_0|\cos(\theta_0-\alpha)+\varepsilon_0}d\mu(e^{i\alpha})\\\label{eqn:C}
	   &=p_0(r_0^2-1+\varepsilon_0).
	 \end{align}	
 and
\begin{equation}
	   \delta=	  \frac{r_0^2+1+\varepsilon_0}{r_0}\label{value-delta}.
\end{equation}
\begin{proposition}
	\label{xpx}
 For all $t$, we have
\begin{equation}\label{eqn:xp_x}
   \varepsilon(t)p_\varepsilon(t)^2=\varepsilon_0p_0^2e^{-Ct},
\end{equation}
where $C$ is given by $(\ref{eqn:C})$. 
\end{proposition}
\begin{proof}
This is established by solving the ODE that $\varepsilon(t)p_\varepsilon(t)^2$ satisfies. See \cite[Proposition 6.6]{DHKBrown} for details. The difference in our initial conditions does not play a role in the proof.
\end{proof}

We now adapt the analysis of the proof for \cite[Proposition 6.9]{DHKBrown} (which is by solving the system of coupled ODEs), 
under the initial conditions (\ref{initialConditions1}) and (\ref{condition_a})-(\ref{condition_epsilon}), to analyze the function $S$, defined in \eqref{Sdefinition}, in our case. We calculate, by \eqref{theHamiltonian} and \eqref{theODEs}, 
\[
\frac{dp_\varepsilon}{dt}=-\frac{\partial H}{\partial{\varepsilon}}=-\frac{H}{\varepsilon(t)}-\varepsilon(t)p_\varepsilon(t)^2.
\] 
The Hamiltonian $H$ is a constant of motion. By Lemma \ref{H0.lem} and Proposition \ref{xpx}, the above equation is expressed as
\begin{align}
\frac{dp_\varepsilon}{dt}&=-\frac{\varepsilon_0 p_0^2}{\varepsilon(t)}-\varepsilon_0 p_0^2e^{-Ct}\nonumber\\
  &=-p_\varepsilon(t)^2e^{Ct}-\varepsilon_0 p_0^2e^{-Ct}.\label{p-epsilon-t-ode}
\end{align}
The ODE \eqref{p-epsilon-t-ode} can be solved explicitly as the constants $C$, $\varepsilon_0$, and $p_0$ have been given. 
Following \cite[Proposition 6.9]{DHKBrown}, the solution of \eqref{p-epsilon-t-ode} is
the $p_\varepsilon$ component of the system (\ref{theODEs}) which is
\begin{equation}
p_{\varepsilon}(t)=p_0\frac{\cosh(kt)+\frac{2\left\vert \lambda_{0}\right\vert
		-\delta}{\sqrt{\delta^{2}-4}}\sinh(kt)}{\cosh(kt)-\frac{\delta}{\sqrt
		{\delta^{2}-4}}\sinh(kt)}e^{-Ct} \label{pxFormula}%
\end{equation}
where
\begin{align}
k&=\frac{1}{2}p_0\cdot r_0\cdot\sqrt{\delta^2-4}\label{value_a}
\end{align}
for as long as the solution to the system (\ref{theODEs}) exists, where we
use the same choice of $\sqrt{\delta^{2}-4}$ as in the definition of $k$ in (\ref{value_a}). If $\delta=2,$ we
interpret $\sinh(kt)/\sqrt{\delta^{2}-4}$ as $\frac{1}{2}%
p_0 r_0 t.$
In addition, if $\varepsilon_0\geq 0$, the numerator $\cosh(kt)-\frac{\delta}{\sqrt
	{\delta^{2}-4}}\sinh(kt)$ is positive for all $t$. Hence, the function $p_\varepsilon(t)$ is positive as long as the solution exists and its reciprocal $1/p_\varepsilon(t)$ is a real analytic function of $t$ defined for all $t\in \mathbb{R}$.  
Moreover, the first time the expression (\ref{pxFormula}) blows up is the time when the denominator is zero, which is
\begin{align}
t_{\ast}(\lambda_{0},\varepsilon_{0})  
&  =\frac{2}{p_0 r_0 }\frac{1}%
{\sqrt{\delta^{2}-4}}\tanh^{-1}\left(  \frac{\sqrt{\delta^{2}-4}}{\delta
}\right)\label{tStar1}\\
& = \frac{1}{p_0r_0 }\frac{1}%
{\sqrt{\delta^{2}-4}}\log\left(  \frac{\delta+\sqrt{\delta^{2}-4}}{\delta-\sqrt{\delta^{2}-4}
}\right)\label{tStar2}.
\end{align}
Here, the principal branch of the inverse
hyperbolic tangent should be used in (\ref{tStar1}), with branch cuts
$(-\infty,-1]$ and $[1,\infty)$ on the real axis, which corresponds to using
the principal branch of the logarithm in (\ref{tStar2}). When $\delta=2,$ we
interpret $t_{\ast}(\lambda_{0},\varepsilon_{0})$ as having its limiting value as
$\delta$ approaches $2$.

We now describe limit behaviors of $\varepsilon(t)$ and $\lambda(t)$ by adapting the arguments in \cite[Section 6.3]{DHKBrown}.
By (\ref{eqn:xp_x}), we have
\begin{equation}\label{eqn:xt}
    \varepsilon(t)=\frac{\varepsilon_0p_0^2e^{-Ct}}{p_\varepsilon(t)^2}.
\end{equation}
As $t$ approaches $t_*(\lambda_0, \varepsilon_0)$ from the left, $p_\varepsilon(t)$ remains positive until it blows up and $\varepsilon(t)$ approaches zero. 
That is to say, for any $\varepsilon_0$, we have
\begin{equation}
\lim_{t\rightarrow t_{\ast}(\lambda_{0},\varepsilon_{0})}\varepsilon(t)=0. \label{xBecomesZero}
\end{equation}

If $\varepsilon_{0}=0,$ from (\ref{eqn:xt}), we see that the solution has $\varepsilon(t)\equiv0$;
and we deduce that
$\lambda(t)\equiv\lambda
_{0}$ from (\ref{logLambdaIntegral}).
By \cite[Theorem 6.7]{DHKBrown} (our initial conditions are different from the system in \cite{DHKBrown}, but it only uses \eqref{eqn:xt} and the fact that $H$ is a constant of motion), if $\varepsilon(t)\to 0$ as $t\to t_\ast(\lambda_0,\varepsilon_0)$ with $\varepsilon_0> 0$, we have
\begin{equation}\label{lambda-life-time}
   \lim_{t\rightarrow t_{\ast}(\lambda_{0},\varepsilon_{0})}\rho(t)=\frac{Ct_{\ast}(\lambda_{0},\varepsilon_{0})}{2}.
\end{equation} 
\begin{proposition}
We have 
\begin{equation}
\lim_{t\rightarrow t_{\ast}(\lambda_{0},\varepsilon_{0})}p_\rho
=2\Psi=C+1=\lim_{t\rightarrow t_{\ast}(\lambda_{0},\varepsilon_{0})}%
\frac{2\rho(t)}{t}+1. \label{apa}%
\end{equation}
\end{proposition}
\begin{proof}
Note that $\lim_{t\to t_*(\lambda_0,\varepsilon_0)}\varepsilon(t) p_\varepsilon(t) =\lim_{t\to t_*(\lambda_0,\varepsilon_0)}\sqrt{\varepsilon(t)}\sqrt{\varepsilon(t) p_\varepsilon(t)^2} = 0$ by Proposition \ref{xpx} and \eqref{xBecomesZero}, 
we then deduce from (\ref{eqn:Psi}) to obtain  
\begin{align*}
 \lim_{t\rightarrow t_{\ast}(\lambda_{0},\varepsilon_{0})}p_\rho
   =2\Psi.
\end{align*}
Then \eqref{eqn:C} and \eqref{lambda-life-time} yield the result. 
\end{proof}


\subsection{Monotonicity of the lifetime}

We will show first that the lifetime of the Hamiltonian system is an increasing function of $\varepsilon_0$. To this end, let us recall the
following elementary lemma appeared in the proof of \cite[Proposition 6.16]{DHKBrown}.
\begin{lemma}\label{lemma:g}
Given $\theta\in (-\pi,\pi]$, the function $g_\theta$ defined by
\begin{equation}
g_{\theta}(x):=\frac{x-2\cos\theta}{\sqrt{x^{2}-4}}%
\log\left(  \frac{x+\sqrt{x^{2}-4}}{x-\sqrt{x^{2}-4}%
}\right). \label{gDelta}%
\end{equation}
is strictly increasing, non-negative, continuous function of $x$ for $x\geq 2$ and tends to $\infty$ as $x$ tends to infinity. 
\end{lemma}

\begin{proposition}
	\label{monotoneTstar.prop}For each $\lambda_{0},$ the function $t_{\ast
	}(\lambda_{0},\varepsilon_{0})$ is a strictly increasing function of $\varepsilon_{0}$ for
	$\varepsilon_{0}\geq0$ and
	\[
	\lim_{\varepsilon_{0}\rightarrow+\infty}t_{\ast}(\lambda_{0},\varepsilon_{0})=+\infty.
	\]
	
\end{proposition}
\begin{proof}
For $\lambda_0$ fixed, recall the definition of $\delta$ in (\ref{value-delta}) and $t_*(\lambda, \varepsilon)$ in (\ref{tStar2}), we define the function
\begin{equation}
    f_{\lambda_0}(\delta)=\frac{1}{ t_{\ast}(\lambda_{0},\varepsilon_{0})  }
       = {p_0r_0 }
      {\sqrt{\delta^{2}-4}}\log\left(  \frac{\delta-\sqrt{\delta^{2}-4}}{\delta+\sqrt{\delta^{2}-4}
      }\right)\label{function_f}.
\end{equation}
By the expression for $p_0$ in (\ref{condition_epsilon}), we obtain that
\[
 p_{0}r_0=\int_{-\pi}^\pi \frac{1}{\delta-2\cos(\theta_0-\alpha)}d\mu(e^{i\alpha}). 
\]
We then can rewrite $f_{\lambda_0}(\delta)$  as
\[
   f_{\lambda_0}(\delta)=\int_{-\pi}^\pi \frac{1}{g_{\theta_0-\alpha}(\delta)}d\mu(e^{i\alpha}),
\]
where $g_{\theta_0-\alpha}(\delta)$ is defined in (\ref{gDelta}).
Hence, as $t_*(\lambda_0,\varepsilon_0)$ is the reciprocal of $f_{\lambda_0}(\delta)$ as in (\ref{function_f}),
it then follows from Lemma \ref{lemma:g} that
 the function $t_*(\lambda_0,\varepsilon_0)$ is a strictly increasing, non-negative, continuous function of $\delta$ for $\delta\geq 2$ that tends to $\infty$ as $\delta$ tends to $\infty$. This finishes the proof. 
\end{proof}

Similar to the additive case, the complex plane $\mathbb{C}$ is naturally divided into two sets when we calculate the Brown measure of the free multiplicative Brownian motion. The following function $T$ determines how we divide $\C$ into two sets. The analysis to achieve $\varepsilon(t)=0$ is different in these two sets. We will see that one of these two sets give full measure of the Brown measure, and we omit the analysis of the other set in this paper. Interested readers can read the corresponding analysis on another set in \cite{DHKBrown} (for the case $u=I$).

Define the function $T:\mathbb{C}:\rightarrow [0,\infty)$
\begin{equation}\label{Tlambda}
T(\lambda_0) = 
\begin{cases}
\frac{1}{p_0}\frac{\log r_0^2}{r_0^2-1}, & \text{for } \lambda_0=r_0e^{i\theta_0}\quad \text{and}\quad r_0\neq 1 \\
\frac{1}{p_0}, & \text{for } \lambda_0=r_0e^{i\theta_0}\quad \text{and}\quad r_0=1. 
\end{cases}
\end{equation}
When $\mu =\delta_1$, the function $T$ reduces to the function in \cite{DHKBrown}. The following results are the analogous results in \cite[Section 6.4]{DHKBrown}. Roughly speaking, it says that the function $T$ is the lifetime of the solution ``when $\varepsilon_0=0$".
\begin{proposition}\label{life-time}
	\label{smallx0.prop} Recall that $t_{\ast}(\lambda_{0},\varepsilon_{0})$ is defined by
	\eqref{tStar2}. Then for all nonzero $\lambda_{0}$ we have%
	\begin{equation}
	t_{\ast}(\lambda_{0},0):=\lim_{\varepsilon_0\downarrow 0}t_\ast (\lambda_0,\varepsilon_0)=T(\lambda_{0}), \label{tStarZero}%
	\end{equation}
	where the function $T$ is defined in (\ref{Tlambda}). Furthermore, when
	$\varepsilon_{0}=0,$ we have
	\begin{equation}
	\lim_{t\rightarrow t_{\ast}(\lambda_{0},\varepsilon_{0})}\rho(t)=\rho_0. \label{rhoZero}%
	\end{equation}
	
\end{proposition}

\begin{proof}
	We first consider the case when $|\lambda_0|=1$. In this case, we have $ \lim_{\varepsilon_0\rightarrow 0} \delta=2$ and $\lim_{\varepsilon_0\to 0}\sqrt{\delta^2-4} = \pm\frac{|\lambda_0|^2-1}{|\lambda_0|}$.
	We can then compute the limit
	   \begin{equation}
	   \lim_{\delta\downarrow
2}\frac{1}{\sqrt{\delta^{2}-4}}\log\left(
	   \frac{\delta+\sqrt{\delta^{2}-4}}{\delta-\sqrt{\delta^{2}-4}}\right)  =1
	   \label{gLim}%
	   \end{equation}
and obtain
\[
  \lim_{\varepsilon_0\downarrow
 0}t_{*} (\lambda_0, \varepsilon_0)=\frac{1}{p_0},
\]	   
where $p_0$, given by (\ref{condition_epsilon}), can also be written as
\[
    p_0=\int_{\mathbb{T}}\frac{1}{|\xi-\lambda_0|^2}d\mu(\xi),\quad \text{with}\; \theta=\arg(\lambda_0).
\]
Hence $\lim_{\varepsilon_0\rightarrow 0}t_{*} (\lambda_0, \varepsilon_0)=T(\lambda_0)$ by \eqref{Tlambda}.

For the case $|\lambda_0|\neq 1$, we note that
\[
  \lim_{\varepsilon_0\downarrow
 0}  \frac{\delta+\sqrt{\delta^{2}-4}}{\delta-\sqrt{\delta^{2}-4}
  }=r_0^2.
\]
and
\[
  \lim_{\varepsilon_0\downarrow
 0} r_0\sqrt{\delta^2-4}=r_0^2-1.
\]
Hence, from (\ref{tStar2}), we deduce that
\[
  \lim_{\varepsilon_0\downarrow
 0} t_{*} (\lambda_0, \varepsilon_0)=\frac{1}{p_0}\frac{\log r_0^2}{r_0^2-1}=T(\lambda_0), \quad \text{where}\; \theta=\arg(\lambda_0),
\]
due to the expressions of $p_\varepsilon(0)$ as in (\ref{valueOFf}) and (\ref{condition_epsilon}) respectively.

The proof of (\ref{rhoZero}) follows from a similar calculation using \eqref{lambda-life-time}, and $\sqrt{\delta^2-4} = \pm\frac{|\lambda_0|^2-1}{|\lambda_0|}$.
\end{proof}

\subsection{The domains $\Dt$ and $\Ot$, and their relations to the lifetime}\label{some-sets}

Recall from Section~\ref{section:multiplicative_convolution} that the set $\Omega_{t,\overline{\mu}}=\{\omega_t(z): z\in\mathbb{D} \}$ defined in \eqref{def:Omega_t} is star-like with respect to the origin, by Proposition \ref{characterization-0}. Moreover, the graph of the function $r_t$ defined in~\eqref{def:r_t_theta} is exactly the boundary set $\partial(\Ot)$. That is, by Proposition~\ref{characterization-0}(1),
\[
    \Ot=\{  r e^{i\theta}: -\pi\leq \theta\leq \pi, 0\leq r< r_t(\theta)\}.
\]
See Figure 5 for an example. 
Define the open set
\[
\Dt:=  \left\{ z, 1/\overline{z}: z\in \overline{\mathbb{D}}\backslash \overline{\Ot} \right \}.
\]
We will prove in Theorem \ref{sLaplacian.main} that the closure $\overline{\Dt}$ of $\Dt$ is the support of the Brown measure of $ub_t$; thus, we call it $\Dt$ instead of $\Delta_{t,\bar{\mu}}$. All other notations are related to the subordination function of $u^*u_t$ with respect to $u^*$ so we use $\bar{\mu}$ as subscripts in those notations.

In this section, we first establish the following theorem which tells us the relation between the domain $\Dt$ and the function $T$ (which is the lifetime of the solution when $\varepsilon_0=0$).
\begin{theorem}\label{Dt-characterization}
For any $t>0$, the region $\Dt$ is invariant under $\lambda\mapsto 1/\overline{\lambda}$ and we have
\begin{equation}\label{eqn:Dt-1}
   \Dt=\{ re^{i\theta}: r_t(\theta)<r<1/r_t(\theta) , \theta\in U_{t,\bar{\mu}}\},
\end{equation}
where $r_t$ is the function defined in \eqref{def:r_t_theta}.
Moreover, $\Dt$ may be
expressed as%
\[
\Dt=\left\{  \left.  \lambda\in\mathbb{C}\right\vert T(\lambda
)<t\right\}.
\]

For any $\lambda\notin \overline{\Dt}$, $T(\lambda)> t$. That is,
\[
   ( \mathbb{C}\backslash \overline{\Dt})\subset \{\lambda\in\mathbb{C}| T(\lambda)>t \}.
\]
For any $\lambda\in \partial \Dt\cap \mathbb{T}$, we have $T(\lambda)\geq t$; and for any 
$\lambda\in \partial \Dt \backslash \mathbb{T}$, we have $T(\lambda)=t$. 
\end{theorem}

Theorem~\ref{Dt-characterization} characterizes the complement of $\overline{\Omega_{t,\mu}}$ using the function $T$. With this characterization, the strategy of letting $\varepsilon_0\to 0^+$ could work on $\C\setminus \overline{\Omega_{t,\mu}}$, using an argument parallel to that in \cite{DHKBrown}, where the case $\mu=\delta_1$ was considered. We do not analyze $\C\setminus \overline{\Omega_{t,\mu}}$ in this paper because we can show that the Brown measure has full measure in $\Omega_{t,\mu}$ (see Lemma \ref{ProbBrownMul}). 

 We will prove Theorem~\ref{Dt-characterization} later. In the following, we establish a corollary of this theorem.

\begin{corollary}
	\label{smallx0.cor}For $\lambda_{0}\in\Dt,$ we have $t_{\ast}%
	(\lambda_{0},0)<t,$ and for $\lambda_{0}\notin\overline{\Dt},$
	we have $t_{\ast}(\lambda_{0},0)>t$.
	For $\lambda_{0}\in\partial\Dt\backslash \mathbb{T},$ we have
	$t_{\ast}(\lambda_{0},0)=t$, and for $\lambda_0\in \partial\Dt \cap \mathbb{T}$, we have
	$t_*(\lambda_0, 0)\geq t$. 
\end{corollary}
\begin{proof}
It is due to \eqref{tStarZero} and Proposition \ref{life-time} and Theorem \ref{Dt-characterization}. 
\end{proof}
\begin{remark}
	\label{rem:disconT}
	We can only conclude
	$t_*(\lambda, 0)\geq t$ for $\lambda\in \partial\Dt \cap \mathbb{T}$ but   \emph{not} $t_*(\lambda,0)=t$, because $t_*(\lambda,0)=T(\lambda)$ may not be continuous at these points and $T>t$ can occur on $ \partial\Dt \cap \mathbb{T}$. For example, if $d\mu(e^{ix}) = \1_{[0,1]}3x^2\,d(e^{ix})$, then, for all $t<1/3$, $T(e^{i})=1/3>t$ and $e^{i}\in \partial\Dt \cap \mathbb{T}$.
	\end{remark}
	

	\begin{figure}[h]
	\label{figure-Ot}
	\begin{center}
		\begin{subfigure}[h]{0.25\linewidth}
			\includegraphics[width=\linewidth]{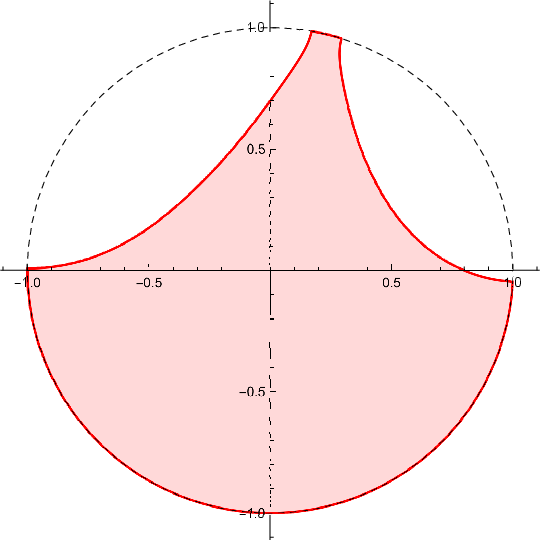}
			\caption{$\Ot$}
		\end{subfigure}\;
		\begin{subfigure}[h]{0.3\linewidth}
			\includegraphics[width=\linewidth]{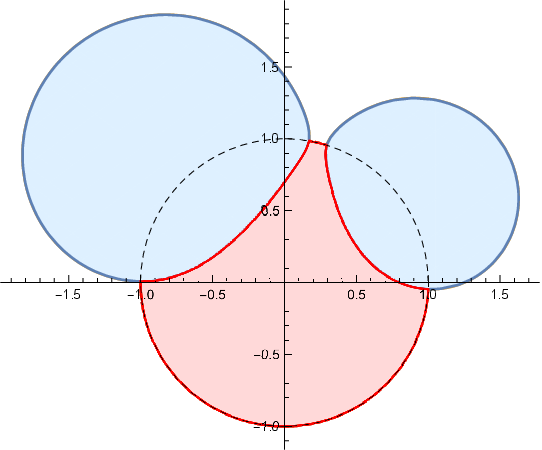}
		\caption{$\Ot$ and $\Dt$}
		\end{subfigure}
		\begin{subfigure}[h]{0.3\linewidth}
			\includegraphics[width=\linewidth]{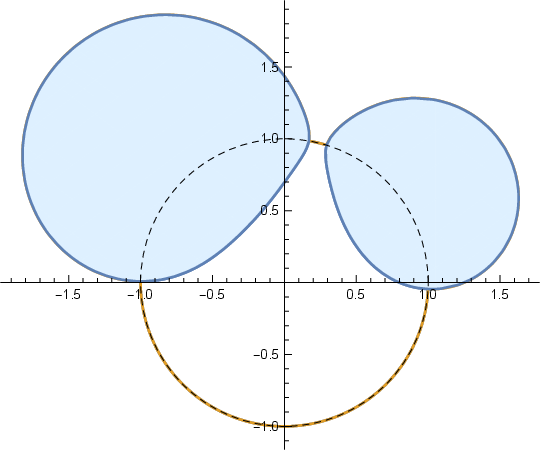}
			\caption{$\Dt$}
		\end{subfigure}
	\end{center}
	\caption{The sets $\Omega_{t,\bar\mu}$ and $\Delta_{t,\mu}$ when $u$ is distrbuted as $\frac{1}{3}\delta_{e^{\frac{\pi i}{6}}}+\frac{2}{3}\delta_{e^{\frac{3\pi i}{4}}}$ and \, $t=0.8$. Plotted with the unit circle (black-dashed).}
\end{figure}	

We now describe some important properties of the set $\Ot$ through subordination function. Recall that
\begin{equation}\label{eqn:Phi_t_mu_bar}
   \Phi_{t,\bar{\mu}}(z)=z	\St(z)=z\exp\left(\frac{t}{2}\int_{\mathbb{T}}\frac{1+\xi z}{1-\xi z}d\bar{\mu}(\xi) \right).
\end{equation}
We define the function
\begin{equation}\label{thetaPhi}
	\phi=\phi(\theta)= \theta+t\int_{-\pi}^\pi \frac{r_{t}(\theta) \sin(\theta-x)}{|1-r_t(\theta)e^{i(\theta-x)}|^2}d\mu(e^{ix}),\quad \theta\in\R. 
\end{equation}
Recall that the function $\omega_t$ extends to a homeomorphism from $\bar\bbD$ onto $\Omega_{t,\bar\mu}$; the function $\phi$ is a strictly increasing continuous function from $[a,a+2\pi)$ onto some interval $[b,b+2\pi)$. Then, given any $a\in\R$, $\phi$ is a continuous version of $\arg(\Phi_{t,\bar{\mu}}(r_t(\theta)e^{i\theta}))$ for any $\theta\in[a,a+2\pi)$; that is, $e^{i\phi}= \Phi_{t,\bar{\mu}}(r_t(\theta)e^{i\theta})$.

We write
\begin{equation}
  \log (|\Phi_{t, \bar{\mu}}(re^{i\theta})|)
  =\log r+\textrm{Re}\left(\frac{t}{2} \int_{\mathbb{T}}\frac{1+\xi re^{i\theta}}{1-\xi re^{i\theta}}d\bar{\mu}(\xi) \right) 
  =(\log r)h_t(r,\theta),
\end{equation}
where 
\[
h_t(r,\theta)=1-\frac{t}{2}\frac{r^2-1}{\log r}\int_{\mathbb{T}}\frac{1}{|1-re^{i\theta} \xi|^2}d\bar{\mu}(\xi),
\]
We let, for $\lambda\in \mathbb{C}$ but $|\lambda|\neq 0, 1$, 
\begin{align}
f(r,\theta)=\frac{1-h_t(r,\theta)}{t}&=\frac{1}{2}\frac{r^2-1}{\log r}\int_{\mathbb{T}}\frac{1}{|1-re^{i\theta} \xi|^2}d\bar{\mu}(\xi)\nonumber\\
  &=\frac{1}{2}\frac{r^2-1}{\log r}\int_{\mathbb{T}}\frac{1}{|1-re^{i(\theta-x)}|^2}d\mu(e^{ix})\label{valueOFf}\\
  &=\frac{1}{2}\frac{r^2-1}{\log r}\int_{-\pi}^\pi \frac{1}{1-2r\cos(\theta-x)+r^2}d\mu(e^{ix})\nonumber.
\end{align}

We will need the following elementary fact. 
\begin{lemma}[Lemma 3.1 of \cite{Zhong2015}]\label{lemma:derivative}
  Given $-1\leq y\leq 1$, define a function of $r$ by
\begin{equation}\nonumber
R_{y}(r)=\frac{r^2-1}{\log r}\frac{1}{1-2ry+r^2}
\end{equation}
on the interval $(0,1)$,
then $R_{y}'(r)>0$ for all $r\in(0,1)$.
\end{lemma}

\begin{corollary}\label{cor:limit}
For fixed $\theta$, the function $r\mapsto f(r,\theta)$ is increasing on $(0,1)$ and decreasing on $(1,\infty)$.
We have
\[
\lim_{r\rightarrow 0} f(r,\theta)=\lim_{r\rightarrow\infty}f(r,\theta)=0.
\]
Moreover, $f(r,\theta)=f(1/r,\theta)$.
\end{corollary}
\begin{proof}
By Lemma \ref{lemma:derivative}, we see directly that $\frac{\partial f(r,\theta)}{\partial r}>0$ for $r\in (0,1)$, so $r\mapsto f(r,\theta)$ is increasing for $r\in(0,1)$. It is straightforward to check that
$f(r,\theta)=f(1/r,\theta)$; thus, we deduce that the function $r\mapsto f(r,\theta)$ is
decreasing on $(1,\infty)$. 
The assertion about the limits can be checked directly from the definition of the function $f$. 
\end{proof}
Hence, it makes sense to define $f(1^-,\cdot):[-\pi, \pi]\rightarrow(0,\infty]$ as follows:
\begin{equation}\label{def-function-f}
 f(1^-,\theta)=\lim_{r\rightarrow 1^-}f(r,\theta)=\int_{-\pi}^\pi \frac{1}{|1-e^{i(\theta+x)}|^2}d\bar{\mu}(e^{ix})=\int_{-\pi}^\pi \frac{1}{|1-e^{i(\theta-x)}|^2}d\mu(e^{ix}). 
\end{equation}
Then the function $T:\mathbb{\mathbb{C}}\rightarrow [0,\infty)$ defined in \eqref{Tlambda} is expressed by $f$ as follows
\begin{equation}\label{Tlambda-f}
T(\lambda) = 
\begin{cases}
1/f(r,\theta), & \text{for } \lambda=re^{i\theta}\quad \text{and}\quad r\neq 1 \\
1/f(1^-,\theta), & \text{for } \lambda=re^{i\theta}\quad \text{and}\quad r=1. 
\end{cases}
\end{equation}
If $f(1^-, \theta) =\infty$, we understand that, for this $\theta$, $T(e^{i\theta}) = 0$. Analogous to the additive case (See~\eqref{lifetimeT}), the function $T$ may not be continuous on $\mathbb{T}$ as mentioned in Remark~\ref{rem:disconT}. If $\mu = \delta_1$, one can use \eqref{Tlambda-f} to check that the function $T$ reduces to the one defined in \cite{DHKBrown}. 

\begin{lemma}\label{lemma:T}
	For each $\theta\in [-\pi,\pi]$, the function $r\mapsto T(re^{i\theta})$ is strictly decreasing for $0<r<1$ and strictly increasing for $r>1$. For each $\theta$, the minimum value of $T(re^{i\theta})$ is achieved at $r=1$, which is $1/f(1^-,\theta)$, and we have
	\[
	\lim_{r\rightarrow 0}T(re^{i\theta})=\lim_{r\rightarrow \infty}T(re^{i\theta})=\infty. 
	\]
Moreover, $T(re^{i\theta})=T((1/r)e^{i\theta})$. 
\end{lemma}
\begin{proof}
	The claim follows from Corollary \ref{cor:limit} and the reciprocal relation \eqref{Tlambda-f}. 
\end{proof}

We now can use the functions $r_t$ defined in (\ref{def:r_t_theta}) and $T$ to describe the set $\Ot$ and its boundary. 
\begin{proposition}\label{characterization-2}
Let $\theta\in [-\pi,\pi]$.
\begin{enumerate}
\item When $r_t(\theta)<1$, we have $T(r_t(\theta)e^{i\theta})=t$. Moreover, 
   \begin{enumerate}
      \item for $r<r_t$, we have $T(r e^{i\theta})>t$;
      \item for $r_t<r\leq 1$, we have $T(r e^{i\theta})<t$.
   \end{enumerate}
\item When $r_t(\theta)=1$, we have $T(r_t(\theta)e^{i\theta})=T(e^{i\theta})\geq t$, and
$T(re^{i\theta})>t$ for $r<1$. If, in addition, $\theta\in [-\pi,\pi]\backslash \overline{U_{t,\bar{\mu}}}$, then $T(re^{i\theta})>t$ for all $r>0$.
\end{enumerate}
\end{proposition}
\begin{proof}
When $r_t(\theta)<1$, by Part (2) of Proposition \ref{characterization-0}, we have 
\[
	      \frac{r_t(\theta)^2-1}{2\log r_t(\theta)}\int_{-\pi}^{\pi}\frac{1}{|1-r_t(\theta)e^{i(\theta+x)}|^2}d\bar{\mu}(e^{ix}) 
	= \frac{1}{t},
\]
which implies that $T(r_t(\theta)e^{i\theta})= t$ by the definition \eqref{Tlambda} of $T$. As for each $\theta$, the function $r\mapsto T(re^{i\theta})$ is strictly decreasing for $0<r<1$ by Lemma \ref{lemma:T}. Items (1a) and (1b) follows. 

By the definition of $r_t$ in (\ref{def:r_t_theta}), when $r_t(\theta)=1$, we have
\[
    \frac{r^2-1}{2\log r}\int_{-\pi}^{\pi}\frac{1}{|1-re^{i(\theta+x)}|^2}d\bar{\mu}(e^{ix}) 
  \leq \frac{1}{t}
\]
for all $0\leq r\leq r_t(\theta)=1$.
Hence 
$T(r_t(\theta)e^{i\theta})\geq t$, and $T(\lambda)>t$ when $|\lambda|<1$. 
If, in addition, $\theta\in [-\pi,\pi]\backslash \overline{U_{t,\bar{\mu}}}$, 
then by \eqref{Ut-00}, 
\[
\int_{-\pi}^{\pi}\frac{1}{|1-e^{i(\theta+x)}|^2}d\bar{\mu}(e^{ix})<\frac{1}{t},
\]
which yields that $T(e^{i\theta})>t$ due to \eqref{Tlambda}. Therefore, $T(re^{i\theta})>t$ for all $0<r\leq 1$ and $\theta\in [-\pi,\pi]\backslash \overline{U_{t,\bar{\mu}}}$.
\end{proof}

\begin{proposition}\label{characterization-U-t-T}
The open set $U_{t,\bar{\mu}}$ defined in (\ref{Ut-00}) may be characterized by
\begin{align}
U_{t,\bar{\mu}}
&=\{-\pi\leq \theta\leq\pi: r_t(\theta)<1 \}\label{Ut-2}\\
&=\{-\pi\leq \theta\leq\pi: T(e^{i\theta})<t \}\label{Ut-1}.
\end{align}
Moreover, $\partial\Ot\cap \bT=\bT \backslash \{e^{i\theta}: \theta\in U_{t,\bar{\mu}}\}$, where $\Ot$ is defined in \eqref{OmegaDef}. 
\end{proposition}
\begin{proof} 
We rewrite the definition of $U_{t,\bar\mu}$ in \eqref{Ut-00} in terms of the function $f$ as	
\[
   U_{t,\bar{\mu}}=\left\{ -\pi\leq \theta\leq \pi: f(1^-,\theta)>\frac{1}{t} \right\}.
\]
Hence, by Proposition \ref{cor:limit}, $\theta\in U_{t,\bar{\mu}}$ if and only if $r_t(\theta)<1$. 
From Proposition \ref{characterization-2}, we see that $r_t(\theta)<1$ if and only if $T(e^{i\theta})<t$. 
Now, \eqref{Ut-2} and \eqref{Ut-1} hold. 

Recall from Proposition \ref{characterization-0}, 
\[
    \partial\Ot=\{r_t(\theta)e^{i\theta}: -\pi\leq \theta\leq \pi \},
\]
which yields that 
\[
\partial\Ot\cap \bT=\{ e^{i\theta}: r_t(\theta)=1\}= \bT \backslash\{e^{i\theta}: \theta\in U_{t,\bar{\mu}}\}.
\]
This concludes the proof.
\end{proof}

We are now ready to prove our main result in this section. 
\begin{proof}[{\bf Proof of Theorem \ref{Dt-characterization}}]
Recall from Proposition \ref{characterization-0} that $\Ot=\{re^{i\theta}: 0\leq r< r_t(\theta),\theta\in[-\pi,\pi] \}$. Hence, by (\ref{Ut-2}) in Proposition \ref{characterization-U-t-T}, we have
\begin{align*}
  \overline{\mathbb{D}}\backslash \overline{\Ot}&=\{re^{i\theta}: r_t(\theta)< r\leq 1, r_t(\theta)<1  \}\\  
     &=\{re^{i\theta}: r_t(\theta)< r\leq 1, \theta\in U_{t, \bar{\mu}} \}
\end{align*}
which yields (\ref{eqn:Dt-1}).

We discuss the function $T(\lambda)$ according to the argument of $\lambda$ in the following cases:
\begin{enumerate}
\item When $\theta\in U_{t,\bar{\mu}}$, 
for $\lambda$ with $\arg(\lambda)=\theta$ and $r_t(\theta)<|\lambda|\leq 1$, we know that $T(\lambda)<t$ from (1b) of 
Proposition \ref{characterization-2}. 
Hence, $\Dt\subset\left\{    \lambda\in\mathbb{C}: T(\lambda
)<t\right\}$ by the definition of the set $\Dt$ and the symmetric property $T(z)=T(1/\overline{z})$ as in Lemma \ref{lemma:T}.
Since $r\mapsto T(re^{i\theta})$ is continuous and decreasing for $0<r<1$, 
 it follows that $$r_t(\theta)e^{i\theta}\in \partial \Ot \cap \partial \Dt$$ and $T(r_t(\theta)e^{i\theta})=t$ by Proposition \ref{characterization-2}. Moreover,  
$\lambda\notin \overline{\Dt}$ and $T(\lambda)>t$ if $|\lambda|<r_t(\theta)$ or $|\lambda|>1/r_t(\theta)$.
\item When $\theta\in [-\pi, \pi]\backslash \overline{ {U_{t,\bar{\mu}}}}$, we know that $T(\lambda)> t$ for all $\lambda$ in the ray with angle $\theta$  by Proposition \ref{characterization-2}. It is also clear that $\lambda\notin \overline{\Dt}$ for such $\lambda$.
\item Finally, when $\theta\in \partial U_{t,\bar{\mu}}$, we have $r_t(\theta)=1$, $e^{i\theta}\in \partial \Ot\cap \partial \Dt$ and $T(e^{i\theta})\geq t$. 
Since $r\mapsto T(re^{i\theta})$ is decreasing for $0<r<1$, and increasing for $r>1$, 
we have 
$\lambda\notin \overline{\Dt}$
and $T(\lambda)>t$ for any other $\lambda$ that is in the ray with angle $\theta$ except $e^{i\theta}$.
\end{enumerate}

By the above discussions, we see that $T(\lambda)>t$ if $\lambda\notin \overline{\Dt}$ for any $\theta$. In addition,  
$T(\lambda)\geq t$ if $\lambda\in \partial\Dt$, and $T(\lambda)=t$ if $\lambda \in \partial \Dt\backslash \mathbb{T}$. 
Therefore, $\lambda\in \Dt$ if and only if $T(\lambda)<t$. 
\end{proof}

\begin{remark}
For $\lambda\in \partial \Dt\cap \mathbb{T}$ (i.e., $\arg(\lambda)\in [-\pi,\pi]\backslash \overline{U_{t,\bar{\mu}}}$), it is not always true that $T(\lambda)=t$ as discussed in Remark~\ref{rem:disconT} (recall that $T(\lambda) = t_\ast(\lambda,0)$). When this is true, then
the boundary of $\Dt$ may be expressed as
\[
\partial\Dt=\left\{  \left.  \lambda\in\mathbb{C}\right\vert
T(\lambda)=t\right\} .
\]
\end{remark}

We end this section with the bound of $d\phi/d\theta$, where $\phi$ is defined in (\ref{thetaPhi}), which will be used to give a bound of the density of the Brown measure of $ub_t$ (see Theorem \ref{sLaplacian.main}). We first establish the following lemma that will be used in establishing the bounds.

\begin{lemma}
	\label{lem:rfunctionInc}
	The function
	\[r\mapsto\frac{-r\log r}{1-r^2}\]
	is strictly increasing for $0<r<1$.
\end{lemma}
\begin{proof}
	We will prove that $r\mapsto(1-r^2)/(-r\log r)$
	for $0<r<1$ is strictly decreasing, which is, by putting $x=-\log r$, if and only if the function
	\[x\mapsto\frac{1-e^{-2x}}{xe^{-x}} = \frac{2\sinh x}{x}\]
	for $x>0$ is strictly increasing. Since the Taylor coefficients of
	\[\frac{2\sinh x}{x} = \sum_{n=0}^\infty \frac{x^{2n}}{(2n+1)!}\]
	are all nonnegative, the function $x\mapsto 2\sinh x/x$ is strictly increasing.
\end{proof}

Recall from the discussion in Section~\ref{section:multiplicative_convolution} that $\omega_t$ is a function on $\bar{\bbD}$ whose left inverse is the function $\Phi_{t,\bar\mu}$. The function $\phi$ defined using the formula in (\ref{thetaPhi}) is a function on $\R$, and is differentiable at $\theta$ where $r_t(\theta)<1$. The next lemma gives an upper and a lower bound for $d\phi/d\theta$.

\begin{lemma}
	\label{lemma:positive-derivative-phi-theta}
	Given any $a\in\R$, the function $\phi$ defined in (\ref{thetaPhi}) is a strictly increasing continuous version of $\arg \Phi_{t,\bar\mu}(r_t(\theta)e^{i\theta})$ for any $\theta\in[a,a+2\pi)$ onto some interval of the form $[b,b+2\pi)$ for some $b\in\R$. If $r_t(\theta)<1$, $\phi$ is differentiable at $\theta$, and we have
	\[0<\frac{d\phi}{d\theta}<2.\]
\end{lemma}
\begin{proof}
	The first assertion follows from the paragraph following (\ref{thetaPhi}). Clearly, $\phi$ is differentiable at $\theta\in\R$ where $r_s(\theta)<1$, since $\mu$ is a measure on the unit circle.
	
	Write $w = r_s(\theta)e^{i\theta}$. Recall from Remark~\ref{rem:rtanalyticity} that $r_t$ is analytic at $\theta$ where $r_t(\theta)<1$. We take the derivative with respect to $\theta$ of the identity $e^{i\phi} = \Phi_{t,\bar\mu}(r_t(\theta)e^{i\theta})$; we have
	\begin{align*}
		ie^{i\phi}\frac{d\phi}{d\theta}& = \Phi_{t,\bar\mu}'(r_t(\theta)e^{i\theta})(r_t'(\theta)e^{i\theta}+ir_t(\theta)e^{i\theta})\\
		&= i \Phi_{t,\bar\mu}'(w) w\left(-i\frac{r_t'(\theta)}{r_t(\theta)}+1\right).
	\end{align*}
	If we divide both sides by $i\Phi_{t,\bar\mu}'(w)w$ and recall that $e^{i\phi}=\Phi_{t,\bar\mu}(w)$,
	\[\frac{\Phi_{t,\bar\mu}(w)}{w\Phi_{t,\bar\mu}'(w)}\frac{d\phi}{d\theta} = -i\frac{r_t'(\theta)}{r_t(\theta)}+1.\]
	Since $d\phi/d\theta$ is real, we may take the real part of both sides to get
	\[\mathrm{Re}\left[\frac{\Phi_{t,\bar\mu}(w)}{w\Phi_{t,\bar\mu}'(w)}\right]\frac{d\phi}{d\theta} = 1.\]
	Therefore,
	\begin{equation}
		\label{eq:dphidtheta}
		\frac{d\phi}{d\theta} = \frac{1}{\mathrm{Re}\left[\frac{\Phi_{t,\bar\mu}(w)}{w\Phi_{t,\bar\mu}'(w)}\right]}.
	\end{equation}
	We will show that $0<d\phi/d\theta<2$ by estimating $\mathrm{Re}\left[\frac{\Phi_{t,\bar\mu}(w)}{w\Phi_{t,\bar\mu}'(w)}\right]$.
	
	We compute 
	\begin{align*}
	\frac{w\Phi_{t,\bar\mu}'(w)}{\Phi_{t,\bar\mu}(w)} = 1+\frac{t}{2}\int\frac{2w\xi}{(\xi-w)^2}\,d\mu(\xi),
	\end{align*}
	so that if $r_t(\theta)<1$, we have
	\begin{equation}
		\label{eq:Lessthan2}
		\begin{split}
			\left\vert 1-\frac{w\Phi_{t,\bar\mu}'(w)}{\Phi_{t,\bar\mu}(w)} \right\vert &= \left\vert\frac{t}{2}\int\frac{2w\xi}{(\xi-w)^2}\,d\mu(\xi)\right\vert\\
			&\leq t r_t(\theta)\int\frac{1}{\left\vert \xi-w\right\vert^2}\,d\mu(\xi)\\
			&=\frac{-2r_t(\theta)\log(r_t(\theta))}{1-r_t(\theta)^2}<1,
		\end{split}
	\end{equation}
	where the last inequality follows from $r_t(\theta)<1$ and that $r\mapsto -2\log r/(1-r^2)$ is strictly increasing by Lemma~\ref{lem:rfunctionInc}.
	
	The M\"obius transform $z\mapsto 1/z$ maps $2$, $1+i$, $1-i$ to the vertical line with real part equal to $1/2$; it also maps $1$ to $1$. Thus, it maps the unit disk centered at $1$ to $\left.\{z\in\C\right\vert \mathrm{Re}\,z>1/2\}$. We conclude by (\ref{eq:Lessthan2}) that
	\begin{equation*}
		\mathrm{Re}\left[\frac{\Phi_{t,\bar\mu}(w)}{w\Phi_{t,\bar\mu}'(w)} \right] > \frac{1}{2}.
	\end{equation*}
	By~(\ref{eq:dphidtheta}), the above display equation shows that $0<d\phi/d\theta<2$.
\end{proof}

\subsection{Surjectivity}

The goal of this subsection is to prove the following surjectivity result
 which extends 
the result in \cite[Theorem 6.17]{DHKBrown} to our general setting. This result allows us to find unique initial conditions $\lambda_0\in \Dt$ and $\varepsilon_0>0$ for each $\lambda\in\Dt$ such that $\lim_{u\uparrow t}\lambda(u) = \lambda$ and $\lim_{u\uparrow t}\varepsilon(u) = 0$. These initial conditions will help us apply the Hamilton--Jacobi formulas to compute $S$ and $\Delta S$ at $\varepsilon = 0$.

\begin{theorem}\label{thm:surjectivity}
	\label{surjectivity} Given $t>0$, for all $\lambda\in\Dt,$ there
	exist unique $\lambda_{0}\in\mathbb{C}$ and $\varepsilon_{0}>0$ such that the
	solution to (\ref{theODEs}) with these initial conditions exists on $[0,t)$ where $t=t_\ast(\lambda_0, \varepsilon_0)$, 
	with $\lim_{u\uparrow
 t}\varepsilon(u)=0$ and $\lim_{u\uparrow
 t}%
	\lambda(u)=\lambda.$ For all $\lambda\in\Dt,$ the corresponding
	$\lambda_{0}$ also belongs to $\Dt.$
	
	Define functions $\lambda_{0}^{t}:\Dt\rightarrow\Dt$ and
	$\varepsilon_{0}^{t}:\Dt\rightarrow(0,\infty)$ by letting $\lambda_{0}%
	^{t}(\lambda)$ and $\varepsilon_{0}^{t}(\lambda)$ be the corresponding values of
	$\lambda_{0}$ and $\varepsilon_{0}$ for $\lambda\in\Dt$ respectively. Then $\lambda_{0}^{t}$ and $\varepsilon_{0}%
	^{t}$ extend to continuous maps from $\overline{\Dt}$ into $\overline
	{\Dt}$ and $[0,\infty)$, respectively, with the continuous extensions
	satisfying $\lambda_0^{t}(\lambda)=\lambda$ and $\varepsilon_{0}^{t}(\lambda)=0$ for
	$\lambda$ in the boundary of $\Dt.$
\end{theorem}

Given $\lambda_0\in \Dt$, we know from Corollary~\ref{smallx0.cor} that $t_*(\lambda_0,0)<t$. On the other hand, by Proposition~\ref{monotoneTstar.prop},  the function $\varepsilon\mapsto t_*(\lambda_0, \varepsilon_0)$ is strictly increasing and $\lim_{\varepsilon_{0}\rightarrow+\infty}t_{\ast}(\lambda_{0},\varepsilon_{0})=+\infty$. Hence, there is a unique value $\varepsilon_0^t(\lambda_0)$ such that $t_*(\lambda_0, \varepsilon_0^t(\lambda_0))=t$.
We set 
\[
    \lambda_t(\lambda_0)=\lim_{u\uparrow
 t}\lambda (u)
\]
where $\lambda(u)$ is computed with initial conditions $\lambda(0)=\lambda_0$ and $\varepsilon(0)=\varepsilon_0^t(\lambda_0)$. This defines a map
\[
     \varepsilon_0^t: \Dt \rightarrow [0,\infty), \qquad \text{by} \qquad \lambda_0\mapsto \varepsilon_0^t(\lambda_0);
\]
and
\[
   \lambda_t: \Dt \rightarrow \Dt, \qquad \text{by} \qquad \lambda_0\mapsto \lambda_t(\lambda_0).
\]

\begin{proposition}\label{extension-1}
The function $\varepsilon_{0}^{t}$ extends continuously from $\Dt$ to $\overline{\Dt}$. The extended map satisfies $\varepsilon_{0}^{t}(\lambda_{0})=0$ for $\lambda_{0}\in\partial\Dt.$
Moreover, the function $\varepsilon_0^t$ is expressed as
\[
 \varepsilon_{0}^{t}(\lambda_{0})=r_0\left(  \frac
{r_t(\theta)^{2}+1}{r_t(\theta)}-\frac{r_0^{2}+1}{r_0 }\right).
\]
\end{proposition}
The formula of $\varepsilon_0^t(\lambda_0)$ is similar to \cite[Eq. (6.67)]{DHKBrown}.
\begin{proof}
Our strategy is to find an explicit formula of $\varepsilon_0^t(\lambda_0)$ in terms of $r_0=|\lambda_0|$ and $\theta$. Recall that $\theta=\arg \lambda_0$, and that $r_t(\theta)e^{i\theta}\in \partial \Dt \cap \mathbb{D}$ satisfies $T(r_t(\theta)e^{i\theta})=t$ by Theorem \ref{Dt-characterization}. 
In light of Proposition \ref{life-time}, we have $t_*({r_t(\theta)e^{i\theta}}, 0)=t$. Thus, we have the following relation between $f_{\lambda_0}$ and $f_{r_t(\theta)e^{i\theta}}$
\begin{equation}\nonumber
f_{\lambda_0}(\delta_1)
=\int_{-\pi}^\pi \frac{1}{g_{\theta_0-\alpha}(\delta_1)}d\mu(e^{i\alpha})=\frac{1}{t}
  =\int_{-\pi}^\pi \frac{1}{g_{\theta_0-\alpha}(\delta_2)}d\mu(e^{i\alpha})
  =f_{r_t(\theta)e^{i\theta}}(\delta_2),
\end{equation}
where 
\begin{equation}
\nonumber
  \delta_1= \frac{r_0^2+1+\varepsilon_0^t(\lambda_0)}{r_0}, \qquad \text{and} \qquad
    \delta_2=\frac{r_t(\theta)^2+1}{r_t(\theta)}.
\end{equation}
This implies that we must have $\delta_1=\delta_2$ thanks to Proposition \ref{monotoneTstar.prop}. That is
\begin{equation}\label{delta-choice}
    \delta= \frac{r_0^2+1+\varepsilon_0^t(\lambda_0)}{r_0}=
    \frac{r_t(\theta)^2+1}{r_t(\theta)}.
\end{equation}
Hence, $\varepsilon_0^t(\lambda_0)$ is expressed as
\begin{equation}\label{x_0_t}
\varepsilon_{0}^{t}(\lambda_{0})=r_0\left(  \frac
{r_t(\theta)^{2}+1}{r_t(\theta)}-\frac{r_0^{2}+1}{r_0 }\right).
\end{equation}
The desired result then follows.
\end{proof}

\begin{proposition}\label{extension-2}
The function $\lambda_{t}$ extends continuously from $\Dt$ to $\overline{\Dt}$, with the extended maps satisfying $\lambda_t(\lambda_{0})=\lambda_0$ for $\lambda_{0}\in\partial\Dt.$
The extended map $\lambda_{t}$ is a homeomorphism of $\overline{\Dt}$
to itself.
\end{proposition}
\begin{proof}
To find the formula for $\lambda_t(\lambda_0)$, we use the facts that $t_*(\lambda_0, \varepsilon_0^t(\lambda_0))=t$ and that the argument $\theta$ of $\lambda$ remains unchanged along the path. We apply (\ref{lambda-life-time}) to obtain
\begin{equation}\label{lambda-t-0}
 \lambda_t(\lambda_0)=\frac{\lambda_0}{|\lambda_0|} e^{Ct/2}.
\end{equation}
By the formula for $C$ given by (\ref{eqn:C}) and the formula for $t_*(\lambda_0, \varepsilon_0)$ given 
by (\ref{tStar2}), we have
\begin{equation}\label{CtStar}
Ct=Ct_*(\lambda_0, \varepsilon_0^t(\lambda_0))
  =\frac{\delta-2/|\lambda_0|}{\sqrt{\delta^2-4}}\log \left( \frac{\delta+\sqrt{\delta^2-4}}{\delta-\sqrt{\delta^2-4}} \right),
\end{equation}
where $\delta=(r_t(\theta)^2+1)/r_t(\theta)$ as in \eqref{delta-choice}. We remark that the formula of $\lambda_t(\lambda_0)$ is similar to \cite[Eq. (6.64)]{DHKBrown}.

By \eqref{delta-choice}, $\delta$ only depends on $\theta$ but not $r_0$; hence, $Ct$ is strictly increasing in $r_0$ by \eqref{CtStar}. From the expression (\ref{lambda-t-0}), we then deduce that, fixing a $\theta\in U_{t,\mu}$, the function $|\lambda_0|\mapsto |\lambda_t(\lambda_0)|$ is a strictly increasing function for $\lambda_0\in \{ re^{i\theta}: r_t(\theta)<r<1/r_t(\theta) \}$. Moreover, when $\lambda_0=r_t(\theta)e^{i\theta}$, we have
\[
  \lambda_t(r_t(\theta)e^{i\theta})=e^{i\theta}e^{Ct/2}=r_t(\theta)e^{i\theta}
\]
and $\lambda_t(e^{i\theta}/r_t(\theta))=e^{i\theta}/r_t(\theta)$
due to (\ref{lambda-life-time}) and (\ref{rhoZero}) (which can also be verified directly). In other words,
$\lambda_t$ maps the interval $\{ re^{i\theta}: r_t(\theta)<r<1/r_t(\theta) \}$ bijectively to itself and fixes the endpoints. Since this holds for any $\theta\in U_{t,\bar{\mu}}$, we hence conclude that $\lambda_t$ maps $\overline{\Dt}\setminus(\partial \Dt\cap \mathbb{T})$ bijectively to itself and fix any $\lambda\in (\partial\Dt\setminus \mathbb{T})$. As $\lambda_t$ is continuous and $r_t(\theta)\to 1$ as $e^{i\theta}$ approaches $\partial \Lambda_{t,\mu}\cap \mathbb{T}$, we then conclude that inverse of $\lambda_t$ extends to a homeomorphism on $\overline\Dt$
\end{proof}

\begin{proof}[{\bf Proof of Theorem \ref{thm:surjectivity}}]
 It is a direct consequence of Propositions \ref{extension-1} and \ref{extension-2}.
\end{proof}


\subsection{The Brown measure of $g_t$ and its connection to $uu_t$}
We will show in Theorem \ref{sLaplacian.main} that the Brown measure of $ub_t$ is supported on the closure $\overline{\Dt}$ by showing that the Brown measure has mass $1$ on $\Dt$. We prove this by first showing that a certain measure on $\Dt$ can be pushed forward to the spectral distribution of $uu_t$ in Lemma \ref{ProbBrownMul}; hence this measure is a probability measure. Then, in Theorem \ref{sLaplacian.main}, we show that the Brown measure has the same formula on $\Dt$ as the one defined in Lemma \ref{ProbBrownMul}. It then follows that the Brown measure has mass $1$ on $\Dt$. Thus, it suffices to focus on the computation for $\lambda\in \Dt$ in the section. 

Our first goal is to calculate the Laplacian of the function $s_t(\lambda)$ defined in Section~\ref{sect:TDE} for fixed $t$ by using Hamilton--Jacobi analysis. Recall that
\[
s_t(\lambda):=\lim_{\varepsilon\downarrow
 0} S(t,\lambda, \varepsilon)
\]
for $t$ fixed. 
By Theorem \ref{surjectivity}, for each $t>0$ and $\lambda\in\Dt$, we can choose $\varepsilon_0>0$ and $\lambda_0\in\Dt$ such that $\lim_{u\rightarrow t}\varepsilon(u)=0$, $\lim_{u\rightarrow t}\lambda(u)=\lambda$ where $t=t_*(\lambda_0, \varepsilon_0)$. 
Moreover, as the argument of $\lambda$ remains unchanged along the flow by Proposition \ref{conservation}, we have $\arg \lambda=\arg\lambda_0$.
We note that the Laplacian in logarithmic polar coordinates has the form
\[
\Delta=\Delta_\lambda
=\frac{1}{r^2}\left( \frac{\partial^2}{\partial \rho^2} +\frac{\partial^2}{\partial \theta^2}\right),
\]
where we recall that $r=|\lambda|$, $\rho=\log|\lambda|$, and $\theta=\arg \lambda$. 

We consider the function on $U_{t,\bar{\mu}}=\Delta_{t,\mu}\cap\mathbb{T}$ defined by
\begin{equation}\label{equation:wt-mt-derivative}
w_t(\theta)= \frac{1}{4\pi}\left(\frac{2}{t}+\frac{d}{d\theta}m_{t}(\theta)\right)
\end{equation}
where
\begin{equation}
	m_t(\theta)=	\int_{-\pi}^{\pi}  \frac{2r_t(\theta)\sin(\theta-\alpha)}{r_{t}(\theta)^{2}+1-2r_{t}(\theta
		)\cos(\theta-\alpha)}d\mu(e^{i\alpha})\label{function_mt}.
\end{equation}

Recall that $r_t(\theta)$ is defined in $(\ref{def:r_t_theta})$, which is the smaller one of the radii  
where the ray with angle
$\theta$ intersects the boundary of $\Dt$, and $\mu$ is the spectral distribution of $u$.
The functions $w_t$ and $m_t$ will play a main role in computing the Brown measure of $ub_t$.

\begin{proposition}
	\label{dPhiDtheta.prop}
	For any $\theta\in \Delta_{t,\mu}\cap\mathbb{T}$, the function $w_{t}$ may be computed as%
	\begin{equation}
		w_{t}(\theta)=\frac{1}{2\pi t}\frac{d\phi}{d\theta}, \label{wtFromPhi}%
	\end{equation}
	where $\phi = \arg(\Phi_{t,\bar\mu}(r_t(\theta)e^{i\theta}))$ is defined by (\ref{thetaPhi}).
\end{proposition}

\begin{proof}
	Recall that by \eqref{thetaPhi}, $\phi$ is given by 
	\begin{equation}\nonumber
 \phi=\theta+t\int_{-\pi}^\pi \frac{r_{t}(\theta) \sin(\theta-x)}{|1-r_t(\theta)e^{i(\theta-x)}|^2}d\mu(e^{ix}).
	\end{equation}
	Hence,
	\begin{align*}
		d\phi &=\left(  1+t\frac{d}{d\theta}\int_{-\pi}^{\pi}
		\frac{r_t(\theta)\sin(\theta-x)}{|1-r_t(\theta)e^{i(\theta-x)}|^2}d\mu(e^{ix})\right)  d\theta\\
		&= \left(1+\frac{t}{2}\frac{d}{d\theta}m_t(\theta)\right)d\theta\\
		&=2\pi t w_t(\theta)d\theta. 
	\end{align*}
	and the formula (\ref{wtFromPhi}) follows.
\end{proof}

We next show that the measure $w_t(\theta)/r^2$ has mass $1$ on $\Dt$. Later, we will show that $w_t(\theta)/r^2$ is indeed the density of the Brown measure of $ub_t$ inside $\Dt$. That it has mass $1$ on $\Dt$ will be used to show that the Brown measure is $0$ outside $\Dt$.
\begin{lemma}
	\label{ProbBrownMul}
We have 
	\[
	\int_{\Dt} \frac{1}{r^2} w_t(\theta) r\,dr\,d\theta=1.
	\]
\end{lemma}
\begin{proof}
	Using the characterization of $\Dt$ in Theorem \ref{Dt-characterization}, we have 
	\begin{align*}
		\int_{\Dt}\frac{1}{r^2} w_t(\theta) r\,dr\,d\theta &=\int_{U_{t,\bar{\mu}}}\int_{r_t(\theta)}^{1/r_t(\theta)}\frac{1}{r^2} w_t(\theta) r\,dr\,d\theta\\
		&=\int_{U_{t,\bar{\mu}}} -2\log (r_t(\theta)) w_t(\theta)\,d\theta\\
		&=\int_{U_{t,\bar{\mu}}} -2\log (r_t(\theta)) \frac{1}{2\pi t}\frac{d\phi}{d\theta}\,d\theta\quad\textrm{(by Proposition \ref{dPhiDtheta.prop})}\\
		&=\int_{\mathbb{T}} -\frac{1}{\pi t}\log (r_t(\theta))\,d\phi\\
		&=\int_{\mathbb{T}} p_t(e^{i\phi})\,d\phi\quad \qquad\qquad \qquad\textrm{(by Theorem \ref{thm:mut})}\\
		&=1.
	\end{align*}
The lemma is established.
\end{proof}

Given any $\lambda\in\Dt$, by Theorem \ref{surjectivity}, there are unique $\lambda_0$ and $\varepsilon_0$ such that
$$\lim_{u\to t}(u, \lambda(u), \varepsilon(u)) = (t, \lambda, 0).$$
We attempt to compute the Brown measure by the limit $(1/4\pi)\Delta S(u, \lambda(u), \varepsilon(u))$ as $u\uparrow t$. However, the definition of the Brown measure of $ub_t$ is
$$\frac{1}{4\pi}\Delta\left( \lim_{\varepsilon\downarrow
 0}S(t, \lambda, \varepsilon)\right),$$
with $t$ and $\lambda$ fixed in the limiting process.
We want to show that this limit is the same as the limit 
\[\lim_{u\uparrow t}\frac{1}{4\pi}\Delta S(u, \lambda(u), \varepsilon(u)).\] 
To achieve this, we want to show that the limit is independent of path approaching $(t,\lambda,\varepsilon)$. We will see that the analogue of \cite[Theorem 7.4]{DHKBrown} holds for our $S$. More precisely, we want to see that, given any $(\sigma, \omega)\in\R^+\times \Dt$, the function
\begin{equation}
\label{eq:originStilde}
\tilde{S}(t,\lambda, z) = S(t,\lambda, z^2), \quad z>0
\end{equation}
extends to a real analytic function in a neighborhood of $(\sigma, \lambda, 0)$ inside $\R\times \C\times \R$. The key here is that the reglarity holds even in the triple $(t, \lambda, z)$, not just in the pair $(\lambda, z)$; the Laplacian of $S$ at $(t,\lambda, 0)$ is then equal to the limit of the Laplacian along the path $(u, \lambda(u), \varepsilon(u))$ since there is no partial derivative with respect to $\varepsilon$ involved. We will give the main lines below. For more details, readers are encouraged to read \cite[Section 7.4]{DHKBrown}.

\begin{theorem}
	\label{thm:regularity}
	The function $\tilde{S}$ extends to a real anaytic function in a neighborhood of $(t, \lambda, 0)$ in $\R^+\times \Dt$.
	\end{theorem}
\begin{remark}
	The function $S$ itself does not have a smooth extension of the same sort that $\tilde{S}$ does. Indeed, using that $\sqrt{\varepsilon}p_\varepsilon$ is a constant of motion, the second Hamilton--Jacobi formula~\eqref{eq:2ndHJ} tells us that $\partial S/\partial\varepsilon$ must blow up like $1/\sqrt{\varepsilon}$ as we approach $(t,\lambda,0)$ along a solution of the ODEs. The same reasoning tells us that the extended $\tilde{S}$ does not satisfy $\tilde{S}(t,\lambda,z) = S(t,\lambda,z^2)$ for $z<0$. Because $\sqrt{\varepsilon}p_\varepsilon$ is a constant of motion, 
	\[\frac{\partial\tilde{S}}{\partial z}(t,\lambda,z) = 2\sqrt{\varepsilon}\frac{\partial S}{\partial\varepsilon}(t,\lambda,z^2)\]
	has a nonzero limit as $z\to 0$. Thus, $\tilde{S}$ cannot have a smooth extension that is even in $z$.
	\end{remark}
\begin{proof}
	Denote by $\lambda(t;\lambda_0,\varepsilon_0)$ the solution of $\lambda$-variable  of the ODEs at time $t$ given initial conditions $\lambda_0$ and $\varepsilon_0$. Write $z(t; \lambda_0, \varepsilon_0) = \sqrt{\varepsilon(t;\lambda_0,\varepsilon_0)}$ where $\varepsilon(t;\lambda_0,\varepsilon_0)$ is the solution of $\varepsilon$-variable of the ODEs at time $t$ given initial conditions $\lambda_0$ and $\varepsilon_0$.
	
	Define the map
	$$V(t, \lambda_0, \varepsilon_0) = (t, \lambda(t;\lambda_0, \varepsilon_0), z(t; \lambda_0, \varepsilon_0)).$$
	We first show that $V$ can be extended to $t\in\R$, given any $\lambda_0, \varepsilon_0$. The main observation is that, given any $\varepsilon_0>0$, $1/p_\varepsilon(t; \lambda_0, \varepsilon_0)$ extends to a real analytic function for all $t\in\R$ (See~\eqref{pxFormula} for the formula of $p_\varepsilon$). Given any $\lambda_0\in \Dt$ and $\varepsilon_0>0$, we can extend $z(t;\lambda_0,\varepsilon_0)$ by the same formula
\begin{equation}\label{eqn:Zt}
z(t; \lambda_0, \varepsilon_0) = \sqrt{\varepsilon_0}p_0e^{-Ct/2}\frac{1}{p_\varepsilon(t; \lambda_0, \varepsilon_0)}
\end{equation}
to \emph{all} $t\in\R$. By~\eqref{eqn:Zt}, $z(t;\lambda_0, \varepsilon_0)^2 = \varepsilon(t;\lambda_0, \varepsilon_0)$ and so $z(t; \lambda, \varepsilon_0) = 0$ when $t = t_*(\lambda_0, \varepsilon_0)$. The function $z(t; \lambda_0, \varepsilon_0)$ is positive when $t<t_*(\lambda_0, \varepsilon_0)$; it is negative when $t>t_*(\lambda_0, \varepsilon_0)$. Since, $1/p_\varepsilon(t; \lambda_0, \varepsilon_0)$ extends to a real analytic function for all $t\in\R$, by (\ref{logLambdaIntegral})
and Proposition \ref{xpx}, 
$$\lambda(t) = \lambda_0\exp\left(\int_0^t \varepsilon(s)p_\varepsilon(s)\,ds\right)= \lambda_0\exp\left(\int_0^t \frac{\varepsilon_0p_0^2e^{-Cs}}{p_\varepsilon(s)}ds\right)$$
also extends to an analytic function of $t\in\R$. 

Fix any $\lambda\in\Dt$. Next, we show that we can extend $\tilde{S}$ to a neighborhood of $(t, \lambda, 0)$, where $\lambda\in \Dt$, by applying the inverse function theorem to the map
$$V(t, \lambda_0, \varepsilon_0) = (t, \lambda(t;\lambda_0, \varepsilon_0), z(t; \lambda_0, \varepsilon_0))$$
which was shown to be extended to $t\in\R$ in the preceding paragraph. Once we show that the Jacobian matrix of $V$ at $(t, \lambda_0^t(\lambda), \varepsilon_0^t(\lambda))$ is invertible (See Theorem~\ref{surjectivity} for the definitions of $\lambda_0^t$ and $\varepsilon_0^t$), there exists a local inverse $V^{-1}$ defined around a neighborhood of $(t, \lambda, 0)$ which satisfies
\begin{equation}
\label{S:tilde}
\tilde{S}(\sigma, \omega, z) = HJ(V^{-1}(\sigma, \omega, z))
\end{equation}
where $HJ$ is the right hand side of \eqref{Sformula}. Note that \eqref{S:tilde} gives a (real) analytic extension of $\tilde{S}$ around a neighborhood of $(t,\lambda, 0)$; recall that $\tilde{S}$ is originally defined only for $z>0$ in~\eqref{eq:originStilde}.

Thus, it remains to show that the Jacobian matrix of $V$ at $(t,\lambda_0^t(\lambda),\varepsilon_0^t(\lambda))$ is invertible. The trick here is to do a change of variable to view the map $V$ as a function of $(t, \theta, \rho, \delta)$, since, by~\eqref{value-delta}, the map $(t,\lambda,\varepsilon)\mapsto (t,\theta,\rho,\delta)$ is smoothly invertible. Because the formula of $t_*(\lambda_0, \varepsilon_0)$ is independent of $\rho_0 = \log r_0$ when $\delta$ and $\theta$ are fixed, when $t_\ast(\lambda_0,\varepsilon_0)=t$, $z(t;\lambda_0,\varepsilon_0)$ remains $0$ if $r_0$ is varied with $\delta$ and $\theta$ fixed; this shows $\partial z/\partial r_0=0$.
Furthermore, by \eqref{lambda-t-0} and \eqref{CtStar}, 
$$\lim_{t\to t_*(\lambda_0, \varepsilon_0)}\rho(t)=\frac{\delta-2/r_0}{2\sqrt{\delta^2-4}}\log \left( \frac{\delta+\sqrt{\delta^2-4}}{\delta-\sqrt{\delta^2-4}} \right)$$
whose partial derivative with respect to $r_0$ is positive. Since $\frac{\partial \rho}{\partial r_0}>0$ as shown in the proof of Proposition~\ref{extension-2}, it remains to check $\frac{\partial z}{\partial \delta}>0$ to prove that the Jacobian matrix of the form
$$\begin{pmatrix}
	I_{2\times 2}& 0&0\\
	\ast&  \frac{\partial\rho}{\partial r_0}&\frac{\partial \rho}{\partial \delta}\\
	\ast&  \frac{\partial z}{\partial r_0}&\frac{\partial z}{\partial \delta}\\
	\end{pmatrix}=\begin{pmatrix}
	I_{2\times 2}& 0&0\\
	\ast&  \frac{\partial\rho}{\partial r_0}&\frac{\partial \rho}{\partial \delta}\\
	\ast& 0&\frac{\partial z}{\partial \delta}\\
	\end{pmatrix}$$
	is invertible. To this end, we write 
	$$\frac{\partial z}{\partial \delta} = \frac{\partial z}{\partial \varepsilon_0}\frac{\partial \varepsilon_0}{\partial \delta} = -r_0\frac{\partial z}{\partial t}\frac{\partial t_*(\lambda_0, \varepsilon_0)}{\partial \varepsilon_0}$$
	where the last equality comes from differentiating $z(t_*(\lambda_0, \varepsilon_0), \lambda_0, \varepsilon_0)) = 0$ with respect to $\varepsilon_0$ using chain rule and $\frac{\partial \varepsilon_0}{\partial \delta} = r_0$. Now, by (\ref{pxFormula}) and (\ref{eqn:Zt}), we have 
	$$z(t) = \frac{\sqrt{\varepsilon_0}e^{Ct/2}}{\cosh(kt)+\frac{2r_0-\delta}{\sqrt{\delta^2-4}}\sinh(kt)}\left(\cosh(kt)-\frac{\delta}{\sqrt{\delta^2-4}}\sinh(kt)\right).$$
Recall that $t_*$ given in (\ref{tStar1}) is chosen such that the denominator in (\ref{pxFormula}) is zero. When $t=t_*$, we have 
	$$\left.\frac{\partial z(t, \lambda_0, \varepsilon_0)}{\partial t}\right|_{t = t_*(\lambda_0, \varepsilon_0)} = \frac{\sqrt{\varepsilon_0}e^{Ct/2}}{\cosh(kt)+\frac{2r_0-\delta}{\sqrt{\delta^2-4}}\sinh(kt)}k\left(\sinh(kt)-\frac{\delta}{\sqrt{\delta^2-4}}\cosh(kt)\right)<0$$
because the denominator is positive (see discussions after (\ref{pxFormula}) for the definition of $p_{\varepsilon}(t)$) and 
$\delta/\sqrt{\delta^2-4}>0$.
Finally, using $t_*(\lambda_0, \varepsilon_0) = 1/\int_{-\pi}^\pi\frac{1}{g_{\theta_0-\alpha}(\delta)}d\mu(e^{i\alpha})$, Proposition \ref{monotoneTstar.prop} and the definition (\ref{value-delta}) for $\delta$, we obtain
	$$\frac{\partial t_*(\lambda_0, \varepsilon_0)}{\partial \varepsilon_0} = \frac{\partial t_*(\lambda_0, \varepsilon_0)}{\partial \delta}\frac{\partial \delta}{\partial \varepsilon_0} > 0.$$
We conclude that $\frac{\partial z}{\partial \delta}>0$ and our proof is established.
\end{proof}
Recall that we want to compute the distributional Laplacian of the function
\[s_t(\lambda) = \lim_{\varepsilon\downarrow
	0}S(t, \lambda, \varepsilon).\]
Theorem~\ref{thm:regularity} shows that $s_t$ is indeed analytic on $\Delta_{t,\mu}$ and hence the distributional Laplacian of $s_t$ is indeed the ordinary Laplacian.

The following is our main theorem in the multiplicative case, which generalizes \cite[Theorem 2.2]{DHKBrown}.
\begin{theorem}
	\label{sLaplacian.main}
	Given any $t>0$, for any $\lambda\in\Dt$ with $\lambda=e^\rho e^{i\theta} = re^{i\theta}$,
	we have%
	\begin{equation}
	\frac{\partial s_{t}}{\partial\rho}(t,\lambda)=\frac{2\rho}{t}+1.
	\label{dsdRho}%
	\end{equation}
	Furthermore, $\partial s_{t}/\partial\theta$ is independent of $\rho$ and
	can be expressed as
	\[
	\frac{\partial s_{t}}{\partial\theta}=m_{t}(\theta)
	\]
	where $m_t$ is defined in~\eqref{function_mt}.

	 The Brown measure $\mu_{g_t}$ of $ub_t$ is supported on $\overline{\Dt}$ and can be expressed as 
	\begin{equation}
W_t(r,\theta)=\frac{1}{4\pi}\Delta s_t(\lambda)
		=\frac{1}{r^{2}%
	}w_t(\theta).
	\label{sLaplacian}%
	\end{equation}
Moreover, in $\Dt,$ the density
	$W_{t}$ of $\mu_{g_{t}}$ with respect to the Lebesgue measure is strictly
	positive and real analytic. We also have
\begin{equation}\label{w_t:bound}
	w_t(\theta)< \frac{1}{\pi t}.
\end{equation}	
\end{theorem}

\begin{proof}
For any $\lambda\in \Dt$, choose $\lambda_0=\lambda_0^t(\lambda)$ and $\varepsilon_0=\varepsilon_0^t(\lambda)$ as in Theorem \ref{surjectivity}. Hence $t=t_*(\lambda_0, \varepsilon_0)$. Recall that $S$ is a function in $t, \rho, $ and $\theta$. Also recall that while the Hamiltonian $H$ does not depend on $\theta$ and $p_{\theta}$, we can still regard $\theta$ and $p_\theta$ as constants of motion. Over the trajectory of $S$ which solves the system  \eqref{theODEs} over the interval $[0,t)$, the momentum $p_\theta$ remains unchanged by Proposition \ref{conservation}. 
Write $\lambda_0=r_0e^{i\theta}$. 
Hence, we have
\begin{align*}
   \frac{\partial S}{\partial\theta} (u, \lambda(u), \varepsilon(u))
     =p_\theta
	&=\int_{-\pi}^\pi \frac{2r_0\sin(\theta-\alpha)}{1+r_0^2-2r_0\cos(\theta-\alpha)+\varepsilon_{0}}d\mu(e^{i\alpha})\\
     &=2\int_{-\pi}^\pi \frac{\sin(\theta_0-\alpha)}{\delta-2\cos(\theta_0-\alpha)} d\mu(e^{i\alpha}).
\end{align*}
Using the fact that $\varepsilon_0$ is chosen so that $\delta=(r_t(\theta)^2+1)/r_t(\theta)$ as in (\ref{delta-choice}), the above expression is independent of $|\lambda|$ and only depends on $\theta$. Hence, the function $m_t$ defined in~\eqref{function_mt} can be written as
\begin{align*}
  m_t(\theta) = &2\int_{-\pi}^\pi \frac{r_t(\theta)\sin (\theta-\alpha)}
     {r_t(\theta)^2+1-2\cos(\theta-\alpha)}d\mu(e^{i\alpha})\\
     =&2\int_{-\pi}^\pi \frac{\sin(\theta_0-\alpha)}{\delta-2\cos(\theta_0-\alpha)} d\mu(e^{i\alpha}),
\end{align*}
which shows
\begin{equation}\nonumber
   \frac{\partial^2 S}{\partial \theta^2}(u, \lambda(u), \varepsilon(u))=\frac{d}{d\theta} m_t(\theta).
\end{equation}
Because $s_t(\lambda) = \lim_{z\to 0}\tilde{S}(t,\lambda,z)$ (see \eqref{S:tilde} for definition of $\tilde{S}$), Theorem~\ref{thm:regularity} implies that taking the limit $u\rightarrow t$ from the above display equation gives us
\begin{equation}
  \frac{\partial^2 s_t}{\partial \theta^2}=\frac{d}{d\theta} m_t(\theta)
\end{equation}

Similarly, we use (\ref{apa}) and Theorem~\ref{thm:regularity} to get
	\begin{align*}
\frac{\partial s_t}{\partial \rho}
  &=\lim_{u\rightarrow t}
\frac{2\log\left\vert \lambda(t)\right\vert }{t}+1\\
&=\frac{2\rho}{t}+1.
	\end{align*}
It follows that
\begin{equation}\nonumber
  \frac{\partial^2 s_t}{\partial \rho^2}=\frac{2}{t}.
\end{equation}

So, the restriction of the Brown measure $\mu_{g_t}$ to $\Dt$ is given by
	\begin{equation}
d(\left.\mu_{g_t}\right|_{\Dt})(\lambda)=\frac{1}{4\pi}\Delta s_t(\lambda)
=\frac{1}{4\pi}\frac{1}{r^2}\left( \frac{\partial^2}{\partial \rho^2} +\frac{\partial^2}{\partial \theta^2}\right) s_t(\lambda)
		=\frac{1}{r^2}w_t(\theta).
	\end{equation}
	As the Brown measure $\mu_{g_t}$ is a probability measure, it then follows from Lemma \ref{ProbBrownMul} that $\mu_{g_t}$ is supported on $\overline{\Dt}$. 
Recall that $r_t(\theta)<1$ and is analytic for all $\theta\in U_{t,\bar{\mu}}$. We conclude that
$W_t$ is strictly positive (by Proposition~\ref{dPhiDtheta.prop} and $\frac{d\phi}{d\theta}>0$) and analytic for all $\lambda\in \Dt$. Finally, the upper bound (\ref{w_t:bound}) follows from Lemma~\ref{lemma:positive-derivative-phi-theta} and Proposition~\ref{dPhiDtheta.prop}. 
\end{proof}

We point out that it is possible to express the formula (\ref{equation:wt-mt-derivative}) for $w_t$ in an alternative formula so that there is no derivative involved. Indeed, when $\theta\in U_{t,\bar{\mu}}$, recall that $r_t(\theta)<1$ and $r_t(\theta)$ is analytic by Proposition~\ref{characterization-0}. We first calculate 
\begin{align}
\frac{1}{2}\frac{d}{d\theta}m_t(\theta)&=\frac{dr_t(\theta)}{d\theta}\cdot \int_{-\pi}^{\pi} \frac{\sin(\theta-x)(1-r_t(\theta)^2)}{(r_{t}(\theta)^{2}+1-2r_{t}(\theta
	)\cos(\theta-x))^2}d\mu(e^{i x})\nonumber \\
&+\int_{-\pi}^{\pi}\frac{\cos(\theta-x)(r_t(\theta)^3+r_t(\theta))-2r_t(\theta)^2}{(r_{t}(\theta)^{2}+1-2r_{t}(\theta
	)\cos(\theta-x))^2}d\mu(e^{i x}).\label{identity-derivative-mt-1}
\end{align}
Next, we recall, from Proposition \ref{characterization-0} and the definition of $f(r,\theta)$ given in (\ref{valueOFf}), that when $r_t(\theta)<1$, $r_t(\theta)$ satisfies the identity
\begin{equation}\label{identity-vt-mu-1}
f(r_t(\theta),\theta)=\frac{1-r_t(\theta)^2}{-2\log r_t(\theta)}\int_{-\pi}^{\pi}\frac{1}{|1-r_t(\theta)e^{i(\theta+x)}|^2}d\bar{\mu}(e^{ix}) 
= \frac{1}{t}.
\end{equation} 
Then $\frac{dr_t(\theta)}{d\theta}$ can be calculated by the differential of the above implicit function
\[
 \left. \frac{dr_t(\theta)}{d\theta}=-\frac{\partial f(r,\theta)/\partial\theta}{\partial f(r,\theta)/\partial r}\,\right\vert_{r=r_t(\theta)},
\]
where the denominator is strictly positive whenever $r_t(\theta)<1$, by Corollary~\ref{cor:limit}. The expression of the above display equation is rather complicated. The numerator can be computed as
\begin{align*}
  \frac{\partial f(r,\theta)}{\partial \theta}
    =\frac{1-r^2}{-2\log r} \int_{-\pi}^{\pi}\frac{2r\sin(\theta+x)}{|1-r_t(\theta)e^{i(\theta+x)}|^4}d\bar{\mu}(e^{ix}) ,
\end{align*}
while the denominator can be computed as
\begin{align*}
  \frac{\partial f(r,\theta)}{\partial r}
    =\left(\frac{1-r^2}{2r(\log r)^2}+\frac{r}{\log r}\right)&
   \int_{-\pi}^{\pi}\frac{1}{|1-r_t(\theta)e^{i(\theta+x)}|^2}d\bar{\mu}(e^{ix}) \\
  - &\frac{1-r^2}{-2\log r} \int_{-\pi}^{\pi}\frac{2r+2\cos(\theta+x)}{|1-r_t(\theta)e^{i(\theta+x)}|^4}d\bar{\mu}(e^{ix}).
\end{align*}
Plugging the above formulas for $\frac{dr_t(\theta)}{d\theta}$ to (\ref{identity-derivative-mt-1}) and (\ref{equation:wt-mt-derivative}), we can obtain an alternative expression for $w_t$. We remark that, in the special case when $u=I$, this expression is a very elegant formula (see \cite[Theorem 8.2]{DHKBrown}).

\begin{corollary}
	The support of the Brown measure $\mu_{g_t}$ of $g_t = ub_t$ is the closure of the open set $\Delta_{t,\mu}$. The number of connected components of interior $\Dt$ of the support of $\mu_{g_t}$ is a non-increasing function $t$.
\end{corollary}
\begin{proof}
The first assertion follows from $\mu_{g_t}$ has mass $1$ on $\Dt$. The second assertion follows from the fact that $\Delta_{t,\mu}$ has the same number of connected components of $U_{t,\bar\mu}$. (See~\eqref{Ut-00} for the definition of $U_{t,\bar\mu}$.) Now, the number of connected components of $U_{t,\bar\mu}$ is non-increasing by Proposition~\ref{Mcomponent}.
	\end{proof}

We now describe the connection between the Brown measure of $g_t = ub_t$ with the density function of the spectral distribution of $uu_t$ obtained in \cite{Zhong2015} by the second author. The following two results generalize Propositions 2.5 and 2.6 in \cite{DHKBrown}.

\begin{corollary}
The distribution of the
argument of $\lambda$ with respect to $\mu_{g_{t}}$ has a density given by
\begin{equation}
a_{t}(\theta)=-2\log[r_{t}(\theta)]w_{t}(\theta).
\end{equation}
Furthermore, the push-forward of $\mu_{g_t}$ under the map $\lambda\mapsto \Phi_{t,\bar{\mu}}(r_t(\theta)e^{i\theta})$ is the distribution of $uu_t$.
\end{corollary}
\begin{proof}
The Brown measure in
the domain is computed in polar coordinates as $(1/r^{2})w_{t}(\theta
)r\,dr\,d\theta$. Integrating with respect to $r$ from $r_{t}(\theta)$ to
$1/r_{t}(\theta)$ then gives the claimed density for $\theta$.
The last assertion follows from a computation similar to Lemma \ref{ProbBrownMul}.
\end{proof}

\begin{proposition}
	\label{multunique}
	The Brown measure of $ub_t$ is the unique measure $\sigma$ on $\overline{\Delta_{t,\mu}}$ with the following two properties: $(1)$ the push-forward of $\sigma$ by the map $\lambda\mapsto \Phi_{t,\bar{\mu}}(r_t(\theta)e^{i\theta})$ is the distribution of $uu_t$ where $\theta=\arg (\lambda)$, and $(2)$ $\sigma$ is absolutely continuous with respect to the Lebesgue measure and its density is given by
	\[
	     W(r,\theta)=\frac{1}{r^2}g(\theta)
	\]
in polar coordinates, where $g$ is a continuous function. 
\end{proposition}
\begin{proof}
Suppose $\sigma$ is a measure on $\overline{\Delta_{t,\mu}}$ satisfying the above two properties. We have 
\[
  \int_{r_t(\theta)}^{1/r_t(\theta)}\frac{1}{r^2}g(\theta)dr=-2\log [r_t(\theta)] g(\theta). 
\]
Hence, using (\ref{wtFromPhi}), we deduce the density of the push-forward measure of $\sigma$ as
\[
  -2\log[r_t(\theta)]g(\theta)\frac{d\theta}{d\phi}d\phi=\frac{-\log[r_t(\theta)]}{\pi t}
  \frac{g(\theta)}{w_t(\theta)}d\phi. 
\]
By comparing the density formula for $uu_t$ in Theorem \ref{thm:mut}, and noticing that $r_t(\theta)<1$ for $\theta\in U_{t,\bar{\mu}}$, we must have $g(\theta)=w_t(\theta)$. 
\end{proof}


\begin{figure}[htp]
	\begin{center}
		\begin{subfigure}[h]{0.45\linewidth}
			\includegraphics[width=\linewidth]{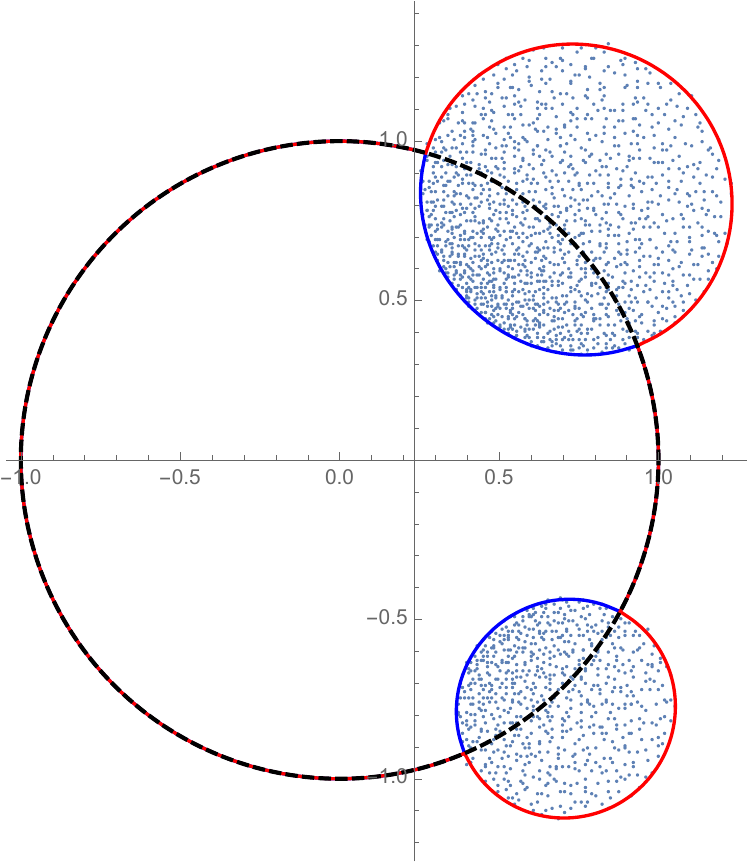}
			\caption{Eigenvalue plot for $ub_t$ and the curves $\omega_t(e^{i\phi})$ (blue), $1/\overline{\omega_t(e^{i\phi})}$ (red) and the unit circle (dashed black)}
		\end{subfigure}\;
		\begin{subfigure}[h]{0.45\linewidth}
			\includegraphics[width=\linewidth]{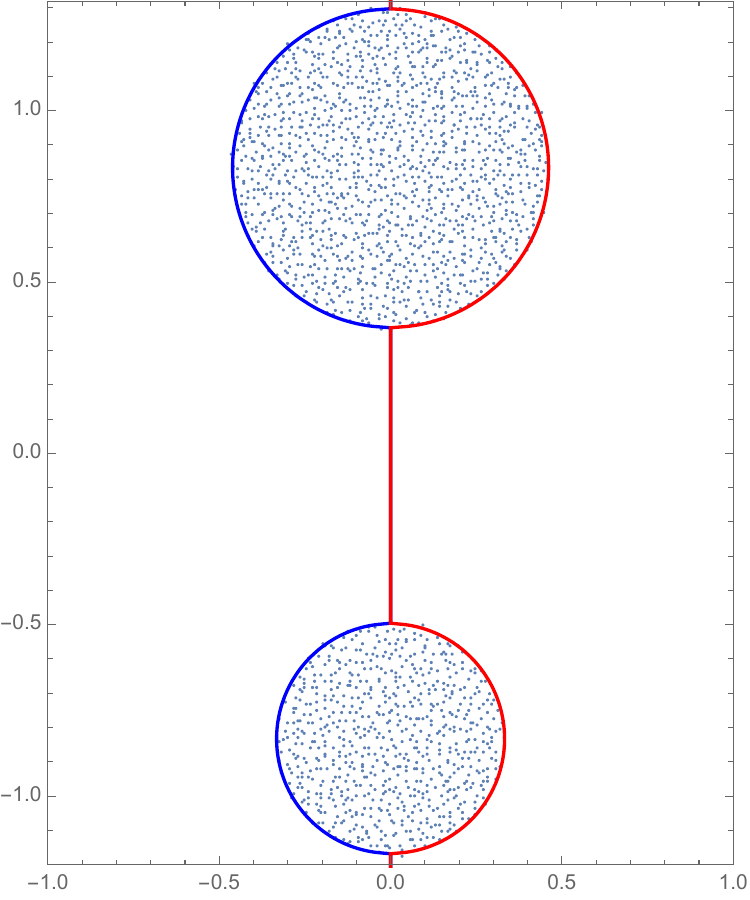}
			\caption{Push-forward of eigenvalues $ub_t$ by principal logarithm and the curves $\log(\omega_t(e^{i\phi}))$, $-\log(\overline{\omega_t(e^{i\phi})})$ and the unit circle}
		\end{subfigure}
		\begin{subfigure}[h]{0.45\linewidth}
			\includegraphics[width=\linewidth]{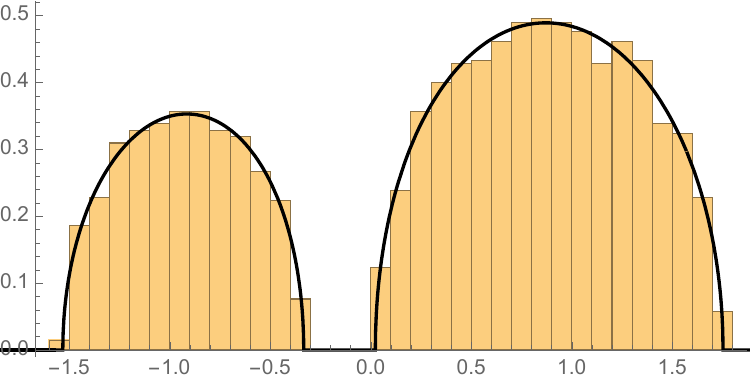}
			\caption{Histogram of the argument of the eigenvalues $ub_t$ push-forward by$\lambda\mapsto \Phi_{t,\bar{\mu}}(r_t(\theta)e^{i\theta})$. The distribution of $uu_t$ is superimposed.}
		\end{subfigure}
	\end{center}
	\caption{$2100\times 2100$ Matrix simulations for $uu_t$ at $t=0.3$, where $u$ is distributed as $\frac{1}{3}\delta_{e^{\frac{-4\pi i}{15}}}+\frac{2}{3}\delta_{e^{\frac{4\pi i}{15}}}$}
\end{figure}

\begin{figure}[htp]
	\begin{center}
		\begin{subfigure}[h]{0.45\linewidth}
			\includegraphics[width=\linewidth]{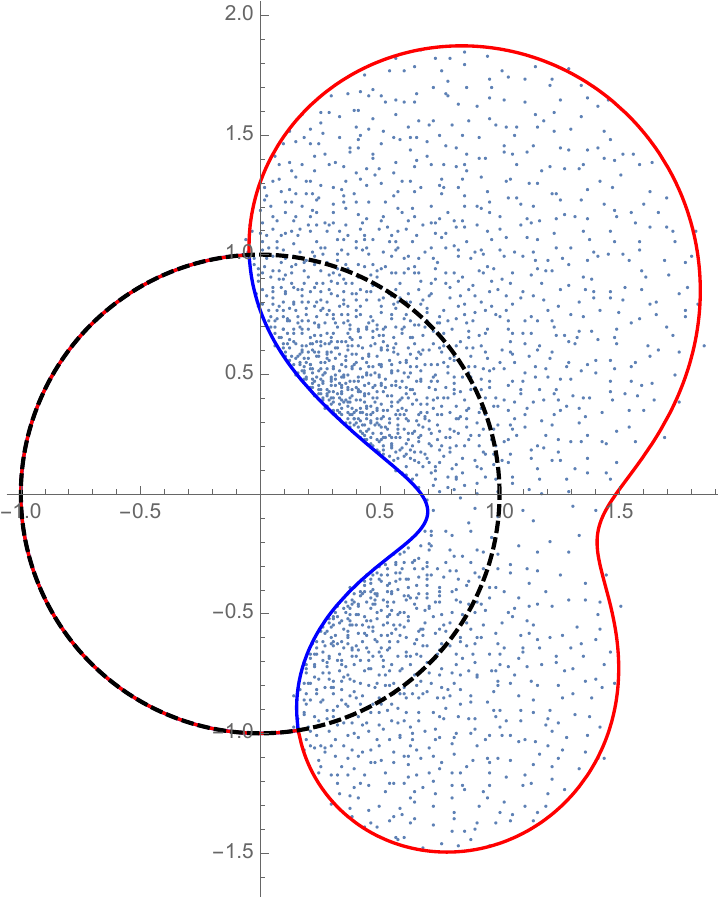}
			\caption{Eigenvalue plot for $ub_t$ and the curves $\omega_t(e^{i\phi})$ (blue), $1/\overline{\omega_t(e^{i\phi})}$ (red) and the unit circle (dashed black)}
		\end{subfigure}\;
		\begin{subfigure}[h]{0.45\linewidth}
			\includegraphics[width=\linewidth]{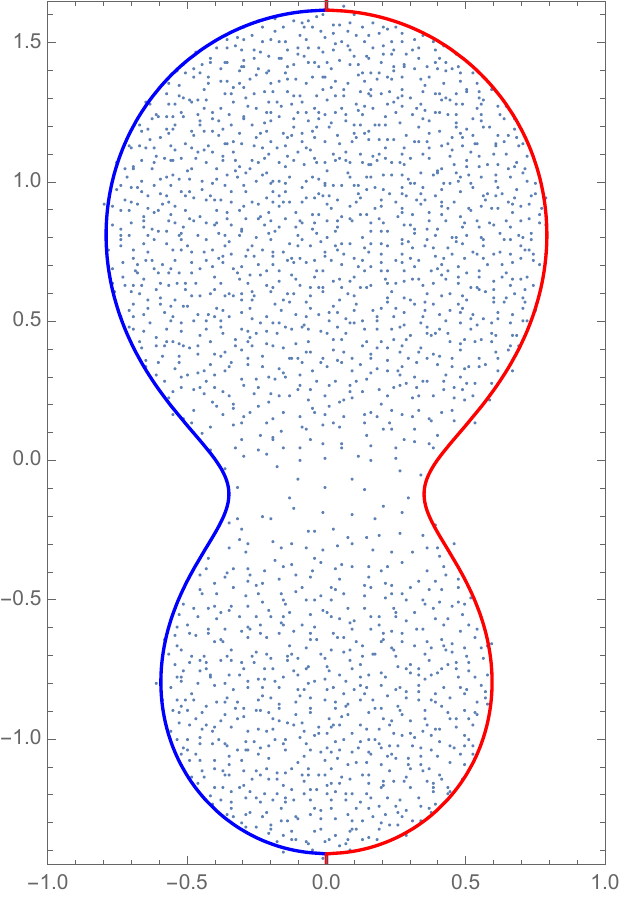}
			\caption{Push-forward of eigenvalues $ub_t$ by logarithm (with branch cut $(0,\infty)$
				and the curves $\log(\omega_t(e^{i\phi}))$, $-\log(\overline{\omega_t(e^{i\phi})})$ and the unit circle}
		\end{subfigure}
		\begin{subfigure}[h]{0.45\linewidth}
			\includegraphics[width=\linewidth]{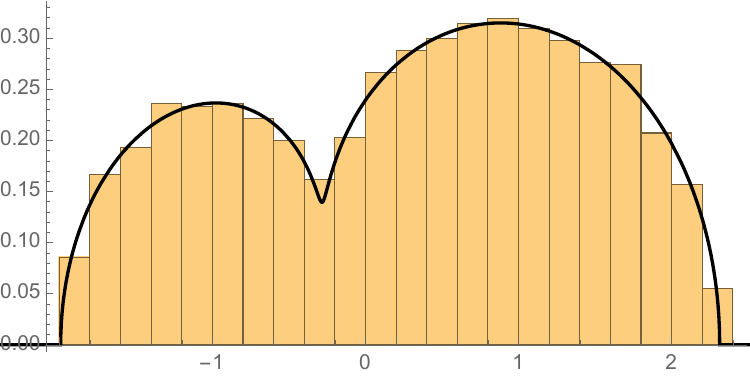}
			\caption{Histogram of the argument of the Eigenvalues $ub_t$ push-forward by$\lambda\mapsto \Phi_{t,\bar{\mu}}(r_t(\theta)e^{i\theta})$. The distribution of $uu_t$ is superimposed.}
		\end{subfigure}
	\end{center}
	\caption{$2100\times 2100$ Matrix simulations for $uu_t$ at $t=0.8$, where $u$ is distributed as $\frac{1}{3}\delta_{e^{\frac{-4\pi i}{15}}}+\frac{2}{3}\delta_{e^{\frac{4\pi i}{15}}}$}
\end{figure}

We include the matrix simulations of $ub_t$ at time $t=0.2$ and $t=0.8$ where $u$ has spectral distribution $$\mu = \frac{1}{3}\delta_{e^{\frac{-4\pi i}{15}}}+\frac{2}{3}\delta_{e^{\frac{4\pi i}{15}}}.$$
We again emphasize that it is still an open problem to prove mathematically that the Brown measure of $ub_t$ is the weak limit of the empirical eigenvalue distribution of $U_NG_N(t)$ where $G_N(t)$ is the Brownian motion on $GL(N)$ and $U_N$ is a deterministic (or random but independent from $G_N(t)$) unitary matrix whose empirical eigenvalue distribution has weak limit $\mu$.

In each of the two figures, part (a) plots the eigenvalues of $ub_t$ against the unit circle, and the curves $\omega_t(e^{i\phi})$, $1/\overline{\omega_t(e^{i\phi})}$. Part (b) of the figures shows the eigenvalues of $ub_t$ and the curves $\omega_t(e^{i\phi})$, $1/\overline{\omega_t(e^{i\phi})}$ after pushing-forward by the principal logarithm. We can see that the density of the points are constant along the horizontal direction. Part (c) shows the histogram of the argument of eigenvalues after pushing-forward by the map $\lambda\mapsto \Phi_{t,\bar{\mu}}(r_t(\theta)e^{i\theta})$ against the theoretical distribution --- the spectral distribution of $uu_t$ --- in argument. We remark that the spectral distribution of $uu_t$ is established in Theorem \ref{thm:mut}.


\subsection{The Brown measure of $hb_t$.}
In this section, we calculate the Brown measure of $hb_t$ where $h$ is Haar unitary as an example. Recall from Section \ref{section:multiplicative_convolution} that when $u$ is a Haar unitary $h$, we have
\begin{equation}
	\Sigma_{t,\bar{h}}(z)=e^{t/2}, \quad\text{and}\quad \Phi_{t, \bar{h}}(z)=ze^{t/2},
\end{equation}
and $r_t(\theta)=e^{-t/2}$ for all $\theta$. Hence $\Phi_{t, \bar{h}}(r_t(\theta)e^{i\theta})=e^{i\theta}$ and  
\[
  \phi=\phi(\theta)=\arg\left(\Phi_{t, \bar{h}}(r_t(\theta)e^{i\theta}) \right)=\theta. 
\]
In addition, 
\[
  \Omega_{t, \bar{h}}=\{z: |z|<e^{-t/2} \},
\]
and the set $\Delta_{t,\bar{h}}$ is the annulus
\begin{equation}
\Delta_{t,\bar{h}}=\{z: e^{-t/2}<|z|<e^{t/2} \}.
\end{equation}
Therefore, by (\ref{wtFromPhi}), we have
\[
 w_t(\theta)=\frac{1}{2\pi t}
\]
Finally, by Theorem \ref{sLaplacian.main}, we have the following formula for the Brown measure of $hb_t$. 
\begin{theorem}
	\label{HaarFormula}
	The Brown measure  $\mu_{hb_t}$ of $hb_t$ is the annulus law. It is supported in the annulus $A_t = \{e^{-t/2}\leq |z|\leq e^{t/2} \}$, rotationally invariant, and the density $W_t$ of $\mu_{hb_t}$ in polar coordinates is given by
	\begin{equation}
	W_t(r,\theta) = \frac{1}{2\pi t}\frac{1}{r^2}
	\end{equation}
	which is independent of $\theta$.
\end{theorem}

\begin{remark}
The name \emph{annulus law} was named by Driver, Hall and Kemp. 
	It is expected that the solution of 
	\[ G_N(t) = G_N(t) \,dZ_N(t)\]
	under the initial condition that $G_N(0)$ is the Haar measure on the unitary group has the limiting eigenvalue distribution equal to the Brown measure of $hb_t$. One can run a computer program to see that the eigenvalues of $G_N(t)$ are distributed on an annulus $A_t$ with inner radius $e^{-t/2}$ and outer radius $e^{t/2}$. The support $A_t$ also occurs in the discussion of the free Hall transform as $s\to\infty$ \cite[Corollary 4.26]{Ho2016}. 
\end{remark}

{\bf Acknowledgment.} 
We would like to thank Hari Bercovici, Brian Hall and Todd Kemp for helpful discussions. We are grateful to Brian Hall for reading the first draft of the manuscript and giving valuable suggestions. Brian also helped us prepare computer simulations. 
The project started when the two authors met at CRM in the University of Montreal when we participated in the program \enquote{New Developments in Free Probability and Applications} in 2019. Part of the work was done when the second author was visiting Prof. Alexandru Nica in the Department of Pure Mathematics at the University of Waterloo. The support of CRM and UWaterloo and their very inspiring environments are gratefully acknowledged. The second author was supported in part by a start-up grant from University of Wyoming and the Collaboration Grants for Mathematicians from Simons Foundation. Finally, we want to thank the anonymous referees for their insightful comments and constructive suggestions.

\bibliographystyle{acm}
\bibliography{HZ_Brown1}

\begin{thebibliography}{10}

\bibitem{AkhiezerBook}
{\sc Akhiezer, N.~I.}
\newblock {\em The classical moment problem and some related questions in
  analysis}.
\newblock Translated by N. Kemmer. Hafner Publishing Co., New York, 1965.

\bibitem{Anderson2013}
{\sc Anderson, G.~W.}
\newblock Convergence of the largest singular value of a polynomial in
  independent {W}igner matrices.
\newblock {\em Ann. Probab. 41}, 3B (2013), 2103--2181.

\bibitem{Belinschi2008}
{\sc Belinschi, S.~T.}
\newblock The {L}ebesgue decomposition of the free additive convolution of two
  probability distributions.
\newblock {\em Probab. Theory Related Fields 142}, 1-2 (2008), 125--150.

\bibitem{BB2005}
{\sc Belinschi, S.~T., and Bercovici, H.}
\newblock Partially defined semigroups relative to multiplicative free
  convolution.
\newblock {\em Int. Math. Res. Not.}, 2 (2005), 65--101.

\bibitem{BB2007new}
{\sc Belinschi, S.~T., and Bercovici, H.}
\newblock A new approach to subordination results in free probability.
\newblock {\em J. Anal. Math. 101\/} (2007), 357--365.

\bibitem{BSS2018}
{\sc Belinschi, S.~T., Śniady, P., and Speicher, R.}
\newblock Eigenvalues of non-{H}ermitian random matrices and {B}rown measure of
  non-normal operators: Hermitian reduction and linearization method.
\newblock {\em Linear Algebra and its Applications 537\/} (2018), 48 -- 83.

\bibitem{survey-2013}
{\sc Bercovici, H.}
\newblock {\em Free convolution. In \emph{Free probability and operator
  algebras},Voiculescu, Dan and Stammeier, Nicolai and Weber, Moritz editors,
  M\"{u}nster Lectures in Mathematics}.
\newblock European Mathematical Society (EMS), Z\"{u}rich, 2016.
\newblock Lecture notes from the masterclass held in M\"{u}nster.

\bibitem{BV1992}
{\sc Bercovici, H., and Voiculescu, D.}
\newblock L\'evy-{H}in\v cin type theorems for multiplicative and additive free
  convolution.
\newblock {\em Pacific J. Math. 153}, 2 (1992), 217--248.

\bibitem{BercoviciVoiculescu1998}
{\sc Bercovici, H., and Voiculescu, D.}
\newblock Regularity questions for free convolution.
\newblock In {\em Nonselfadjoint operator algebras, operator theory, and
  related topics}, vol.~104 of {\em Oper. Theory Adv. Appl.} Birkh\"{a}user,
  Basel, 1998, pp.~37--47.

\bibitem{Biane1997b}
{\sc Biane, P.}
\newblock Free {B}rownian motion, free stochastic calculus and random matrices.
\newblock In {\em Free probability theory ({W}aterloo, {ON}, 1995)}, vol.~12 of
  {\em Fields Inst. Commun.} Amer. Math. Soc., Providence, RI, 1997, pp.~1--19.

\bibitem{Biane1997}
{\sc Biane, P.}
\newblock On the free convolution with a semi-circular distribution.
\newblock {\em Indiana Univ. Math. J. 46}, 3 (1997), 705--718.

\bibitem{Biane1997JFA}
{\sc Biane, P.}
\newblock Segal-{B}argmann transform, functional calculus on matrix spaces and
  the theory of semi-circular and circular systems.
\newblock {\em J. Funct. Anal. 144}, 1 (1997), 232--286.

\bibitem{Biane1998}
{\sc Biane, P.}
\newblock Processes with free increments.
\newblock {\em Math. Z. 227}, 1 (1998), 143--174.

\bibitem{BianeLehner2001}
{\sc Biane, P., and Lehner, F.}
\newblock Computation of some examples of {B}rown's spectral measure in free
  probability.
\newblock {\em Colloq. Math. 90}, 2 (2001), 181--211.

\bibitem{BianeSpeicher1998}
{\sc Biane, P., and Speicher, R.}
\newblock Stochastic calculus with respect to free {B}rownian motion and
  analysis on {W}igner space.
\newblock {\em Probab. Theory Related Fields 112}, 3 (1998), 373--409.

\bibitem{Bordenave-Chafai-circular}
{\sc Bordenave, C., and Chafa\"{\i}, D.}
\newblock Around the circular law.
\newblock {\em Probab. Surv. 9\/} (2012), 1--89.

\bibitem{Brown1986}
{\sc Brown, L.~G.}
\newblock Lidskii's theorem in the type {II} case.
\newblock In {\em Geometric methods in operator algebras ({K}yoto, 1983)},
  vol.~123 of {\em Pitman Res. Notes Math. Ser.} Longman Sci. Tech., Harlow,
  1986, pp.~1--35.

\bibitem{BGNTW2014}
{\sc Burda, Z., Grela, J., Nowak, M.~A., Tarnowski, W., and Warcho\l{}, P.}
\newblock Dysonian dynamics of the {G}inibre ensemble.
\newblock {\em Phys. Rev. Lett. 113\/} (Sep 2014), 104102.

\bibitem{BGNTW2015}
{\sc Burda, Z., Grela, J., Nowak, M.~A., Tarnowski, W., and Warchoł, P.}
\newblock Unveiling the significance of eigenvectors in diffusing
  non-{H}ermitian matrices by identifying the underlying {B}urgers dynamics.
\newblock {\em Nuclear Physics B 113\/} (03 2015).

\bibitem{DemniHamdi2020}
{\sc Demni, N., and Hamdi, T.}
\newblock Support of the {B}rown measure of the product of a free unitary
  brownian motion by a free self-adjoint projection.
\newblock {\em arXiv:2002.04585\/} (2020).

\bibitem{DHKBrown}
{\sc Driver, B.~K., Hall, B.~C., and Kemp, T.}
\newblock The {B}rown measure of the free multiplicative {B}rownian motion.
\newblock {\em arXiv:1903.11015v2\/} (2019).

\bibitem{EvansBook}
{\sc Evans, L.~C.}
\newblock {\em Partial {D}ifferential {E}quations}.
\newblock {V}ol 19, Graduate Studies in Mathematics. American Mathematical
  Society, Providence, RI, 1998.

\bibitem{FugledeKadison1952}
{\sc Fuglede, B., and Kadison, R.~V.}
\newblock Determinant theory in finite factors.
\newblock {\em Ann. of Math. (2) 55\/} (1952), 520--530.

\bibitem{Ginibre1965}
{\sc Ginibre, J.}
\newblock {Statistical Ensembles of Complex, Quaternion and Real Matrices}.
\newblock {\em J. Math. Phys. 6\/} (1965), 440--449.

\bibitem{HaagerupLarsen2000}
{\sc Haagerup, U., and Larsen, F.}
\newblock Brown's spectral distribution measure for {R}-diagonal elements in
  finite von {N}eumann algebras.
\newblock {\em J. Funct. Anal. 176}, 2 (2000), 331 -- 367.

\bibitem{HaagerupSchultz2007}
{\sc Haagerup, U., and Schultz, H.}
\newblock Brown measures of unbounded operators affiliated with a finite von
  {N}eumann algebra.
\newblock {\em Math. Scand. 100}, 2 (2007), 209--263.

\bibitem{Hall2001}
{\sc Hall, B.~C.}
\newblock Harmonic analysis with respect to heat kernel measure.
\newblock {\em Bull. Amer. Math. Soc. (N.S.) 38}, 1 (2001), 43--78.

\bibitem{Hall2019pde}
{\sc Hall, B.~C.}
\newblock P{DE} methods in random matrix theory.
\newblock In {\em Harmonic analysis and applications}, vol.~168 of {\em
  Springer Optim. Appl.} Springer, Cham, 2021, pp.~77--124.

\bibitem{HallHo2020}
{\sc Hall, B.~C., and Ho, C.-W.}
\newblock The {B}rown measure of the sum of self-adjoint element and an
  imaginary multiple of a semicircular element.
\newblock {\em arXiv:2006.07168\/} (2020).

\bibitem{HallKemp2019}
{\sc Hall, B.~C., and Kemp, T.}
\newblock Brown measure support and the free multiplicative {B}rownian motion.
\newblock {\em Adv. Math. 355\/} (2019), 106771, 36.

\bibitem{Ho2016}
{\sc Ho, C.-W.}
\newblock The two-parameter free unitary {S}egal-{B}argmann transform and its
  {B}iane-{G}ross-{M}alliavin identification.
\newblock {\em J. Funct. Anal. 271}, 12 (2016), 3765--3817.

\bibitem{Kemp2016}
{\sc Kemp, T.}
\newblock The large-{$N$} limits of {B}rownian motions on {$\Bbb{GL}_N$}.
\newblock {\em Int. Math. Res. Not. IMRN}, 13 (2016), 4012--4057.

\bibitem{KummererSpeicher1992}
{\sc K\"{u}mmerer, B., and Speicher, R.}
\newblock Stochastic integration on the {C}untz algebra {$O_\infty$}.
\newblock {\em J. Funct. Anal. 103}, 2 (1992), 372--408.

\bibitem{Liao}
{\sc Liao, M.}
\newblock {\em Levy Processes in Lie Groups}, vol.~162 of {\em Cambridge tracts
  in mathematics}.
\newblock Cambridge University Press, 2004.

\bibitem{MingoSpeicherBook}
{\sc Mingo, J.~A., and Speicher, R.}
\newblock {\em Free probability and random matrices}, vol.~35 of {\em Fields
  Institute Monographs}.
\newblock Springer, New York; Fields Institute for Research in Mathematical
  Sciences, Toronto, ON, 2017.

\bibitem{SpeicherNicaBook}
{\sc Nica, A., and Speicher, R.}
\newblock {\em Lectures on the Combinatorics of Free Probability}.
\newblock London Mathematical Society Lecture Note Series. Cambridge University
  Press, 2006.

\bibitem{Sniady2002}
{\sc \'{S}niady, P.}
\newblock Random regularization of {B}rown spectral measure.
\newblock {\em J. Funct. Anal. 193}, 2 (2002), 291--313.

\bibitem{DVV-circular}
{\sc Voiculescu, D.}
\newblock Circular and semicircular systems and free product factors.
\newblock In {\em Operator algebras, unitary representations, enveloping
  algebras, and invariant theory ({P}aris, 1989)}, vol.~92 of {\em Progr.
  Math.} Birkh\"{a}user Boston, Boston, MA, 1990, pp.~45--60.

\bibitem{Voiculescu1991}
{\sc Voiculescu, D.}
\newblock Limit laws for random matrices and free products.
\newblock {\em Invent. Math. 104}, 1 (1991), 201--220.

\bibitem{Voiculescu1993}
{\sc Voiculescu, D.}
\newblock The analogues of entropy and of {F}isher's information measure in
  free probability theory. {I}.
\newblock {\em Comm. Math. Phys. 155}, 1 (1993), 71--92.

\bibitem{DVV-general}
{\sc Voiculescu, D.}
\newblock {The coalgebra of the free difference quotient and free probability}.
\newblock {\em Int. Math. Res. Not. 2000}, 2 (01 2000), 79--106.

\bibitem{Wigner1955}
{\sc Wigner, E.~P.}
\newblock Characteristic vectors of bordered matrices with infinite dimensions.
\newblock {\em Ann. of Math. (2) 62\/} (1955), 548--564.

\bibitem{Zhong2014}
{\sc Zhong, P.}
\newblock Free {B}rownian motion and free convolution semigroups:
  multiplicative case.
\newblock {\em Pacific J. Math. 269\/} (2014), 219--256.

\bibitem{Zhong2015}
{\sc Zhong, P.}
\newblock On the free convolution with a free multiplicative analogue of the
  normal distribution.
\newblock {\em J. Theoret. Probab. 28}, 4 (2015), 1354--1379.

\end{thebibliography}

\begin{appendices}
\section{The Brown measure of $hb_t$ and $R$-diagonal operators.}\label{section:haar}
In this section, we provide an alternative way to calculate the Brown measure of $hb_t$, the free multiplicative Brownian motion with the Haar unitary as the initial condition using Haagerup and Larsen's formula. The operator $hb_t$ is an $R$-diagonal operator \cite[Corollary 15.9]{SpeicherNicaBook} and its Brown measure can be characterize by the $S$-transform of $b_tb_t^*$ \cite[Theorem 4.4]{HaagerupLarsen2000}. (See~\eqref{eq:Stransform} for the definition of $S$-transform.) The $S$-transform of $b_tb_t^*$ has been calculated by Biane in \cite{Biane1997JFA} using different (but equivalent) normalization. We put these work together to give an explicit formula for the Brown measure of $hb_t$. 

For every $t\in\mathbb{R}$, it is known by \cite[Lemma 6.3 and Lemma 7.1]{BV1992} that the functions 
\[
S(z)=\exp\left( t\left(z+\frac{1}{2}\right)\right)
\]
are the $S$-transform of probability measures $\gamma_t$ (See~\eqref{eq:Stransform} for the definition of $S$-transform). The corresponding probability measures are supported on the unit circle $\mathbb{T}$ for $t>0$, and the corresponding probability measures are supported on the positive half-line $\mathbb{R}^+$ for $t<0$.
For a probability measure $\mu$ on $\mathbb{T}$ or on $\mathbb{R}^+$, the $S$-transform $S_\mu$ of $\mu$ is equivalent to the $\Sigma$-transform $\Sigma_\mu$ by
\[
\Sigma_\mu(z)=S_\mu\left(\frac{z}{1-z}\right).
\]
Converting the $S$-transform to the $\Sigma$-transform, we have a family of probability measures $\gamma_t$ ($t\in\mathbb{R}$) such that
\begin{equation}\label{eqn:gamma_t}
\Sigma_{\gamma_t}(z)=\exp\left(\frac{t}{2}\frac{1+z}{1-z}\right).
\end{equation}
When $t>0$, the measure $\gamma_t$ is known \cite{Biane1997b} to be the spectral distribution of the free unitary Brownian motion $u_t$ defined in (\ref{eqn:UBM}). When $t=0$, the measure $\gamma_0$ is the Dirac measure at $x=1$. When $t<0$, the measure is known as
the positive Brownian motion on $\mathbb{R}^+$ in free probability literature. 
The formulas for the densities and moments of $\gamma_t$ can be found in \cite{Biane1997b, Biane1997JFA, Zhong2015}. 

Let us recall Haagerup--Larsen's formula \cite{HaagerupLarsen2000} for the Brown measure of $R$-diagonal operators reformulated in \cite[Theorem 2]{BSS2018}. 
\begin{theorem}[Haagerup--Larsen \cite{HaagerupLarsen2000}]\label{thm:formula-R-diagonal}
	Let $a$ be an operator that is $*$-free from the Haar unitary $h$ in a $W^*$-probability space $(\A, \tau)$.
	The Brown measure $\mu_x$ of the $R$-diagonal operator $x=ha$ is the unique rotationally invariant probability measure such that
	\begin{align*}
	\mu_x( \{z\in\mathbb{C}:|z|\leq r \}) =\begin{cases}
	0   & \text{for}\quad r\leq [\tau(xx^*)^{-1}]^{-1/2}; \\
	1+S_{xx^*}^{\langle-1\rangle}(r^{-2}) &	\text{for}\quad [\tau(xx^*)^{-1}]^{-1/2}\leq r\leq [\tau(xx^*)]^{1/2};\\
	1  & \text{for}\quad r\geq [\tau(xx^*)]^{1/2}.
	\end{cases}
	\end{align*} 
\end{theorem}

We then have to calculate the $S$-transform of $b_tb_t^*$. 
It is known in Biane's work \cite{Biane1997b, Biane1997JFA} (see also \cite{Kemp2016}) that probability measures $\gamma_t$ that connects to the spectral distribution of $b_tb_t^*$ are the case $t<0$. As the parametrizations in \cite{Biane1997b, Biane1997JFA} are slightly different from ours, we provide details here for convenience.  

For $t>0$, we denote $Q_t=e^{-t}b_tb_t^*$ and set the generating function
\[
\psi(t,z)=\sum_{n=1}^{\infty}z^n \tau(Q_t^n).
\]
The following result is essentially from \cite[Lemma 13]{Biane1997JFA}.
\begin{lemma}
	The generating function $\psi(t,z)$ satisfies 
	\begin{equation}\label{eqn:diff-psi}
	\frac{\partial \psi(t,z)}{\partial t}=2z\psi(t,z)\frac{\partial\psi(t,z)}{\partial z}. 
	\end{equation}
\end{lemma}
\begin{proof}
	Recall that $b_t$ is the solution of (\ref{eqn:GBM}), by the free It\^{o} formula~\eqref{eq:Ito}, we have
	\[
	dQ_t=e^{-t}b_t(dc_t+dc_t^*)b_t^*.
	\]
	By induction, we obtain by the free It\^o formula
	\[
	dQ_t^n=\sum_{k=0}^{n-1}Q_t^kdQ_tQ_t^{n-k-1}+2\sum_{k=1}^{n-1}kQ_t^k\tau(Q_t^{n-k})dt.
	\]
	Integrating it and take the trace, we deduce
	\begin{equation}\label{Q_t-induction}
	\tau(Q_t^n)=2\int_0^t\sum_{k=1}^{n-1}k\tau(Q_s^k)\tau(Q_s^{n-k})ds+1,
		\end{equation}
	which yields
	\[
	\psi(t,z)=2\int_0^t z\psi(t,z)\frac{\partial \psi(t,z)}{\partial z} ds+\frac{z}{1-z}.
	\]
	The identity then follows by taking the derivative with respect to $t$. 
\end{proof}

The following result was mentioned implicitly in \cite{Biane1997b} and the proof can be adapted from the proof of \cite[Lemma 1]{Biane1997b}.
\begin{proposition}
	The spectral distribution of $b_tb_t^*$ is the measure $\gamma_{-2t}$ defined in (\ref{eqn:gamma_t}). The $S$-transform of $b_tb_t^*$ is given by
	\[
	S_{b_tb_t^*}(z)=e^{-t-2tz}. 
	\]
\end{proposition}
\begin{proof}
	Let $\rho(t,z)=\psi_{\gamma_{-2t}}(e^{-t}z)$, we shall prove that $\rho(t,z)=\psi(t,z)$. Fix $t>0$, by (\ref{eqn:gamma_t}) and the definition of $\Sigma$-transform, a simple calculation yields 
	\[
	\frac{\psi_{\gamma_{-2t}}(z)}{1+\psi_{\gamma_{-2t}}(z)}\exp\left(-t-2t\psi_{\gamma_{-2t}}(z)\right)=z.
	\]
	By replacing $z$ by $e^{-t}z$, we obtain
	\begin{equation}\label{eqn:ind-rho}
	\frac{\rho(t,z)}{1+\rho(t,z)}\exp\left(-2t\rho(t,z)\right)=z.
	\end{equation}
	
	We next show that $\rho(t,z)$ satisfies the same differential equation (\ref{eqn:diff-psi}) for $\psi(t,z)$. To this end, from (\ref{eqn:ind-rho}), we calculate implicit differentials
	\[
	\frac{\partial \rho(t,z)}{\partial t}=\frac{2z\rho(t,z)(1+\rho(t,z))e^{2t\rho(t,z)}}{1-ze^{2t\rho(t,z)}-2tz(1+\rho(t,z))e^{2t\rho(t,z)}},
	\]
	and
	\[
	\frac{\partial \rho(t,z)}{\partial z}=\frac{(1+\rho(t,z))e^{2t\rho(t,z)}}{1-ze^{2t\rho(t,z)}-2tz(1+\rho(t,z))e^{2t\rho(t,z)}},
	\]
	which implies that $\rho(t,z)$ satisfies the same differential equation (\ref{eqn:diff-psi}) for $\psi(t,z)$. Since $\psi$ is a power series in $z$, the differential equation \eqref{eqn:diff-psi} is equivalent to the differential equations \eqref{Q_t-induction} for the coefficients $\{\tau(Q_t^n)\}_{n=1}^\infty$ of $\psi(t,z)$. Because $\rho$ is an analytic function in $z$ satisfying the same differential equation as $\psi$, the coefficients $a_{t,n}$ of
	\[\rho(t,z) = \sum_{n=1}^\infty a_{t, n}z^n\]
	satisfy the same differential equations as $\tau(Q_t^n)$ in~\eqref{Q_t-induction}.  Since $\rho(0,z)=z/(1-z)=\psi(0,z)$, we have the initial conditions $\tau(Q_0^n) = a_{0,n}$. The differential equations for $\tau(Q_t^n)$ (and $a_{t,n}$) given in ~\eqref{Q_t-induction} have a unique solution. Consequently, $\tau(Q_t^n) = a_{t,n}$ and $\rho=\psi$. Therefore,
	\[
	S_{b_tb_t^*}(z)=S_{\gamma_{-2t}}(z)=e^{-t-2tz},
	\]
	which concludes the proof.
\end{proof}
\begin{remark}
	Indeed, the moments $\tau(Q_t^n)$ can be computed expicitly. Interested readers are referred to \cite[Lemma 1 and Proposition 5]{Biane1997b}.
	\end{remark}

We can calculate $\tau(b_tb_t^*)$ directly from the $S$-transform of $b_tb_t^*$ as
\[
\tau(b_tb_t^*)=\psi_{b_tb_t^*}'(0)=\frac{1}{S_{b_tb_t^*}(0)}=e^t.
\]
It is known in \cite{Zhong2015} that the support of $\gamma_{-2t}$ is $\text{supp}(\gamma_{-2t})=[r_-(2t),r_+(2t)]$, where
\[
r_{\pm}(2t)=\left[(1-t)\pm \sqrt{t(t-2)}e^{-\sqrt{t(t-2)}} \right].
\]
In particular, we have $r_-(2t){r_+(2t)}=1$ and so  $b_tb_t^*$ is strictly positive. By the free It\^{o} formula, we can check that
\[
db_t^{-1}=-dc_t b_t^{-1}.
\]
Hence it is routine to check that $(b_tb_t^*)^{-1}$ has the same distribution as $b_tb_t^*$. Therefore, $\tau((b_tb_t^*)^{-1})=\tau(b_tb_t^*)=e^t$. We are now ready to apply Theorem \ref{thm:formula-R-diagonal} to deduce the formula for the Brown measure of $hb_t$. For any $e^{-t/2}\leq r\leq e^{t/2}$, we have 
\begin{equation}
\label{cumulativemu}
\mu_{hb_t}(\{|z|\leq r \}) =1+S_{b_tb_t^*}^{\langle -1\rangle}\left(\frac{1}{r^2}\right)=\frac{1}{2}+\frac{\log r}{t}.
\end{equation}
This leads to an alternation derivation of the density for the Brown measure of $hb_t$ computed in Theorem~\ref{HaarFormula}. 
\begin{theorem}
	The Brown measure  $\mu_{hb_t}$ of $hb_t$ is the annulus law. It is supported in the annulus $A_t = \{e^{-t/2}\leq |z|\leq e^{t/2} \}$, rotationally invariant, and the density $W_t$ of $\mu_{hb_t}$ in polar coordinates is given by
	\begin{equation}
	\label{densityhaar}
	W_t(r,\theta) = \frac{1}{2\pi t}\frac{1}{r^2}
	\end{equation}
	which is independent of $\theta$.
\end{theorem}
\begin{proof}
	Since the density of the Brown measure of the $R$-diagonal operator $hb_t$ is rotational invariant, we assume the density of the Brown measure is $f(r)$. Then~\eqref{cumulativemu} shows that
	\[\frac{1}{2}+\frac{\log r}{t} = \mu_{hb_t}(\{|z|\leq r\}) = \int_0^r\int_0^{2\pi} f(R)\,R\,dR\,d\theta = 2\pi\int_0^rf(R)\,R\,dR.\]
	Differentiating the above equation with respect to $r$ gives us $1/(rt) = 2\pi f(r)r$. This proves
	\[f(r) = \frac{1}{2\pi t}\frac{1}{r^2}\]
	which is the claimed density~\eqref{densityhaar} of the Brown measure of $hb_t$.
	\end{proof}

\end{appendices}

\bigskip

\author{
	\noindent\begin{tabular}{ll}
		CHING-WEI HO & \quad\quad PING ZHONG\\
		Institute of Mathematics & \quad\quad	Department of Mathematics and Statistics \\
		Academia Sinica & \quad\quad	University of Wyoming \\
		Taipei 10617, Taiwan; and& \quad\quad Laramie, WY 82070, USA\\
		Department of Mathematics \\
		University of Notre Dame \\
		Notre Dame, IN 46556, USA\\
		Email: \href{mailto:cho2@nd.edu}{cho2@nd.edu} & \quad\quad 	Email: \href{mailto:pzhong@uwyo.edu}{pzhong@uwyo.edu}
	\end{tabular}
}


\end{document}